\documentclass[11pt,a4paper,reqno]{amsart}
\usepackage{amsmath}
\usepackage{amsfonts}
\usepackage{amssymb}
\usepackage{amscd}
\usepackage{url}
\usepackage{enumerate}
\usepackage[pdftex,bookmarks=true]{hyperref}

\newcommand{\indicator}[1]{\ensuremath{\mathbf{1}_{\{#1\}}}}

\renewcommand\Re{{\operatorname{Re}}}
\renewcommand\Im{{\operatorname{Im}}}

\newcommand\rank{{\operatorname{rank}}}

\newcommand\R{{\mathbf{R}}}
\newcommand\C{{\mathbf{C}}}

\renewcommand\P{{\mathbf{P}}}
\newcommand\E{{\mathbf{E}}}

\newcommand\tr{{\operatorname{tr}}}
\newcommand\Var{{\operatorname{Var}}}

\newcommand\dist{{\operatorname{dist}}}
\newcommand\Z{{\mathbf{Z}}}

\newcommand\col{{\mathbf{c}}}
\newcommand\row{{\mathbf{r}}}

\newcommand\Br{{\mathbf r}}

\newcommand\Bu{{\mathbf u}}
\newcommand\Bv{{\mathbf v}}
\newcommand\Bw{{\mathbf w}}
\newcommand\Bx{{\mathbf x}}
\newcommand\By{{\mathbf y}}
\newcommand\Bz{{\mathbf z}}
\newcommand\ep{\epsilon}

\theoremstyle{plain}
  \newtheorem{theorem}[subsection]{Theorem}

  \newtheorem{fact}[subsection]{Fact}
  \newtheorem{lemma}[subsection]{Lemma}
  \newtheorem{corollary}[subsection]{Corollary}

  \newtheorem{condition}{Condition}
  \newtheorem{remark}[subsection]{Remark}
  
  \newtheorem{claim}[subsection]{Claim}

\theoremstyle{definition}
  \newtheorem{definition}[subsection]{Definition}

\begin{document}

\title[The elliptic law]{The elliptic law}

\author{Hoi H. Nguyen}
\address{The Ohio State University, Department of Mathematics, 231 West 18th Avenue, Columbus, OH 43210}
\email{nguyen.1261@math.osu.edu}

\author{Sean O'Rourke}
\address{Department of Mathematics, University of Colorado at Boulder, Boulder, CO 80309  }
\email{sean.d.orourke@colorado.edu}

\thanks{The first author is partly supported by research grant DMS-1200898}


\begin{abstract}
We show that, under some general assumptions on the entries of a random complex $n \times n$ matrix $X_n$, the empirical spectral distribution of $\frac{1}{\sqrt{n}}X_n$ converges to the uniform law of an ellipsoid as $n$ tends to infinity. This generalizes the well-known circular law in random matrix theory.  
\end{abstract}

\maketitle

\section{Introduction}

Let $X_n$ be a $n \times n$ matrix with complex eigenvalues $\lambda_1, \lambda_2, \dots,\lambda_n$. The empirical spectral measure $\mu_{X_n}$ of $X_n$ is defined as
$$\mu_{X_n}:=\frac{1}{n} \sum_{i=1}^n \delta_{\lambda_i}$$ 
and the corresponding empirical spectral distribution (ESD) $F^{X_n}(x,y)$ is given by
$$ F^{X_n}(x,y) := \frac{1}{n} \#\{1 \leq j \leq n : \Re(\lambda_j) \leq x, \Im(\lambda_j) \leq y \}.$$
Here $\#E$ denotes the cardinality of the set $E$.  In the case when the eigenvalues of $X_n$ are real, we write the ESD $F^{X_n}$ as just a function of $x$,
$$ F^{X_n}(x) := \frac{1}{n} \#\{1 \leq j \leq n : \lambda_j \leq x \}. $$

A fundamental problem in random matrix theory is to determine the limiting distribution of the ESD as the size of the matrix tends to infinity. In certain cases when the entries have special distribution, such as Gaussian, the joint distribution of the eigenvalues can be given explicitly, and so the limiting distribution can be derived directly. However, these explicit formulas are not available for many random matrix ensembles, and so the problem of finding the limiting distribution becomes much more difficult. On the other hand, the well-known {\it universality phenomenon} in random matrix theory predicts that the limiting distribution should not depend on the distribution of the entries. We give two famous examples below.  

In the 1950s, Wigner studied the limiting ESD for a large class of random Hermitian matrices whose entries on or above the diagonal are independent \cite{W}.  In particular, Wigner showed that, under some additional moment and symmetry assumptions on the entries, the ESD of such a matrix converges to the semi-circular law $F_{\text{sc}}$ with density given by 
$$F'_{\text{sc}}(x):= \left\{
     \begin{array}{lr}
       \frac{1}{2\pi}\sqrt{4-x^2} , &-2 \leq x \leq 2 \\
       0 ,& \text{otherwise}
     \end{array}
   \right.  .$$ 

The most general form of the semi-circular law assumes only the first two moments of the entries \cite{B:M}.

\begin{theorem}[Semi-circular law for Wigner matrices]\label{theorem:sc}
Let $\zeta$ be a real random variable, and let $\xi$ be a complex random variable with variance one.  For each $n \geq 1$, assume $X_n$ is a $n \times n$ Hermitian matrix whose entries on or above the diagonal are independent.  Further assume that the diagonal entries are i.i.d. copies of $\zeta$ and those above the diagonal are i.i.d copies of $\xi$.  Then the ESD of the matrix $\frac{1}{\sqrt{n}}X_n$ converges almost surely to the semi-circular law as $n \rightarrow \infty$.
\end{theorem}

The ESD for non-Hermitian random matrices with i.i.d. entries was first studied by Mehta \cite{M}.  In particular, in the case where the entries of $X_n$ are i.i.d. complex normal random variables, Mehta showed that the ESD of $\frac{1}{\sqrt{n}}X_n$ converges, as $n \rightarrow \infty$, to the circular law $F_{\text{cir}}$ given by
$$F_{\text{cir}}(x,y):= \frac{1}{\pi} \operatorname{mes}\Big(|z|\le 1: \Re(z)\le x, \Im(z)\le y\Big).$$
In other words, $F_{\text{cir}}$ is the two-dimensional distribution function for the uniform probability measure on the unit disk in the complex plane.  

Mehta used the joint density function of the eigenvalues of $\frac{1}{\sqrt{n}}X_n$ which was derived by Ginibre \cite{Gi}.  The real Gaussian case was studied by Edelman in \cite{Ed-cir}.  For the general (non-Gaussian) case when there is no formula, the problem appears much more difficult. Important results were obtained by Girko \cite{G1,G2}, Bai \cite{B,BS}, and more recently by G\"otze and Tikhomirov \cite{GT}, Pan and Zhou \cite{PZ}, and Tao and Vu \cite{TVcir}. These results confirm the same limiting law under some moment or smoothness assumptions on the distribution of the entries. Recently, Tao and Vu (appendix by Krishnapur) were able to remove all these additional assumptions, establishing the law under the first two moments \cite{TVuniv}.

\begin{theorem}[Circular law for non-Hermitian i.i.d. matrices]\label{theorem:cir}
Let $\xi$ be a complex-valued random variable with mean zero and variance one.  For each $n \geq 1$, assume that the entries of the $n \times n$ matrix $X_n$ are i.i.d. copies of $\xi$. Then the ESD of the matrix $\frac{1}{\sqrt{n}}X_n$ converges almost surely to the circular law as $n \rightarrow \infty$.
\end{theorem}

The two celebrated results above provide a somewhat complete picture of the limiting law for the ESD of Hermitian and non-Hermitian i.i.d matrices. In the 1980s, Girko initiated a study of the limiting distribution for more general matrices which interpolate between Hermitian and non-Hermitian models. 

\begin{definition}[Condition {\bf C0}] \label{def:C0}
Let $(\xi_1, \xi_2)$ be a random vector in $\mathbb{C}^2$ where both $\xi_1$ and $\xi_2$ have mean zero and unit variance.  Let $\{x_{ij}\}$ be an infinite double array of random variables on $\mathbb{C}$.  For each $n \geq 1$ we define the random $n \times n$ matrix $X_n = (x_{ij})_{1 \leq i,j \leq n}$.  We say that the sequence of random matrices $\{X_n\}_{n \geq 1}$ satisfies condition {\bf C0} with atom variables $(\xi_1, \xi_2)$ if the following conditions hold:
\begin{enumerate}[(i)]
\item (Independence) $\{ x_{ii} : i \geq 1\} \cup \{ (x_{ij}, x_{ji}) : 1 \leq i < j \}$ is a collection of independent random elements,
\item (Common distribution) each pair $(x_{ij}, x_{ji})$, $1 \leq i < j$ is an i.i.d. copy of $(\xi_1, \xi_2)$,
\item (Flexibility of the main diagonal) the diagonal elements, $\{x_{ii} : i \geq 1\}$, are i.i.d. with mean zero and finite variance.
\end{enumerate}
\end{definition}

It is clear that many Hermitian and non-Hermitian i.i.d matrix ensembles belong to the above class. In fact, it also consists of linear combinations of independent Hermitian and non-Hermitian i.i.d. matrices.

Over the past thirty years, Girko has established a number of results for the limiting law of random matrices satisfying condition {\bf C0}. We refer the reader to \cite{G3-elliptic, G4-elliptic, G5-elliptic,G1-elliptic, G2-elliptic} and references therein. To our best understanding, Girko's proofs are incomplete and lack rigor.  The familiar reader may also relate this to Girko's controversial works on the circular law (see the discussions in \cite{B,Ed-cir}).  

When $\{X_n\}_{n \geq 1}$ is a sequence of random matrices that satisfy condition {\bf C0} with jointly Gaussian atom variables $(\xi_1,\xi_2)$, the joint eigenvalue density can be derived explicitely and the limiting ESD can be computed directly; see \cite{J,KS,Le} and references therein.  Recently, Naumov \cite{N} has been able to verify the same limiting law for a much more general class of real random matrices whose entries have finite fourth moment.
 
For $-1<\rho<1$, denote by $\mathcal{E}_\rho$ the ellipsoid
$$\mathcal{E}_\rho:=\left\{z \in \mathbb{C} : \frac{\Re(z)^2}{(1+\rho)^2}+\frac{\Im(z)^2}{(1-\rho)^2} \leq 1 \right \}.$$

\begin{theorem}[Naumov \cite{N}]\label{theorem:Gir}
Let $\{X_n\}_{n \geq 1}$ be a sequence of real random matrices that satisfy condition {\bf C0} with real atom variables $(\xi_1, \xi_2)$ where $\E[\xi_1 \xi_2] = \rho$, $-1 < \rho < 1$. Also, assume that $\max(\E|\xi_1|^{4},\E|\xi_2|^{4}) <\infty$. Then the ESD of the matrix $\frac{1}{\sqrt{n}}X_n$ converges in probability as $n \rightarrow \infty$ to the elliptic law $F_\rho$ with parameter $\rho$ given by

$$F_\rho(x,y) := \frac{1}{\pi (1-\rho^2)} \operatorname{mes}\Big(z\in \mathcal{E}_\rho : \Re(z)\le x, \Im(z)\le y\Big).$$
\end{theorem}

In conjunction with Theorems \ref{theorem:sc}, \ref{theorem:cir}, and with the universality phenomenon, it is tempting to conjecture that Theorem \ref{theorem:Gir} should hold without any further moment assumption. One of the main goals of this paper is to resolve this conjecture for the real case.  

For any matrix $M$, we define the Hilbert-Schmidt norm $\|M\|_2$ by the formula
\begin{equation} \label{eq:def:hs}
	\|M\|_2 := \sqrt{\tr (M^\ast M)} = \sqrt{\tr (M M^\ast)}. 
\end{equation}

\begin{theorem}[Elliptic law for real random matrices] \label{thm:real}
Let $\{X_n\}_{n \geq 1}$ be a sequence of real random matrices that satisfy condition {\bf C0} with real atom variables $(\xi_1, \xi_2)$ where $\E[\xi_1 \xi_2] = \rho$, $-1 < \rho < 1$.  Assume that $\{F_n\}_{n \geq 1}$ is a sequence of deterministic matrices such that $\rank(F_n)=o(n)$\footnote{We use asymptotic notation under the assumption that $n \to \infty$.  See Section \ref{section:notation} for a complete description of the asymptotic notation used here and throughout the paper.} and $\sup_{n} \frac{1}{n^2} \|F_n\|^2_2 < \infty$.  Then the ESD of ${\frac{1}{\sqrt{n}}(X_n + F_n)}$ converges almost surely to the elliptic law with parameter $\rho$ as $n \rightarrow \infty$.  
\end{theorem}

In fact, we are able to extend Theorem \ref{thm:real} to the following more general setting. 

\begin{definition}[$(\mu,\rho)$-family] Given parameters $0\le \mu \le 1$ and $-1 < \rho < 1$, we say that the complex random variable pair $(\xi_1,\xi_2)$ belongs to the {\it $(\mu,\rho)$-family} if the following holds.

\begin{enumerate}[(i)]
\item Both $\xi_1$ and $\xi_2$ have mean zero and unit variance;
\item $\E[(\Re(\xi_1))^2] = \E[(\Re(\xi_2))^2] =\mu$ and $\E[(\Im(\xi_1))^2] = \E[(\Im(\xi_2))^2] =1-\mu$;
\item $\E[\Re(\xi_1)\Re(\xi_2)] = \mu \rho$ and $\E[\Im(\xi_1)\Im(\xi_2)] = -(1-\mu)\rho$;  
\item $\E[\Re(\xi_i)\Im(\xi_j)] = 0$ for any $i,j \in \{1,2\}$. 
\end{enumerate}
\end{definition}

\begin{remark}
If $(\xi_1,\xi_2)$ belongs to the $(\mu,\rho)$-family, then the covariance matrix of $\mathbf{\xi} = (\Re(\xi_1), \Im(\xi_1), \Re(\xi_2), \Im(\xi_2))^\mathrm{T}$ is given by 
$$ \E \mathbf{\xi} \mathbf{\xi}^\mathrm{T} = \begin{pmatrix} \mu & 0 & \mu \rho & 0 \\ 0 & 1-\mu & 0 & -(1-\mu) \rho \\ \mu \rho & 0 & \mu & 0 \\ 0 & -(1-\mu)\rho & 0 & 1-\mu \end{pmatrix}. $$
\end{remark}

Notice that if $(\xi_1,\xi_2)$ belongs to the $(\mu,\rho)$-family then $\E|\xi_1|^2=\E|\xi_2|^2=1$ and $\E[ \xi_1\xi_2 ]=\rho$.  More importantly, we do not require the imaginary and real parts of $\xi_1,\xi_2$ to be independent.

\begin{theorem}[Elliptic law for complex random matrices] \label{thm:complex} Let $0\le \mu \le 1$ and $-1<\rho<1$ be given. Let $\{X_n\}_{n \ge 1}$ be a sequence of complex matrices such that $\{X_n\}_{n \geq 1}$ satisfies condition {\bf C0} with atom variables $(\xi_1,\xi_2)$ from the $(\mu,\rho)$-family. Assume furthermore that $\{F_n\}_{n \geq 1}$ is a sequence of deterministic matrices such that $\rank(F_n)=o(n)$ and $\sup_{n} \frac{1}{n^2} \|F_n\|^2_2 < \infty$. Then the ESD of ${\frac{1}{\sqrt{n}}(X_n + F_n)}$ converges almost surely to the elliptic law with parameter $\rho$ as $n \rightarrow \infty$. 
\end{theorem}

In light of the universality phenomenon, we conjecture that Theorem \ref{thm:complex} continues to hold when $\E|\xi_1|^2=\E|\xi_2|^2=1$ and $\E[ \xi_1\xi_2 ]=\rho$, where $\rho$ is a complex number satisfying $|\rho|<1$. In this optimal setting, the ESD of ${\frac{1}{\sqrt{n}}X_n}$ is conjectured to converge to the elliptic law associated with the rotated ellipsoid $\mathcal{E}_\rho$ given by
$$ \mathcal{E}_\rho := \left\{ z \in \mathbb{C} : \frac{ \left(\Re(z) \cos \frac{\theta}{2} - \Im(z) \sin \frac{\theta}{2}\right)^2}{(1 + |\rho|)^2} + \frac{\left(\Re(z) \sin \frac{\theta}{2} + \Im(z) \cos \frac{\theta}{2}\right)^2}{(1 - |\rho|)^2} \leq 1 \right\}, $$
where $\theta = \operatorname{Arg}(\rho)$.  (This formula for the rotated ellipsoid can be derived by multiplying the matrix by $e^{-i\theta/2}$ so that the resulting atom variables have a real-valued correlation.)

One of the key ingredients in the proof of Theorems \ref{thm:real} and \ref{thm:complex} is a lower bound on the least singular value of $X_n$.  If $M$ is a $n \times n$ matrix, we let 
$$ \sigma_1(M) \geq \sigma_2(M) \geq \cdots \geq \sigma_n(M) \geq 0 $$
denote the singular values of $M$.  In particular, the largest and smallest singular values satisfy
$$ \sigma_1(M) = \sup_{\|x\| = 1} \|Mx\| $$
and
$$ \sigma_n(M) = \inf_{\|x\| = 1} \|M x\|, $$
where $\|v\|$ denotes the Euclidean norm of a vector $v$.   

In particular, we will verify the following polynomial bound for the smallest singular value.  

\begin{theorem}[Bound on the least singular value for pertubed random matrices]\label{theorem:singularvalue}
Assume that $M_n=F_n+X_n$, where the entries of the given complex matrix $F_n$ are bounded by $n^{\alpha}$ in absolute value, and $X_n$ is a random matrix from Theorem \ref{thm:complex} for given $0\le \mu\le 1$ and $-1<\rho<1$. Then for any $B>0$, there exists $A>0$ and $n_0 > 0$ (both depending on $B,\alpha,\mu,\rho$, and the distribution of $(\xi_1,\xi_2)$ and $x_{11}$) such that 
$$\P(\sigma_n(M_n) \le n^{-A})\le n^{-B} $$
for all $n > n_0$.  
\end{theorem}

Our polynomial bound here is motivated by \cite[Lemma 4.1]{TVcir} of Tao and Vu, which plays a fundamental rule in their establishment of the circular law (Theorem \ref{theorem:cir}). We also refer the reader to the work \cite{RV} of Rudelson and Vershynin for an almost complete treatment for the least singular values of random non-Hermitian matrices with independent entries.  Similar techniques have also been used by G\"otze and Tikhomirov \cite{GT} to prove a version of Theorem \ref{theorem:cir}.  Recently, a similar study for random real symmetric matrices has been carried out independently by Vershynin in \cite{Ver} and by the first author in \cite{Ng-sym}.

\subsection{Overview and Outline}
Because of its importance, we prove Theorem \ref{theorem:singularvalue} first.  Indeed, in Section \ref{section:singular}, we outline the proof of Theorem \ref{theorem:singularvalue}.  We then complete the proof in Sections \ref{section:ILO:linear}--\ref{section:step2}.  In Section \ref{section:maintheorem:proof}, we use Theorem \ref{theorem:singularvalue} to prove our main results,  Theorems \ref{thm:real} and \ref{thm:complex}.  In particular, Section \ref{section:maintheorem:proof} is independent of Sections \ref{section:singular}--\ref{section:step2} and can be read separately.  The appendix contains a number of auxiliary results.

\subsection{Notation} \label{section:notation}
For a $m \times n$ matrix $M$, we let 
$$ \sigma_1(M) \geq \cdots \geq \sigma_{\min\{m,n\}}(M) \geq 0 $$
denote the singular values of $M$.  We use the notations $\row_i(M)$ and $\col_j(M)$ to denote its $i$-th row vector and its $j$-th column vector respectively; we use the notation $(M)_{ij}$ and $M_{ij}$ to denote its $(i,j)$ entry.  We let $\|M\|_2$ denote the Hilbert-Schmidt norm of $M$ (defined in \eqref{eq:def:hs}) and let $\|M\| := \sigma_1(M)$ denote the spectral norm of $M$.  

We consider $n$ an asymptotic parameter tending to infinity.  We use $Z \ll Y$, $Y \gg Z$, $Y=\Omega(Z)$, or $Z=O(Y)$ to denote the bound $|Z| \leq CY$ for all sufficient large $n$ for some constant $C$.  Notations such as $Z \ll_k Y$, $Z=O_k(Y)$ mean that the hidden constant $C$ depends on another constant $k$.  $Z=o(Y)$ or $Y=\omega(Z)$ means that $Z/Y \to 0$ as $n \to 0$.  We write $Z = \Theta(Y)$ or $Z \asymp Y$ for $Y \ll Z \ll Y$.  

As customary, we use $\eta$ to denote a Bernoulli random variable (thus $\eta$ takes values $\pm 1$ with
probability 1/2). For a given $0\le \mu \le 1$, we use $\eta^{(\mu)}$ to denote a modified-Bernoulli random variable
of parameter $\mu$ (thus $\eta^{(\mu)}$ takes values $\pm 1$ with probability $\mu/2$ and zero with probability
$1-\mu$).

Let A be an event.  Sometimes we will write $\P_{y_1,\ldots, y_k}(A)$ to emphasize that the probability under consideration is taken with respect to the specified random variables $y_1,\ldots,y_k$ (while fixing all other random variables).

We write a.s., a.a., and a.e. for almost surely, Lebesgue almost all, and Lebesgue almost everywhere respectively.

We use $\sqrt{-1}$ to denote the imaginary unit and reserve $i$ as an index.

\section{The least singular value problem}\label{section:singular}


In this section, we begin the proof of Theorem \ref{theorem:singularvalue}.  Broadly speaking, our proof follows the approach of \cite{Ng-sym}.  Nevertheless, because the matrix $X_n$ under consideration is much more complicated than a Hermitian matrix, it is of great necessity to generalize and string a series of previous results \cite{NgV, Ng-QLO, Ng-sym} together. As a result, our ideas will not be fully original but a highly non-trivial generalization of existing ones. The rest of this section is devoted to sketching our approach; complete details of the proofs will be presented subsequently. 

First of all, we will assume $n$ to be sufficiently large. For the sake of simplicity, we will prove our result under the following condition. 

\begin{condition}\label{condition:bound}
With probability one, $|x_{ij}| \le n^{B+1}$ for all $i,j$.
\end{condition}

In fact, because all $x_{ij}$ have bounded variance, we have $\P(|x_{ij}|\ge n^{B+1})= O(n^{-2B-2})$. Thus, we can assume that $|x_{ij}|\le n^{B+1}$ at the cost of an additional negligible term $o(n^{-B})$ in probability.

We next assume that $\sigma_n(M_n)\le n^{-A}$. Thus $M_n\Bx=\By$ for some $\|\Bx\|_2=1$ and $\|\By\|_2 \le n^{-A}$. There are two cases to consider.

\subsection{Case 1.} {\it $M_n$ has full rank}. This is the main case to consider as most of random matrices are non-singular with very high probability. 

Let $C(M_n)=(c_{ij}(M_n))$, $1\le i,j\le n$, be the matrix of the cofactors of $M_n$. By definition, $C(M_n)\By = \det(M_n) \cdot \Bx$, and thus we have $\|C(M_n)\By\|_2 = |\det(M_n)|$. 

By paying a factor of $n$ in probability, without loss of generality we can assume that the first component of $C(M_n)\By$ is greater than $\det(M_n)/n^{1/2}$, 
\begin{equation}\label{eqn:firstrow}
|c_{11}(M_n)y_1+\dots c_{1n}(M_n)y_n|\ge |\det(M_n)|/n^{1/2}.
\end{equation}

Note that $\|\By\|_2\le n^{-A}$, it thus follows
\begin{equation}\label{eqn:intro:1}
\sum_{j=1}^n |c_{1j}(M_n)|^2 \ge n^{2A-1} |\det(M_n)|^2.
\end{equation}

For $j\ge 2$, we write 

$$c_{1j}(M_n)=\sum_{i=2}^n m_{i1}c_{ij}(M_{n-1}),$$
where $M_{n-1}$ is the matrix obtained from $M_n$ by removing its first row and first column, and $c_{ij}(M_{n-1})$ are the corresponding cofactors of $M_{n-1}$, and $m_{ij}$ are the entries of $M_n$.

Hence, by the Cauchy-Schwarz inequality, by Condition \ref{condition:bound}, and by the bounds $f_{ij}\le n^\alpha$ for the entries of $F_n$, we have
\begin{eqnarray}\label{eqn:intro:2}
|c_{1j}(M_n)|^2 &\le \sum_{i=2}^n |m_{i1}|^2 \sum_{i=2}^n |c_{ij}(M_{n-1})|^2 \nonumber\\
&\le n^{2B+2\alpha + 3}  \sum_{i=2}^n |c_{ij}(M_{n-1})|^2.
\end{eqnarray}
Similarly, for $j=1$ we write $c_{11}(M_n)=\sum_{i=2}^n m_{i2}c_{i2}(M_{n-1})$, and thus, 
\begin{equation}\label{eqn:intro:3}
|c_{11}(M_n)|^2\le n^{2B+2\alpha+3} \sum_{i=2}^n |c_{i2}(M_{n-1})|^2.
\end{equation}

It follows from \eqref{eqn:intro:1}, \eqref{eqn:intro:2}, and \eqref{eqn:intro:3} that
$$2 \sum_{2\le i,j \le n} |c_{ij}(M_{n-1})|^2 \ge n^{2A-2B-2\alpha-4}|\det(M_{n})|^2.$$
Hence, for proving Theorem \ref{theorem:singularvalue}, it suffices to justify the following result (after an appropriate modification for $A$).

\begin{theorem}\label{theorem:singularvalue''}
For any $B>0$, there exists $A>0$ such that 
$$\P\big((\sum_{2\le i,j \le n} |c_{ij}(M_{n-1})|^2)^{1/2} \ge n^A|\det(M_n)| \big)\le n^{-B}.$$
\end{theorem}

To see why the assumption $(\sum_{2\le i,j \le n} |c_{ij}(M_{n-1})|^2)^{1/2} \ge n^A|\det(M_n)|$ is useful, we next express $\det(M_n)$ as a bilinear form of its first row and column, 
$$\det(M_n) = c_{11}(M_n)m_{11} + \sum_{2\le i,j \le n} c_{ij}(M_{n-1})m_{1i}m_{j1}.$$
In other words, with $c:=(\sum_{2\le i,j \le n} |c_{ij}(M_{n-1})|^2)^{1/2}$ (which is nonzero as $M_n$ has full rank) and with $a_{ij}:=c_{ij}(M_{n-1})/c$ we have
\begin{equation}\label{eqn:bilinear}
\frac{1}{c}\det(M_n) =\frac{1}{c}m_{11}c_{11}(M_n)+\sum_{2\le i,j\le n}a_{ij}m_{1i}m_{j1}.
\end{equation}
 
Intuitively, if we condition on $M_{n-1}$ and $m_{11}$, then the right hand side of \eqref{eqn:bilinear}, as a bilinear form of the random variables $x_{1i},x_{i1},2\le i$, is comparable to 1 in absolute value with probability extremely close to one. Thus the assumption $\P(|\det(M_n)|/c\le n^{-A})\ge n^{-B}$ of Theorem \ref{theorem:singularvalue''}, with  appropriately large $A$, must yield a high cancelation of the bilinear form.  

Basing on this intuition, our rough approach will consist of two main steps below.
  
\begin{itemize}
\item  {\it Step 1} (Inverse step). Assume that for appropriately large $A$ we have 
$$\P_{x_{11},\dots,x_{1n},x_{21},\dots, x_{n1}}\left(\left|(c_{11}(M_n)/c)m_{11}+\sum_{2\le i,j\le n}a_{ij}m_{1i}m_{j1}\right|\le n^{-A}|M_{n-1}\right)\ge n^{-B}.$$ 
Then there must be a strong structure among the cofactors $c_{ij}$ of $M_{n-1}$.
\item {\it Step 2} (Counting step). The probability, with respect to $M_{n-1}$, that there is a strong structure among the $c_{ij}$ is negligible.
\end{itemize}

Before stating the steps above in greater detail, we pause to introduce the structure appearing in our analysis.

A set $Q\subset \C$ is a  \emph{generalized arithmetic progression} (GAP) of rank $r$ if it can be expressed as in the form
$$Q= \{g_0+ k_1g_1 + \dots +k_r g_r| k_i\in \Z, K_i \le k_i \le K_i' \hbox{ for all } 1 \leq i \leq r\}$$ for some $\{g_0,\ldots,g_r\},\{K_1,\ldots,K_r\}$ and $\{K'_1,\ldots,K'_r\}$.

It is convenient to think of $Q$ as the image of an integer box $B:= \{(k_1, \dots, k_r) \in \Z^r| K_i \le k_i \le K_i' \} $ under the linear map
$$\Phi: (k_1,\dots, k_r) \mapsto g_0+ k_1g_1 + \dots + k_r g_r. $$
The numbers $g_i$ are the \emph{generators } of $Q$, the numbers $K_i'$ and $K_i$ are the \emph{dimensions} of $Q$. We say that $Q$ is \emph{proper} if this map is one to one, or equivalently if $|Q| =|B|$.  For non-proper GAPs, we of course have $|Q| < |B|$. If $-K_i=K_i'$ for all $i\ge 1$ and $g_0=0$, we say that $Q$ is {\it symmetric}.

We refer the reader to Sections \ref{section:ILO:linear} and \ref{section:ILO:quadratic} for further explanation as to why GAPs are the right object to study here. In the sequel we state our main steps rigorously with the help of GAPs. 

\begin{theorem}[Step 1]\label{theorem:step1}
Let $0<\ep<1$ be a given constant. Assume that $M_{n-1}$ is fixed and
$$\sup_a \P_{x_2,\dots,x_{n},x_2',\dots,x_n'}\left(\left|\sum_{2\le i,j\le n} a_{ij}(x_i+f_i)(x_j'+f_j')-a\right|\le n^{-A}\right)\ge n^{-B}$$
for some sufficiently large integer $A$, where 
\begin{itemize}
\item $a_{ij}=c_{ij}(M_{n-1})/c$, 
\item $f_i=f_{1i},f_i'=f_{i1}$ are the entries of $F_n$, and thus fixed,
\item $(x_i,x_i')$ are i.i.d copies of  $(\xi_1,\xi_2)$ of a given $(\mu,\rho)$-family with $0\le \mu \le 1$ and $-1<\rho <1$.
\end{itemize}
Then there exists a complex vector $\Bu=(u_1,\dots,u_{n-1})$ which satisfies the following properties.
\begin{itemize}
\item (orthogonality) $\|\Bu\|_2\asymp 1$ and  either $|\langle \Bu,\row_i(M_{n-1})\rangle| \le n^{-A/2+O_{B,\ep}(1)}$ for $n-O_{B,\ep}(1)$ rows of $M_{n-1}$ or $|\langle \Bu,\col_i(M_{n-1})\rangle| \le n^{-A/2+O_{B,\ep}(1)}$ for $n-O_{B,\ep}(1)$ columns of $M_{n-1}$;
\item (additive structure) there exists a generalized arithmetic progression $Q$ of rank $O_{B,\ep}(1)$ and size $n^{O_{B,\ep}(1)}$ that contains at least $n-2n^\ep$ components $u_i$;
\item (controlled form) all the components $u_i$, and all the generators of the generalized arithmetic progression are rational complex numbers of the form $\frac{p}{q}+ \sqrt{-1} \frac{p'}{q'} $, where $|p|,|q|,|p'|,|q'| \le n^{A/2+O_{B,\ep}(1)}$.
\end{itemize}
\end{theorem}

In the second step of the approach, we  show that the probability for $M_{n-1}$ having the above properties is negligible.
\begin{theorem}[Step 2]\label{theorem:step2}
With respect to $M_{n-1}$, the probability that there exists a vector $\Bu$ as in Theorem \ref{theorem:step1} is $\exp(-\Omega(n))$.
\end{theorem}

\subsection{Case 2.} {\it $M_n$ does not have full rank}, which is the case to consider if $\xi_1,\xi_2$ have discrete distribution. We show that for any fixed $B>0$ this event holds with probability less than $n^{-B}$ for large enough $n$ depending on $B$.

First, instead of the entries $x_{ij}$ of $X_n$, consider $x_{ij}':= (1-\ep^2)x_{ij}+ \ep \xi_{ij}$, where $\xi_{ij}$ are independently uniform on the interval $[-1,1]$ and $\ep$ is very small, say $n^{-1000An}$. It is clear that the continuous matrix $M_n'=X_n'+F_n$, where $X_n'$ is formed by the $x_{ij}'$ above, has full rank with probability one. By applying Theorem \ref{theorem:singularvalue} obtained from Case 1 for the matrix $M_n'$, with probability at least $1-n^{-B}$ one has $\sigma_n(M_n')\ge n^{-A}$, and thus
\begin{equation}\label{eqn:largedeterminant}
|\det(M_n')|\ge n^{-An}.
\end{equation}

Next, because $M_n'= M_n -\ep(\ep x_{ij}+\xi_{ij})$ and as $|x_{ij}|\le n^{B+1}$, by the Brunn-Minkowski inequality and Hadamard's bound we have 
$$|\det(M_n')|\le (|\det(M_n)|^{1/n}+O(n^{-500A}))^n,$$
where we use the fact that $A$ is chosen sufficiently large compared to $B$.

Combining with \eqref{eqn:largedeterminant}, we then infer that $|\det(M_n)|\ge n^{-(1+o(1))An}$, and thus $\det(M_n)\neq 0$ with probability at least $1-n^{-B}$, concluding the treatment for this case.

The proof of Theorem \ref{theorem:step1} will be given in Section \ref{section:step1} thanks to useful tools from Sections \ref{section:ILO:linear} and \ref{section:ILO:quadratic}. Theorem \ref{theorem:step2} will be concluded in Section \ref{section:step2}.

\section{Anti-concentration, a warm-up}\label{section:ILO:linear}

Recall that in the inverse step, Theorem \ref{theorem:step1}, we assumed that 
\begin{equation}\label{eqn:step1:discussion}
\sup_a \P_{x_2,\dots,x_{n},x_2',\dots,x_n'}\left(\left|\sum_{2\le i,j\le n} a_{ij}(x_i+f_i)(x_j'+f_j')-a\right|\le n^{-A}\right)\ge n^{-B}.
\end{equation}
This can be considered as a high concentration of the bilinear form $\sum_{2\le i,j\le n} a_{ij}(x_i+f_i)(x_j'+f_j')$ on a small ball of radius $n^{-A}$, where $x_i$ and $x_i'$ are not necessarily jointly independent. The main idea to extract this bit of information is to relate it to a high concentration of an appropriate linear form. This step is postponed until Section \ref{section:ILO:quadratic}. Our goal now is to focus on linear forms.

A classical result of Erd\H{o}s \cite{E} and Littlewood-Offord \cite{LO} in the 1940s asserts that if $a_i$ are complex numbers of magnitude $|a_i|\ge 1$, then the probability that the linear form $\sum_{i=1}^n a_ix_i$ concentrates on a disk of radius one is of order $O(n^{-1/2})$, where $x_i$ are i.i.d. copies of a Bernoulli random variable. Recently, motivated by inverse theorems from additive combinatorics, Tao and Vu studied the underlying reason as to why the concentration probability of $\sum_{i=1}^n a_ix_i$ on a small ball is large. They call this the {\it inverse Littlewood-Offord problem}. A closer look at the definition of generalized arithmetic progressions defined in Section \ref{section:singular} reveals that if $a_i$ are very {\it close} to the elements of a $GAP$ of rank $O(1)$ and size $n^{O(1)}$, then the probability that $\sum_{i=1}^n a_ix_i$ concentrates on some small ball is of order $n^{-O(1)}$, where $x_i$ are i.i.d. copies of a Bernoulli random variable. 

It was shown implicitly by Tao and Vu in \cite{TVbull,TVcir,TVcomp} that these are essentially the only examples that have high concentration probability. An explicit and somewhat optimal version has been proved in a recent paper by the first author and Vu in \cite{NgV}. Before stating this result, we pause to introduce some terminology. 

We say that a real random variable $\xi$ is {\it anti-concentrated} if there exist positive constants $\alpha_1,\alpha_2,\alpha_3$ such that $\P(\alpha_1< |\xi-\xi'|<\alpha_2)\ge \alpha_3,$ where $\xi'$ is an i.i.d. copy of $\xi$. (Note that the requirement of anti-concentration is somewhat weaker than having mean zero and unit variance.) We say that a complex number  $a\in \C$ is {\it $\delta$-close} to a set $Q\subset \C$ if there exists $q\in Q$ such that $|a-q|\le \delta$.

\begin{theorem}[Inverse Littlewood-Offord theorem for linear forms, \cite{NgV}]\label{theorem:ILOlinear} Let $0 <\ep < 1$ and $B>0$. Let 
$ \beta >0$ be an arbitrary real number that may depend on $n$. Suppose that $\sum_{i=1}^n |a_i|^2=1$, and 
$$\sup_a\P_{\Bx} \Big(|\sum_{i=1}^n a_i(x_i+f_i)-a| \le \beta\Big)=\gamma \ge n^{-B},$$ 
where $\Bx=(x_1,\dots,x_n)$, and $x_i$ are i.i.d. copies of a real random variable $\xi$ satisfying the anti-concentration condition. Then, for any number $n'$ between $n^\ep$ and $n$, there exists a proper symmetric GAP $Q=\{\sum_{i=1}^r k_ig_i : k_i\in \Z, |k_i|\le L_i \}$ such that
\begin{enumerate}[(i)]
\item (control of rank and size) $Q$ has small rank, $r=O_{B,\ep}(1)$, and small cardinality
$$|Q| \le \max \left(O_{B,\ep}\left(\frac{\gamma^{-1}}{\sqrt{n'}}\right),1\right);$$
\item (control of the steps) there is a non-zero integer $p=O_{B,\ep}(\sqrt{n'})$ such that all
 steps $g_i$ of $Q$ have the form  $g_i=\beta\frac{p_{i}} {p} $, with $p_{i} \in \Z$ and $p_{i}=O_{B,\ep}(\beta^{-1} \sqrt{n'})$;
\item (good approximation) at least $n-n'$ elements of $a_i$ are $\beta$-close to $Q$.
\end{enumerate}
\end{theorem}

Here and later, if not specified, the implied constants are allowed to depend on the distribution of the random variables under consideration. Thus, for instance the implied constants in Theorem \ref{theorem:ILOlinear} also depend on $\alpha_1,\alpha_2$ and $\alpha_3$. The interested reader is also invited to read \cite{RV} for a similar but milder setting of the inverse Littlewood-Offord for linear forms.

To attack Theorem \ref{theorem:step1}, the first step is to study the concentration of a more general linear form $\sum_i a_ix_i+b_ix_i'$, where $(x_i,x_i')$ are i.i.d copies of a pair random complex variables $(\xi_1,\xi_2)$ from a given $(\mu,\rho)$-family. Intuitively, as $\E|\xi_1|^2=\E |\xi_2|^2=1$ and $|\E \xi_1\xi_2|=|\rho|<1$, the random variables $\xi_1$ and $\xi_2$ are not totally dependent on each other. (See for instance Claim \ref{claim:difference} of Appendix \ref{section:ILOlinear:proof} for a more precise statement.) This fact may suggest a way to apply Theorem \ref{theorem:ILOlinear} with respect to $x_2,\dots, x_n$ while holding $x_2',\dots,x_n'$ ``fixed'', and vice versa. One of the main results is to justify this intuition.  

\begin{theorem}[Inverse Littlewood-Offord theorem for mixing linear forms]\label{theorem:ILOlinear:new} Let $0\le \mu \le 1, -1<\rho <1$ and $0<\ep <1, B>0$ be given. Let 
$ \beta >0$ be an arbitrary real number that may depend on $n$. Suppose that $a_i,b_i\in \C$ such that $\sum_{i=1}^n |a_i|^2+\sum_{i=1}^n|b_i|^2=1$ and 
$$\sup_a\P_{\Bx,\Bx'} \left(\left|\sum_{i=1}^n (a_ix_i+b_ix_i')-a\right| \le \beta\right)=\gamma \ge n^{-B},$$ 
where $(x_i,x_i')$ are i.i.d copies of  $(\xi_1,\xi_2)$ from a given $(\mu,\rho)$-family.
Then there exist positive constants $\alpha,c_0,C_0$ depending on $(\xi_1,\xi_2)$ and two pairs of complex numbers $(c_1,c_2)$ and $(c_1',c_2')$ (which may depend on $n$) such that 
\begin{itemize}
\item $|c_1|,|c_2|,|c_1'|,|c_2'|$ are bounded from below and above by $c_0$ and $C_0$ respectively,
\item  $|c_1/c_2-c_1'/c_2'|>\alpha$,
\item  for any number $n'$ between $n^\ep$ and $n$, there exists a proper symmetric GAP $Q=\{\sum_{i=1}^r k_ig_i : k_i\in \Z, |k_i|\le L_i \} \subset \C$ whose parameters satisfy (i) and (ii) of Theorem \ref{theorem:ILOlinear} and for at least $n-n'$ indices $i$, the pairs $c_1a_i+c_2b_i, c_1'a_i + c_2'b_i$ are $\beta$-close to $Q$.
\end{itemize}
\end{theorem}

As Theorem \ref{theorem:ILOlinear:new} can be shown by modifying the proof of Theorem \ref{theorem:ILOlinear} from \cite{NgV}, we postpone its proof until Appendix \ref{section:ILOlinear:proof}. We now introduce several useful corollaries. 

Firstly, by choosing $b_i=0$, Theorem \ref{theorem:ILOlinear:new} immediately implies the following version of Theorem \ref{theorem:ILOlinear}.

\begin{corollary}\label{cor:complex}
The conclusion of Theorem \ref{theorem:ILOlinear} also holds if $x_i$ are i.i.d. copies of a complex random variable $\xi$ satisfying $\E|\xi|^2=1$ and $\E\xi=\E [\Im(\xi) \Re (\xi)]=0$.
\end{corollary}

Secondly, if we $\beta$-approximate the components of $c_i,c_i'$ by rational numbers of the form $p/q, |p|,|q|=O(\beta)$, then we obtain the following.

\begin{corollary}\label{cor:ILOlinear}
Assume as in Theorem \ref{theorem:ILOlinear:new}. Then there exist two pairs of complex numbers $(c_1,c_2)$ and $(c_1',c_2')$ for which $|c_i|,|c_i'|$ are bounded from below and above by $c_0$ and $C_0$, $|c_1/c_2 -c_1'/c_2'|>\alpha$, and the components of $c_i,c_i'$ are rational numbers of the form $p/q,|p|,|q|=O(\beta)$ such that for any number $n'$ between $n^\ep$ and $n$, there exists a proper symmetric GAP $Q=\{\sum_{i=1}^r k_ig_i : k_i\in \Z, |k_i|\le L_i \} \subset \C$ whose parameters satisfy (i) and (ii) of Theorem \ref{theorem:ILOlinear} and for at least $n-n'$ indices $i$ the following holds:
\begin{itemize}
\item $a_i$ are $O(\beta)$-close to the GAP $P_1:=\frac{c_2}{c_1c_2'-c_1'c_2} \cdot Q + \frac{c_2'}{c_1c_2'-c_1'c_2}\cdot Q$;
\item $b_i$ are $O(\beta)$-close to the GAP $P_2:=\frac{c_1}{c_1c_2'-c_1'c_2} \cdot Q + \frac{c_1'}{c_1c_2'-c_1'c_2}\cdot Q$;
\item consequently, $a_i$ and $b_i$ are $O(\beta)$-close to the combined GAP $P=\frac{c_2}{c_1c_2'-c_1'c_2} \cdot Q + \frac{c_2'}{c_1c_2'-c_1'c_2}\cdot Q + \frac{c_1}{c_1c_2'-c_1'c_2} \cdot Q + \frac{c_1'}{c_1c_2'-c_1'c_2}\cdot Q$. 
\end{itemize}
\end{corollary}

Notice that the rank of $P$ is $O_{B,\ep}(1)$ and the size of $P$ is $n^{O_{B}(1)}$. Roughly speaking, the fact that the parameters involved in $P$ are rational numbers will enable us to control the number of such GAPs easily. We will exploit this pleasant fact in Sections \ref{section:ILO:quadratic} and \ref{section:step2}.

We remark that the assumption $-1<\rho <1$ is necessary because Theorem \ref{theorem:ILOlinear:new} is not valid for the boundary case $|\rho|=1$. For instance, if $\xi$ is a symmetric random variable and if $x_i'=-x_i$ (in which case $\rho=-1$), then the assumption $\P_{\Bx,\Bx'} (|\sum_{i=1}^n a_ix_i+b_ix_i'-a| \le \beta) \ge n^{-B}$  is equivalent to $\P_{\Bx} (|\sum_{i=1}^n (a_i-b_i)x_i-a| \le \beta) \ge n^{-B}$. From here, only information for the $a_i-b_i$ can be deduced but not for the individual $a_i$ and $b_i$ separately. 

Finally, the conclusion of Theorem \ref{theorem:ILOlinear:new} is somewhat optimal. Indeed, assume that there exist $(c_1,c_1'), (c_2,c_2')$  with $c_1^2+c_1'^2=c_2^2 +c_2'^2 =1$ and $c_1c_2'\neq c_1'c_2$ such that $\xi_1=c_1\psi_1 + c_1'\psi_2 $ and  $\xi_2=c_2\psi_1 + c_2'\psi_2 $, where $\psi_2$ is an independent copy of $\psi_1$. Then the assumption $\P_{\Bx,\Bx'} (|\sum_{i=1}^n (a_ix_i+b_ix_i')-a| \le \beta) \ge n^{-B}$ becomes $\P_{\psi_{ij}} (|\sum_{i=1}^n (c_1a_i+c_2b_i)\psi_{1i}+(c_1'a_i+c_2'b_i)\psi_{2i})-a| \le \beta )\ge n^{-B}$. So, as $\psi_{ij}$ are independent, structural information for $c_1 a_i+c_2b_i$ and $c_1'a_i+c_2'b_i$ can be deduced using Theorem \ref{theorem:ILOlinear}, in the same way as we concluded using Theorem \ref{theorem:ILOlinear:new}.

\section{Anti-concentration of bilinear forms}\label{section:ILO:quadratic}

We will next apply Corollary \ref{cor:ILOlinear} to infer an inverse version for the concentration of the bilinear form $\sum_{1\le i,j \le n} a_{ij}(x_i+f_i)(x_j'+f_j')$ appearing in Theorem \ref{theorem:step1}.

\begin{theorem}\label{theorem:ILO:quadratic} Let $0 <\ep < 1, |\rho|<1$ and $B>0$ be given. Let $ \beta >0$ be an arbitrary real number that may depend on $n$. Assume that $\sum_{i,j} |a_{ij}|^2 =1$ and 
$$\sup_a \P_{\Bx,\Bx'}\left(\left|\sum_{1\le i,j\le n}a_{ij}(x_i+f_i)(x_j'+f_j')-a\right|\le \beta\right)=\gamma \ge n^{-B},$$ 
where $(x_i,x_i')$ are i.i.d copies of  $(\xi_1,\xi_2)$ from a given $(\mu,\rho)$-family.


Then, there exist an integer $k\neq 0, |k|=n^{O_{B,\ep}(1)}$, a set of $r = O(1)$ rows $\row_{i_1},\dots,\row_{i_r}$ of the array $A_n=(a_{ij})_{1\le i,j\le n}$, and set $I$ of size at least $n- 2n^{\ep}$ such that for each $i\in I$, there exist integers $k_{ii_1},\dots, k_{ii_r}$, all bounded by $n^{O_{B,\ep}(1)}$, such that the following holds.
\begin{equation}\label{eqn:ILOquadratic}
\P_\Bz\left(\left|\left\langle \Bz, k\row_i(A_n)+\sum_{j=1}^r k_{ii_j}\row_{i_j}(A_n)\right\rangle\right|\le \beta n^{O_{B,\ep}(1)}\right)\ge n^{-O_{B,\ep}(1)},
\end{equation}
where $\Bz=(z_1,\dots, z_n)$ and $z_i$ are i.i.d. copies of $\eta^{(1/2)} (\xi_2-\xi_2')$, where $\xi_2'$ is an i.i.d. copy of $\xi_2$ and $\eta^{(1/2)}$ is a modified-Bernoulli random variable of parameter $1/2$ independent of $\xi_2$ and $\xi_2'$.
\end{theorem}

We remark that this result is an analogue of \cite[Theorem 1.8]{Ng-QLO} in which case we studied the concentration of the quadratic forms of type  $\sum_{ij}a_{ij} x_ix_j$. It seems plausible that  after an appropriate linear transform we can trap most of the entries $a_{ij}$ of $A_n$ into a GAP of small size and small rank (in the spirit of \cite{Ng}). However, we do not consider this matter here. Roughly speaking, in order to justify Theorem \ref{theorem:singularvalue}, we just need the conclusion of Theorem \ref{theorem:ILO:quadratic} for only one row.  

To prove \ref{theorem:ILO:quadratic}, we will follow the machinery from \cite{Ng-QLO} with some extra twists. As the first step, we free the dependencies between $\Bx$ and $\Bx'$. 

\subsection{Decoupling lemma} Let $U$ be an arbitrary subset of $\{1,\dots,n\}$ such that  both of $U$ and $\bar{U}$ are of size $\Theta(n)$.  Let $A_U$ be a matrix of size $n\times n$ defined as
\[
A_U(ij)= 
\begin{cases}
a_{ij} & \text{ if $i\in U$ and $j\in \bar{U}$ or $i\in \bar{U}$ and $j\in U$},\\
0 & \text{otherwise,}
\end{cases}
\]
where we denoted by $A_U(ij)$ the $ij$ entry of $A_U$. We prove the following lemma by a series applications of the Cauchy-Schwarz inequality.

\begin{lemma}\label{lemma:decoupling} Assume that 
$$\gamma=\sup_{a,b_i,b_i'} \P_{\Bx,\Bx'} \left(\left|\sum_{i,j}a_{ij}x_ix_j'+\sum_i b_ix_i+\sum_ib_i'x_i' -a\right|\le \beta\right)\ge n^{-B},$$
where $\Bx,\Bx'$ are defined as in Theorem \ref{theorem:ILO:quadratic}. Then,
\begin{equation}\label{eqn:ILObilinear:0}
\P_{\Bv,\Bw}\left(\left|\sum_{1\le i,j\le n} A_U(ij)v_iw_j \right|=O_B(\beta \sqrt{\log n})\right) =\Theta(\gamma^4),
\end{equation}
where $\Bv=(v_1,\dots,v_n)$, $\Bw=(w_1,\dots,w_n)$, and  $(v_i,w_i)$ are i.i.d copies of a vector $(\xi_1-\xi_1',\xi_2-\xi_2')$, where $(\xi_1',\xi_2')$ is an independent copy of $(\xi_1,\xi_2)$.
\end{lemma}

An advantage of considering the sum $\sum_{1\le i,j\le n} A_U(ij)v_iw_j$ over the original form $\sum_{ij}a_{ij}x_ix_j'$ is that we can rewrite the former as $\sum_{i\in U} (\sum_{j\in \bar{U}}a_{ij}w_j)v_i + \sum_{i\in U} (\sum_{j\in \bar{U}}a_{ji}v_j)w_i$. Thus, if all $v_j,w_j, j\in \bar{U}$ are held fixed, Theorem \ref{theorem:ILOlinear:new} applied to \eqref{eqn:ILObilinear:0} allows us to extract useful information on $\sum_{j\in \bar{U}}a_{ij}w_j$ and $\sum_{j\in \bar{U}}a_{ji}v_j$. As the proof of Lemma \ref{lemma:decoupling} is standard, we postpone it until Appendix \ref{section:decoupling:proof}. 

We next apply Theorem \ref{theorem:ILOlinear:new} to obtain the following key structure for the entries of $A_U$.

\begin{lemma}\label{lemma:A_U} There exist a set $I_0(U)$ of size $O_{B,\ep}(1)$ and a set $I(U)$ of size at least $n-n^\ep$, and a nonzero integer $k(U)$ bounded by $n^{O_{B,\ep}(1)}$ such that for any $i\in I$, there are integers $k_{ii_0}(U), i_0\in I_0(U)$, all bounded by $n^{O_{B,\ep}(1)}$, such that 
\begin{equation}\label{eqn:A_U}
\P_\By \left(\left|\left\langle k(U)\Br_i(A_U)+ \sum_{i_0\in I_0} k_{ii_0}(U) \Br_{i_0}(A_U), \By \right\rangle \right| \le \beta n^{O_{B,\ep}(1)}\right) = n^{-O_{B,\ep}(1)},
\end{equation}
where $\By=(y_1,\dots,y_n)$ and $y_i$ are i.i.d copies of $\xi_2-\xi_2'$.
\end{lemma}

As the deduction of Theorem \ref{theorem:ILO:quadratic} from Lemma \ref{lemma:A_U} is quite straightforward (by gathering the structural information from \eqref{eqn:A_U} for each $U$ carefully), we refer the reader to  \cite[Section 4]{Ng-QLO} for a complete treatment.

For the rest of this section we prove Lemma \ref{lemma:A_U} using Lemma \ref{lemma:decoupling}. First of all, as $A_U=A_{\bar{U}}$, it is enough to verify \eqref{eqn:A_U} for any index $i$ from $U$. Also, it suffices to assume $\xi$ to have discrete distribution. The continuous case can be recovered by approximating the continuous distribution by a discrete one while holding $n$ fixed.  

We begin by applying Corollary \ref{cor:ILOlinear}.

\begin{lemma}\label{lemma:roworthogonal} 
Assume as in the conclusion of Lemma \ref{lemma:decoupling} where (without loss of generality) $\beta \sqrt{\log n}$ and $\Theta(\gamma^4)$ are replaced by $\beta$ and $\gamma$ respectively. Then, the following holds with probability at least $3\gamma/4$ with respect to $\Bv_{\bar{U}}$ and $\Bw_{\bar{U}}$. There exist a proper symmetric GAP $P_{\Bw_{\bar{U}}}\subset \C$ of rank $O_{B,\ep}(1)$ and size $n^{O_{B,\ep}(1)}$, and an index set $I_{\Bw_{\bar{U}}}\subset U$ of size $|U|-n^\ep$ such that $ \langle \Br_i(A_U),  \Bw_{\bar{U}} \rangle $ is $\beta$-close to $P_{\Bw_{\bar{U}}}$ for all $i\in I_{\Bw_{\bar{U}}}$.
\end{lemma}

\begin{proof}(of Lemma \ref{lemma:roworthogonal}) Write
\begin{align*}
\sum_{i\in U,j\in \bar{U}}a_{ij}v_iw_j+\sum_{i\in \bar{U},j\in U}a_{ij}v_iw_j&=\sum_{i\in U} (\sum_{j\in \bar{U}}a_{ij}w_j)v_i+  (\sum_{j\in \bar{U}}a_{ji}v_j)w_i\\
&= \sum_{i\in U} \langle \Br_i(A_U),\Bw_{\bar{U}} \rangle v_i + \sum_{i\in U} \langle \Br_i({A_U}^T), \Bv_{\bar{U}} \rangle w_i.  
\end{align*}

We say that a pair vector $(\Bv_{\bar{U}},\Bw_{\bar{U}})$ is {\it good} if
$$\P_{\Bv_U,\Bw_U}\left(\left| \sum_{i\in U} \langle \Br_i(A_U),\Bw_{\bar{U}} \rangle v_i + \sum_{i\in U} \langle \Br_i({A_U}^T), \Bv_{\bar{U}} \rangle w_i -a\right|\le \beta\right)\ge \gamma/4.$$
We call $(\Bv_{\bar{U}},\Bw_{\bar{U}})$ {\it bad} otherwise. 

Let $G$ denote the collection of good pairs. We are going to estimate the probability $p$ of a randomly chosen pair $(\Bv_{\bar{U}},\Bw_{\bar{U}})$ being bad by an averaging method.
\begin{align*}
\P_{\Bv_{\bar{U}},\Bw_{\bar{U}},\Bv_U,\Bw_U} \left(\left| \sum_{i\in U} \langle \Br_i(A_U),\Bw_{\bar{U}} \rangle v_i + \sum_{i\in U} \langle \Br_i({A_U}^T), \Bv_{\bar{U}} \rangle w_i-a\right|\le \beta\right) &=\gamma\\
p \gamma/4 + 1-p &\ge \gamma\\
(1-\gamma)/(1-\gamma/4) &\ge p.
\end{align*}
Thus, the probability of a randomly chosen $(\Bv_{\bar{U}},\Bw_{\bar{U}})$ belonging to $G$ is at least 
$$1-p \ge (3\gamma/4)/(1-\gamma/4) \ge 3\gamma/4.$$
Consider a good vector $(\Bv_{\bar{U}}, \Bw_{\bar{U}})\in G$. By definition, we have
$$\P_{\Bv_U,\Bw_U}\left(\left| \sum_{i\in U} \langle \Br_i(A_U),\Bw_{\bar{U}} \rangle v_i + \sum_{i\in U} \langle \Br_i({A_U}^T), \Bv_{\bar{U}} \rangle w_i -a\right|\le \beta\right)\ge \gamma/4.$$

Next, if $\langle \Br_i(A_U) , \Bw_{\bar{U}}\rangle =\mathbf{0}$ for all $i$, then the conclusion of the lemma holds trivially for $P_{\Bw_{\bar{U}}} :=\mathbf{0}$. Otherwise, we apply the last conclusion of Corollary \ref{cor:ILOlinear} to the sequence $\{\langle \Br_i(A_U), \Bw_{\bar{U}}\rangle, \langle \Br_i({A_U}^T), \Bv_{\bar{U}}\rangle , i\in U\}$ (after a rescaling). As a consequence, we obtain an index set $I_{\Bw_{\bar{U}}}\subset U$ of size $|U|-n^\ep$ and a proper symmetric GAP $P_{\Bw_{\bar{U}}}\subset \C$ of rank $O_{B,\ep}(1)$ and size $n^{O_{B,\ep}(1)}$, together with its elements $q_i(\Bw_{\bar{U}})$, such that $|\langle \Br_i(A_U) , \Bw_{\bar{U}} \rangle -q_i(\Bw_{\bar{U}})|\le \beta$ for all $i\in I_{\Bw_{\bar{U}}}$.
\end{proof}

\subsection{Property  of the $q_i(\Bw_{\bar{U}})$'s.}  We now work with the GAP elements $q_i(\Bw_{\bar{U}})$, where $\Bw_{\bar{U}}\in G$. Because these points occupy a large part of an integer box, we can infer a great deal of structural relation among them. To do this, we first pause to introduce a pleasant property of generalized arithmetic progressions.

Assume that $P=\{k_1g_1+\dots+k_rg_r | -K_i\le k_i \le K_i\}$ is a proper symmetric GAP, which contains a set $U=\{u_1,\dots. u_n\}$. We consider $P$ together with the map $\Phi: P \rightarrow \R^r$ which maps $k_1g_1+\dots+k_rg_r$ to $(k_1,\dots,k_r)$. Because $P$ is proper, this map is bijective. We know that $P$ contains $U$, but we do not know yet that $U$ is non-degenerate in $P$ in the sense that the set $\Phi(U)$ has full rank in $\R^{r}$. In the later case, we say $U$ {\it spans} $P$. The following lemma states that we can always assume this without loss of any additive structure.

\begin{lemma}\label{theorem:fullrank}
Assume that $U$ is a subset of a proper symmetric GAP $P$ of rank $r$, then there exists a proper symmetric GAP $Q$ that contains $U$ such that the followings hold.

\begin{itemize}
\item $\rank(Q)\le r$ and $|Q|\le O_r(1)|P|$.
\item $U$ spans $Q$, that is, $\phi(U)$ has full rank in $\R^{\rank(Q)}$.
\end{itemize}
\end{lemma}

We refer the reader to \cite[Theorem 2.1]{Ng-QLO} for a short proof of this lemma.

{\bf Common generating indices.} By Lemma \ref{theorem:fullrank}, we may assume that the $q_i(\Bw_{\bar{U}})$ span $P_{\Bw_{\bar{U}}}$. We choose $s$ indices $i_{w_1},\dots,i_{w_s}$ from $I_{\Bw_{\bar{U}}}$  such that $q_{i_{y_j}}(\Bw_{\bar{U}})$ span $P_{\Bw_{\bar{U}}}$, where $s$ is the rank of $P_{\Bw_{\bar{U}}}$. Note that $s=O_{B,\ep}(1)$ for all $\Bw_{\bar{U}}\in G$. 

Consider the tuples $(i_{w_1},\dots,i_{w_s})$ for all $\Bw_{\bar{U}}\in G$. Because there are $\sum_{s} O_{B,\ep}(n^s) = n^{O_{B,\ep}(1)}$ possibilities these tuples can take, there exists a tuple, say $(1,\dots,r)$ (by rearranging the rows of $A_U$ if needed), such that $(i_{w_1},\dots,i_{w_s})=(1,\dots,r)$ for all $\Bw_{\bar{U}}\in G'$, a subset $G'$ of $G$ which satisfies 
\begin{equation}
\P_{\Bw_{\bar{U}}}(\Bw_{\bar{U}}\in G')\ge \P_{\Bw_{\bar{U}}}(\Bw_{\bar{U}}\in G)/n^{O_{B,\ep}(1)} =\gamma/n^{O_{B,\ep}(1)}.
\end{equation}

{\bf Common coefficient tuple}. For each $1\le i\le r$, we express $q_i(\Bw_{\bar{U}})$ in terms of the generators of $P_{\Bw_{\bar{U}}}$ for each $\Bw_{\bar{U}}\in G'$, 
$$q_i(\Bw_{\bar{U}}) = c_{i1}(\Bw_{\bar{U}})g_{1}(\Bw_{\bar{U}})+\dots + c_{ir}(\Bw_{\bar{U}})g_{r}(\Bw_{\bar{U}}),$$ 
where $c_{i1}(\Bw_{\bar{U}}),\dots c_{ir}(\Bw_{\bar{U}})$ are integers bounded by $n^{O_{B,\ep}(1)}$, and $g_{i}(\Bw_{\bar{U}})$ are the generators of $P_{\Bw_{\bar{U}}}$.

We will show that there are many $\Bw_{\bar{U}}$ that correspond to the same coefficients $c_{ij}$. 

Consider the collection of the coefficient-tuples $\Big(\big(c_{11}(\Bw_{\bar{U}}),\dots,c_{1r}(\Bw_{\bar{U}})\big);\dots; \big(c_{r1}(\Bw_{\bar{U}}),\dots c_{rr}(\Bw_{\bar{U}})\big)\Big)$ for all $\Bw_{\bar{U}}\in G'$. The number of possibilities these tuples can take is at most
$$(n^{O_{B,\ep}(1)})^{r^2} =n^{O_{B,\ep}(1)}.$$
There exists a coefficient-tuple, say  $\Big((c_{11},\dots,c_{1r}),\dots, (c_{r1},\dots c_{rr})\Big)$, such that  
$$\Big(\big(c_{11}(\Bw_{\bar{U}}),\dots,c_{1r}(\Bw_{\bar{U}})\big);\dots; \big(c_{r1}(\Bw_{\bar{U}}),\dots c_{rr}(\Bw_{\bar{U}})\big)\Big) =\Big((c_{11},\dots,c_{1r}),\dots, (c_{r1},\dots c_{rr})\Big)$$  
for all $\Bw_{\bar{U}}\in G''$, a subset of $G'$ which satisfies 
\begin{equation}
\P_{\Bw_{\bar{U}}}(\Bw_{\bar{U}}\in G'')\ge \P_{\Bw_{\bar{U}}}(\Bw_{\bar{U}}\in G')/n^{O_{B,\ep}(1)} \ge \gamma/n^{O_{B,\ep}(1)}.
\end{equation}

In summary, there exist $r$ tuples  $(c_{11},\dots,c_{1r}),\dots, (c_{r1},\dots c_{rr})$ (where we recall that $r=O(1)$ is the rank of the GAP) whose components are integers bounded by $n^{O_{B,\ep}(1)}$, such that the following holds for all $\Bw_{\bar{U}}\in G''$.

\begin{itemize}
\item $q_i(\Bw_{\bar{U}}) = c_{i1}g_{1}(\Bw_{\bar{U}})+\dots + c_{ir}g_{r}(\Bw_{\bar{U}})$, for $i=1,\dots,r$.
\item The vectors  $(c_{11},\dots,c_{1r}),\dots, (c_{r1},\dots c_{rr})$ span $\R^{\rank(P_{\Bw_{\bar{U}}})}$.
\end{itemize} 

Next, because $|I_{\Bw_{\bar{U}}}|\ge |U|-n^\ep$ for each $\Bw_{\bar{U}}\in G''$, by an averaging argument using Chebyshev's inequality, there exists a set $I\subset U$ of size $|U|-2n^\ep$ such that for each $i\in I$ we have 
\begin{equation}\label{eqn:shit}
\P_{\Bw_{\bar{U}}}(i\in I_{\Bw_{\bar{U}}}, \Bw_{\bar{U}}\in G'') \ge \P_{\Bw_{\bar{U}}}(\Bw_{\bar{U}}\in G'')/2.
\end{equation}  
From now on we fix an arbitrary row $\Br$ of index from $I$. We will focus on those $\Bw_{\bar{U}}\in G''$ where the index of $\Br$ belongs to $I_{\Bw_{\bar{U}}}$. 

{\bf Common coefficient tuple for each individual.} Because $q(\Bw_{\bar{U}}) \in P_{\Bw_{\bar{U}}}$ ($q(\Bw_{\bar{U}})$ is the element of $P_{\Bw_{\bar{U}}}$ that is $\beta$-close to $\langle \Br ,\Bw_{\bar{U}} \rangle$), we can write 
$$q(\Bw_{\bar{U}}) = c_{1}(\Bw_{\bar{U}})g_{1}(\Bw_{\bar{U}})+\dots +c_{r}(\Bw_{\bar{U}})g_{r}(\Bw_{\bar{U}})$$ 
where $c_{i}(\Bw_{\bar{U}})$ are integers bounded by $n^{O_{B,\ep}(1)}$.

For short, for each $i$ we denote by $\Bv_i$ the vector $(c_{i1},\dots,c_{ir})$, we will also denote by $\Bv_{\Br,\Bw_{\bar{U}}}$ the vector $(c_{1}(\Bw_{\bar{U}}),\dots c_{r}(\Bw_{\bar{U}}))$. 

Because $P_{\Bw_{\bar{U}}}$ is spanned by $q_1(\Bw_{\bar{U}}),\dots, q_{r}(\Bw_{\bar{U}})$, we have $k=\det(\mathbf{v}_1,\dots \mathbf{v}_r)\neq 0$, and by basic linear algebra
\begin{equation}\label{eqn:ILObilinear:det}
k q(\Bw_{\bar{U}}) + \det(\mathbf{v}_{\Br,\Bw_{\bar{U}}},\mathbf{v}_2,\dots,\mathbf{v}_r)q_{1}(\Bw_{\bar{U}}) +\dots + \det(\mathbf{v}_{\Br,\Bw_{\bar{U}}},\mathbf{v}_1,\dots,\mathbf{v}_{r-1})q_{r}(\Bw_{\bar{U}}) =0.
\end{equation}
It is crucial to note that $k$ is independent of the choice of $\Br$ and $\Bw_{\bar{U}}$. 

Next, because each coefficient of \eqref{eqn:ILObilinear:det} is bounded by $n^{O_{B,\ep}(1)}$, there exists a subset $G_{\Br}''$ of $G''$ such that all $\Bw_{\bar{U}}\in G_{\Br}''$ correspond to the same identity, and by \eqref{eqn:shit}
\begin{equation}
\P_{\Bw_{\bar{U}}}(\Bw_{\bar{U}}\in G_{\Br}'') \ge (\P_{\Bw_{\bar{U}}}(\Bw_{\bar{U}}\in G'')/2)/(n^{O_{B,\ep}(1)})^r = \gamma/n^{O_{B,\ep}(1)} = n^{-O_{B,\ep}(1)}.
\end{equation}

In other words, there exist integers $k_1,\dots,k_r$ depending on $\Br$, all bounded by $n^{O_{B,\ep}(1)}$, such that 
\begin{equation}\label{eqn:ILObilinear:q}
k q(\Bw_{\bar{U}}) + k_1 q_{1}(\Bw_{\bar{U}}) + \dots + k_r q_{r}(\Bw_{\bar{U}})=0
\end{equation}
for all $\Bw_{\bar{U}}\in G_{\Br}''$. 

\subsection{Passing back to $A_U$.} Because $q_i(\Bw_{\bar{U}})$ are $\beta$-close to $\langle \Br_i ,\Bw_{\bar{U}}\rangle $, it follows from \eqref{eqn:ILObilinear:q} that 
$$\Big| \langle k \Br,\Bw_{\bar{U}} \rangle+ \langle k_1  \Br_1 ,\Bw_{\bar{U}}\rangle + \dots +\langle  k_r \Br_{r} , \Bw_{\bar{U}}\rangle \Big| = \Big|\langle k\Br+k_1\Br_1+\dots+\Br_r, \Bw_{\bar{U}} \rangle \Big|\le n^{O_{B,\ep}(1)}\beta.$$
Furthermore, as $\P_{\Bw_{\bar{U}}}(\Bw_{\bar{U}}\in G_{\Br}'') =n^{-O_{B,\ep}(1)}$, we have
\begin{equation}\label{eqn:ILObilinear:b}
\P_{\Bw_{\bar{U}}}\Big(|\langle k\Br+k_1\Br_1+\dots+k_r\Br_r ,\Bw_{\bar{U}} \rangle |\le n^{O_{B,\ep}(1)}\beta\Big)=n^{-O_{B,\ep}(1)}.
\end{equation}

As \eqref{eqn:ILObilinear:b} holds for any row $\Br$ indexing from $I$, this completes the proof of Lemma \ref{lemma:A_U}.

\section{ Random matrix: the inverse step}\label{section:step1}

We now give a proof of Theorem \ref{theorem:step1}. We first apply Theorem \ref{theorem:ILO:quadratic} to $a_{ij}$ to obtain
$$\P_{\Bz} \left(\left|\left\langle \Bz,  k\row_i(A_{n-1})+\sum_{j}k_{ii_j}\row_{i_j}(A_{n-1}) \right\rangle\right|\le n^{-A+O_{B,\ep}(1)} \right) \ge n^{-O_{B,\ep}(1)},$$
where $A_{n-1}=(a_{ij})_{2\le i,j\le n}$.

For short, we denote by $\row_i'$ the vector $k\row_i(A_{n-1})+\sum_{j}k_{ii_j}\row_{i_j}(A_{n-1})$. Thus, for any $i\in I$,
\begin{equation}\label{eqn:step1:0}
\P_{\Bz} \left(|\langle \Bz, \row_i' \rangle|\le n^{-A+O_{B,\ep}(1)} \right) \ge n^{-O_{B,\ep}(1)}.
\end{equation}

Set 
$$K=n^{-A/2}.$$ 
We consider two cases.

{\bf Case 1.}({\it{non-degenerate case}}). There exists $i_0\in I$ such that $\|\row_{i_0}'\|_2 \ge K$. Because $\row_{i_0}'=k\row_{i_0}(A_{n-1})+ \sum_{j\in I_0} k_{i_0j} \row_j(A_{n-1}y)$, $\row_{i_0}'$ is orthogonal to $n-|I|-1=n-O_{B,\ep}(1)$ column vectors of $M_{n-1}$.

Set 
$$\Bv:=\row_{i_0}'/\|\row_{i_0}'\|_2.$$
Hence, $\langle \Bv ,\col_i(M_{n-1}) \rangle =0$ for at least $n-O_{B,\ep}(1)$ column vectors of $M_{n-1}$.

Also, it follows from \eqref{eqn:step1:0} that
\begin{equation}\label{eqn:step1:2}
\P_\Bz\left(|\langle \Bz, \Bv \rangle|\le n^{-A/2+O_{B,\ep}(1)}\right)\ge n^{-O_{B,\ep}(1)}.
\end{equation}

Next, Corollary \ref{cor:complex} applied to \eqref{eqn:step1:2} implies that $\Bv$ can be approximated by a vector $\Bu$ as follows.
\begin{itemize}
\item $|u_i-v_i|\le n^{-A/2+O_{B,\ep}(1)}$ for all $i$.
\item There exists a GAP of rank $O_{B,\ep}(1)$ and size $n^{O_{B,\ep}(1)}$ that contains at least $n-n^\ep$ components $u_i$.
\item All the components $u_i$, and all the generators of the GAP are rational complex numbers of the form $\frac{p}{q}+\sqrt{-1}\frac{p'}{q'}$, where $|p|,|q|,|p'|,|q'| \le n^{A/2+O_{B,\ep}(1)}$.
\end{itemize}

Note that, by the approximation above, we have $\|\Bu\|_2\asymp 1$ and $|\langle \Bu,\col_i(M_{n-1}) \rangle| \le n^{-A/2+O_{B,\ep}(1)}$ for at least $n-O_{B,\ep}(1)$ column vectors of $M_{n-1}$.

{\bf Case 2.}({\it{degenerate case}})  $\|\row_i'\|_2 \le K$ for all $i\in I$. Hence, with $I_0:=\{i_1,\dots, i_r\}$ 
\begin{equation}\label{eqn:step1:1}
\left\|k\row_i(A_{n-1})+ \sum_{j\in I_0} k_{ij} \row_j(A_{n-1})\right\|_2=\|\row_i'\|_2\le K.
\end{equation} 

Without loss of generality we can assume that $I$ and $I_0$ are disjoint. Next, because $\sum_j\|\col_j(A_{n-1})\|_2^2=1$, there exists an index $j_0$ such that $\|\col_{j_0}(A_{n-1})\|_2\ge n^{-1/2}$. Consider this column vector.

It follows from \eqref{eqn:step1:1} that for any $i\in I$,
$$\left|k\col_{j_0}(i)+\sum_{j\in I_0} k_{ij}\col_{j_0}(j)\right|\le K.$$
The above inequality means that the components $\col_{j_0}(i)$ of $\col_{j_0}(A_{n-1})$ belong to a GAP generated by $\col_{j_0}(j)/k, j\in I_0$, up to an error $K$. This suggests the following approximation.  

For each $j\notin I$, we approximate $\col_{j_0}(j)$ by a number $v_j$ of the form $(1/\lfloor 2K^{-1} \rfloor) \cdot \Z^2$ such that $|v_j-\col_{j_0}(j)|\le K$. We next set 
$$v_i:=\sum_{j\in I_0}k_{ij}v_j/k$$ 
for any $i\in I$. 
Thus, $v_i$ belongs to a GAP of rank $O_{B,\ep}(1)$ and size $n^{O_{B,\ep}(1)}$ for all $i\in I$.

With $\Bv=(v_1,\dots, v_{n-1})$, we have
$$\|\Bv -\col_{j_0}(A_{n-1})\|_2\le Kn^{O_{B,\ep}(1)}.$$ 
Furthermore, by Condition \ref{condition:bound}, and because $\langle \col_{j_0}(A_{n-1}),\row_i(M_{n-1}) \rangle =0$ for $i\neq j_0$, we infer that
$$|\langle \Bv,\row_i(M_{n-1}) \rangle| \le Kn^{O_{B,\ep}(1)}.$$ 

Note that $\|\Bv\|_2\gg n^{-1/2}$. Set $\Bu:=\lfloor 1/\|\Bv\|_2\rfloor \cdot \Bv$, we then obtain
\begin{itemize}
\item $|\langle \Bu,\row_i(M_{n-1})\rangle| \le n^{-A/2+O_{B,\ep}(1)}$ for $n-2$ rows of $M_{n-1}$.
\item There exists a GAP of rank $O_{B,\ep}(1)$ and size $n^{O_{B,\ep}(1)}$ that contains at least $n-2n^\ep$ components $u_i$.
\item All the components $u_i$, and all the generators of the GAP are rational complex numbers of the form $\frac{p}{q}+\sqrt{-1} \frac{p'}{q'}$, where $|p|,|q|,|p'|,|q'| \le n^{A/2+O_{B,\ep}(1)}$.
\end{itemize}

\section{Random matrix: the counting step}\label{section:step2}

We now give a proof of Theorem \ref{theorem:step2}. Without loss of generality, we assume $\ep$ to be sufficiently small. Our argument, which follows the ``divide and conquer'' strategy, is simple and purely combinatorial. We note that a similar but simpler treatment for symmetric matrices has appeared in \cite[Section 5]{Ng-sym}.

 For convenience, let us replace $M_{n-1}$ by $M_n$.  We will consider the case $|\langle \Bu,\row_i(M_{n})\rangle| \le n^{-A/2+O_{B,\ep}(1)}$ for $n-O_{B,\ep}(1)$ rows of $M_{n}$ only, the remaining case  $|\langle \Bu,\col_i(M_{n})\rangle| \le n^{-A/2+O_{B,\ep}(1)}$ can be treated identically. 

Let $\mathcal{N}$ be the number of such structural vectors $\Bu$. Because each GAP is determined by its generators and dimensions, the number of $Q$'s is bounded by 
$$\# \{Q, \mbox{ there exists } \Bu \in \mathcal{N} \mbox{ such that } \Bu\in Q \} = (n^{2A+O_{B,\ep}(1)})^{O_{B,\ep}(1)} (n^{O_{B,\ep}(1)})^{O_{B,\ep}(1)} = n^{O_{A,B,\ep}(1)}.$$ 

Next, for a given $Q$ of rank $O_{B,\ep}(1)$ and size $n^{O_{B,\ep}(1)}$, there are at most $n^{n-2n^\ep}|Q|^{n-2n^\ep} = n^{O_{B,\ep}(n)}$ ways to choose the $n-2n^\ep$ components $u_i$ that $Q$ contains. Because the remaining components belong to the set $\{\frac{p}{q}+i \frac{p'}{q'}, |p|,|q|,|p'|,|q'|\le n^{A/2+O_{B,\ep}(1)}\}$, there are at most $(n^{2A+O_{B,\ep}(1)})^{2n^\ep}= n^{O_{A,B,\ep}(n^\ep)}$ ways to choose them. 

Hence, we obtain the key bound 
\begin{equation}\label{eqn:step2:N'}
\mathcal{N}\le n^{O_{A,B,\ep}(1)}   n^{O_{B,\ep}(n)}  n^{O_{A,B,\ep}(n^\ep)} = n^{O_{B,\ep}(n)}.
\end{equation}

Set $\beta_0:=n^{-A/2+O_{B,\ep}(1)}$, the bound obtained from the conclusion of Theorem \ref{theorem:step1}. For a given vector $\Bu$, we define $\P_{\beta_0}(\Bu)$ as follows
$$\P_{\beta_0}(\Bu):=\P\Big(|\langle \Bu,\row_i(M_{n})\rangle| \le  \beta_0 \mbox{ for } n-O_{B,\ep}(1) \mbox{ rows of } M_{n-1}\Big).$$

For the sake of discussion, let us pretend for now that the rows of $X_{n}$ are independent. By definition, the vector $\Bu$ is orthogonal to almost every row of $M_{n}$. Thus, if $\Bu$ is fixed, the probability of this event is bounded by 
$$\P_{\beta_0}(\Bu) \le (\P_{\Bx}(|u_1x_1+\dots+u_nx_n|\le \beta_0))^{n-O(1)}:=\gamma^{n-O(1)},$$ 
where $x_1,\dots,x_n$ are i.i.d. copies of $\xi$. 

Now, if $\gamma$ is small, say $n^{-\Omega(1)}$, then $\P_{\beta_0}(\Bu)$ is $n^{-\Omega(n)}$. Thus the contribution of these $\P_{\beta_0}(\Bu)$ in the total sum  $\sum_{\Bu}\P_{\beta_0}(\Bu)$ is negligible, taking into account of the bound $n^{O(n)}$ of $\mathcal{N}$. 

Next, if $\gamma$ is comparably large, $\gamma=n^{-O(1)}$, then by Theorem \ref{theorem:ILOlinear}, most of the components $u_i$ are close to a new GAP of rank $O(1)$ and of size $O(\gamma^{-1}/\sqrt{n})$. This would then enable us to approximate $\Bu$ by a new vector $\Bu'$ in such a way that $|\langle \Bu',\row_i(M_n)\rangle|$ is still of order $O(\beta_0)$ and the components of $\Bu'$ are now from the new GAPs. The number $\mathcal{N'}$ of these $\Bu'$ can be bounded by $(\gamma^{-1}/n^\ep)^{n}$, while we recall that $\P_{\beta_0}(\Bu')$ is of order $\gamma^{-n}$. Thus, summing over $\Bu'$ we obtain the desired bound
$$\sum_{\Bu'} \P_{\beta_0}(\Bu') \le \#\{ \mbox{ new GAPs } \} (\gamma^{-1}/n^\ep)^{n} \gamma^{-n} = O(n^{-\ep n+O(1)}).$$

To our model $M_n=F_n+X_n$, we will mainly follow the heuristic above. Our strategy is to classify $\Bu$ into two classes: 
\begin{enumerate}[(i)]
\item $\mathcal{B}'$ contains those $\Bu$ for which $\P_{\beta_0}(\Bu)$ is very small, and thus $\sum_{\Bu\in \mathcal{B}'}\P_{\beta_0}(\Bu)$  is negligible;
\item the other class $\mathcal{B}$ contains of $\Bu$ of relatively large $\P_{\beta_0}(\Bu)$. To deal with those $\Bu$ of the second type, we will not control $\sum_{\Bu\in \mathcal{B}} \P_{\beta_0}(\Bu)$ directly but pass to a class of new vectors $\Bu'$ that are also almost orthogonal to many rows of $M_{n}$, while the probability $\sum_{\Bu'} \P_{\beta_0}(\Bu')$ is of order $O(n^{-\ep n})$.
\end{enumerate}

What makes our analysis harder is that the estimate $\P_{\beta_0}(\Bu) \le (\P_{\Bx}(|u_1x_1+\dots+u_nx_n|\le \beta_0))^{n-O(1)}$ is no-longer valid for our random matrix model.

\subsection{ Technical reductions and upper bounds for $\P_{\beta_0}(\Bu)$ }\label{subsection:observation} By paying a factor of $n^{O_{B,\ep}(1)}$ in probability, we may assume that $|\langle \Bu,\row_i(M_{n}) \rangle | \le  \beta_0$ for the first  $n-O_{B,\ep}(1)$ rows of $M_{n}$. Also, by paying another factor of $n^{n^\ep}$ in probability, we may assume that the first $n_0$ components $u_i$ of $\Bu$  \footnote{Roughly speaking, in later analysis we will be fixing the columns of $M_n$ that correspond to the unstructured components of $\Bu$. Thus the assumption that  the first $n_0$ components of $\Bu$ come from a GAP is slightly more difficult than other cases as it involves more dependencies and less structural components. However, there is no major difference among the treatments.} belong to a GAP $Q$, and $u_{n_0}\ge 1/2\sqrt{n-1}$ (recall that $\Bu \asymp 1$), where 
$$n_0:=n-2n^\ep.$$ 
We refer to the remaining $u_i$'s as exceptional components. Note that these extra factors do not affect our final bound $\exp(-\Omega(n))$.

For given $\beta>0$ and $i\le n_0$, we define
$$\gamma_{\beta}^{(i)}(\Bu):=\sup_a\P_{x_i,\dots,x_{n_0}}(|x_iu_i+\dots+x_{n_0}u_{n_0}-a|\le \beta),$$
where $x_i,\dots, x_{n_0}$ are i.i.d copies of $\xi$.

A crucial observation is that, by exposing the rows of $M_{n-1}$ one by one, and due to symmetry (i.e. $x_{ij}$ is independent from all other entries except $x_{ji}$), the probability $\P_{\beta}(\Bu)$ that $|\langle \Bu,\row_i(M_{n-1})\rangle|\le \beta$ for all $i\le n-O_{B,\ep}(1)$ can be bounded by 
\begin{eqnarray}\label{eqn:step2:Pu}
\P_{\beta}(\Bu)&\le& \prod_{1\le i\le n-O_{B,\ep}(1)}\sup_a\P_{x_i,\dots,x_{n-1}}(|x_iu_i+\dots+x_{n-1}u_{n-1}-a|\le \beta)\nonumber \\
&\le&  \prod_{1\le i\le n_0}\sup_a\P_{x_i,\dots,x_{n_0}}(|x_iu_i+\dots+x_{n_0}u_{n_0}-a|\le \beta) \nonumber \\
&=&\prod_{1\le i\le n_0}\gamma_{\beta}^{(i)}(\Bu).
\end{eqnarray}

Also, because $u_{n_0}\ge 1/2\sqrt{n-1}$, there exist absolute positive constants $c_1,c_2$ such that $c_2<1$ and for any $\beta<c_1/2 \sqrt{n-1}$ we have
\begin{eqnarray}\label{eqn:step2:upper}
\gamma_{\beta}^{(k)}(\Bu) &\le& \sup_a \P_{x_{n_0}}(|x_{n_0}u_{n_0}-a|\le \beta)\nonumber \\
&\le& 1-c_2.
\end{eqnarray}
Thus, 
$$\P_\beta(\Bu)\le (1-c_2)^{n_0}=(1-c_2)^{(1-o(1))n}.$$

\subsection{Classification} Next, let $C$ be a sufficiently large constant depending on $B$ and $\ep$ but not $A$. We classify $\Bu$ into two classes $\mathcal{B}$ and $\mathcal{B}'$, depending on whether $\P_{\beta_0}(\Bu)\ge n^{-Cn}$ or not. 

Because of \eqref{eqn:step2:N'} and with $C$ sufficiently large,
\begin{equation}\label{eqn:step2:B'}
\sum_{\Bu\in \mathcal{B}'}\P_{\beta_0}(\Bu)\le n^{O_{B,\ep}(n)}/n^{Cn} \le n^{-n/2}.
\end{equation}
For the rest of this section, we focus on $\Bu \in \mathcal{B}$.

\subsection{Approximation for  vectors of ``low complexity''} Let $\mathcal{B}_1$ be the collection of $\Bu\in \mathcal{B}$ satisfying the following property: for any $n'$ components $u_{i_1},\dots,u_{i_{n'}}$ among the $u_1,\dots, u_{n_0}$, we have 
\begin{equation}\label{eqn:step2:degenerate}
\sup_a\P_{x_{i_1},\dots,x_{i_{n'}}}\left(\left|u_{i_1}x_{i_1}+\dots+u_{i_{n'}}x_{i_{n'}}-a\right|\le n^{-B-4}\right)\ge (n')^{-1/2+o(1)}.
\end{equation}

Here we set 
$$n':=n^{1-\ep}.$$
For concision we set $\beta=n^{-B-4}$. It follows from Theorem \ref{theorem:ILOlinear} that, among any $u_{i_1},\dots,u_{i_{n'}}$, there are, say, at least $n'/2+1$ components that belong to a ball of radius $\beta$ (because our GAP now has only one element). A simple covering argument then implies that there is a ball of radius $2\beta$ that contains all but $n'-1$ components $u_i$.
 
Thus there exists a vector $\Bu'\in (2\beta)\cdot (\Z+ \sqrt{-1} \Z)$ satisfying the following conditions.
\begin{itemize}
\item $|u_i-u_i'|\le 4\beta$ for all $i$.
\item $u_i'$ takes the same value $u$ for at least $n_0-n'$ indices $i$.
\end{itemize}

In other words, $\Bu$ can be approximated by a vector of  ``low complexity''. Because of the approximation and Condition \ref{condition:bound}, whenever $|\langle \Bu, \row_i(M_{n-1})\rangle|\le \beta_0$, we have 
$$|\langle \Bu',\row_i(M_{n-1})\rangle|\le n(n^{B+1}+n^\alpha)(4\beta)+\beta_0:=\beta'.$$ 
It is clear from the bound on $\beta$ and $\beta_0$ that $\beta' \le c_1/2\sqrt{n-1}$, and thus by \eqref{eqn:step2:upper}, 
$$\P_{\beta'}(\Bu') \le (1-c_2)^{(1-o(1))n}.$$

Now we bound the number of $\Bu'$ obtained from the approximation. First, there are $O(n^{n-n_0+n'}) = O(n^{2n^{1-\ep}})$ ways to choose those $u_i'$ that take the same value $u$, and there are just $O(\beta^{-1})$ ways to choose $u$. The remaining components belong to the set $(2\beta)^{-1}\cdot (\Z + \sqrt{-1}\Z)$, and thus there are at most $O((\beta^{-1})^{n-n_0+n'})= O(n^{O_{A,B,\ep}(n^{1-\ep})})$ ways to choose them.

Hence we obtain the total bound
\begin{align*}
\P&(\exists \Bu \in \mathcal{B}_1 \mbox{ such that for all } i\le n-O_{B,\ep}(1), \langle \Bu, \Br_i(M_{n-1})\rangle \le \beta_0) \\
&\qquad\le \sum_{\Bu'}\P_{\beta'}(\Bu') \\
&\qquad\le O(n^{2n^{1-\ep}}) O(n^{O_{A,B,\ep}(n^{1-\ep})}) (1-c_2)^{(1-o(1))n}\\
&\qquad\le (1-c_2)^{(1-o(1))n}.
\end{align*}

\subsection{Approximation for vectors of ``high complexity''} Assume that $\Bu\in \mathcal{B}_2:=\mathcal{B}\backslash \mathcal{B}_1$. 
By exposing the rows of $M_{n-1}$ accordingly, and by paying an extra factor $\binom{n_0}{n'}=O(n^{n^{1-\ep}})$ in probability, we may assume that the components $u_{n_0-n'+1},\dots,u_{n_0}$ satisfy the property 
$$\sup_a\P_{x_{n_0-n'+1},\dots,x_{n_0}}\left(\left|u_{n_0-n'+1}x_{n_0-n'+1}+\dots+u_{n_0}x_{n_0}-a\right|\le n^{-B-4}\right)\le (n')^{-1/2+o(1)}$$
\begin{equation}\label{eqn:step2:non-degenerate}
\le n^{-1/2+\ep/2+o(1)}.
\end{equation}

{\bf Preparation.} Next, define a radius sequence $\beta_k, k\ge 0$ where $\beta_0=n^{-A/2+O_{B,\ep}(1)}$ is the bound obtained from the conclusion of Theorem \ref{theorem:step1}, and
$$\beta_{k+1}:= (n^{B+2}+n^{\alpha+1}+1)^2 \beta_k.$$

Recall from \eqref{eqn:step2:Pu} that 
$$\P_{\beta_k}(\Bu) \le \prod_{1 \le i \le n_0-n'} \gamma_{\beta_{k}}^{(i)}(\Bu)=:\pi_{\beta_k}(\Bu).$$

Roughly speaking, the reason we truncated the product here is that whenever $i\le n_0-n'$ and $\beta_k$ is small enough, the terms $\gamma_{\beta_k}^{(i)}(\Bu)$ are smaller than $(n')^{-1/2+o(1)}$, owing to \eqref{eqn:step2:non-degenerate}. This fact will allow us to gain some significant factors when applying Theorem \ref{theorem:ILOlinear}. 

Observe that if $|\langle \Bu,\row_i(M_n)\rangle|\le \beta_k$ and if $\Bu'$ is an approximation of $\Bu$ such that $|u_i-u_i'|\le \beta_k$ for all $i$, then
\begin{align}\label{eqn:u':discussion}
&\pi_{\beta_k}(\Bu)=\prod_{1 \le i\le n_0-n'} \sup_a \P_{x_i,\dots,x_{n_0}}\Big(|u_ix_i+\dots+u_{n_0}x_{n_0}-a|\le \beta_k\Big)\nonumber \\
&\le \prod_{1 \le i\le n_0-n'} \sup_a \P_{x_i,\dots,x_{n_0}}\Big(|u_i'x_i+\dots+u_{n_0}'x_{n_0}-a|\le (n(n^{B+1})+n^\alpha) \beta_k +\beta_k\Big) \nonumber \\
&=\prod_{1 \le i\le n_0-n'} \sup_a \P_{x_i,\dots,x_{n_0}}\Big(|u_i'x_i+\dots+u_{n_0}'x_{n_0}-a|\le (n^{B+2}+n^{\alpha+1}+1)\beta_k\Big) \nonumber \\
&\le \prod_{1 \le i\le n_0-n'} \sup_a \P_{x_i,\dots,x_{n_0}}\Big(|u_ix_i+\dots+u_{n_0}x_{n_0}-a|\le (n^{B+2}+n^{\alpha+1}+1)^2\beta_k\Big) \nonumber \\
&=\pi_{\beta_{k+1}}(\Bu).
\end{align}

Naturally, we hope that after the approximation $\P_{(n^{B+2}+n^{\alpha+1}1)\beta_k}(\Bu')$ does not increase much compared to the original $\P_{\beta_k}(\Bu)$, where we recall that $n^\alpha$ is the upper bound for the entries of $F_n$. That motivates us to consider a special radius $\beta_{k_0}$ with respect to $\Bu$ defined below.

Note that the bounded sequence $\pi_{\beta_k}(\Bu)$ increases with $k$, and recall that $\pi_{\beta_0}(\Bu)\ge n^{-Cn}$ for $\Bu\in \mathcal{B}$. Thus, by the pigeonhole principle, there exists $k_0:=k_0(\Bu)\le C\ep^{-1}$ such that 
\begin{equation}\label{eqn:step2:pigeon-hole}
\pi_{\beta_{k_0+1}}(\Bu) \le n^{\ep n} \pi_{\beta_{k_0}}(\Bu).
\end{equation}

It is crucial to note that, since $A$ was chosen to be sufficiently large compared to $O_{B,\ep}(1)$ and $C$, we have 
$$\beta_{k_0+1}\le n^{-B-4}.$$

Having mentioned the upper bound of $\gamma_{\beta_i}^{(i)}(\Bu)$, we now turn to its lower bound. Because of Condition \ref{condition:bound}, and $u_i\le 1$ for all $i$, and by the pigeonhole principle, the following trivial bound holds for any $\beta\ge \beta_0$ and $i\le n_0-n'$,
$$\gamma_{\beta}^{(i)}(\Bu) \ge \beta n^{-B-2} \ge \beta_0 n^{-B-2} = n^{-A/2+O_{B,\ep}(1)}.$$

{\bf Subclasses of $\Bu$ in terms of the sequence $(\gamma^{(i)}(\Bu))$.} Set 
$$I:=[n^{-A/2+O_{B,\ep}(1)}, n^{-1/2+\ep/2+o(1)}]:=[l_I,r_I].$$ 
We next divide it into $K=(A/2+O_{B,\ep}(1))\ep^{-1}$ sub-intervals $I_k=[l_I n^{k\ep},l_I n^{(k+1)\ep}]$. For short, we denote by $l_k$ the left endpoint of each $I_k$. Thus $l_k=n^{-A/2+O_{B,\ep}(1)+k\ep}$.

With all the necessary settings above, we now classify $\Bu$ based on the distribution of the $\gamma_{\beta_{k_0}}^{(i)}(\Bu), 1\le i \le n_0-n^{1-\ep}$ . 

For each $0\le k_0 \le C\ep^{-1}$ and each tuple $(m_0,\dots,m_K)$ satisfying $m_0+\dots+m_K=n_0-n'$, we let $\mathcal{B}_{k_0}^{(m_0,\dots,m_K)}$ denote the collection of those
$\Bu$ from $\mathcal{B}_2$ that satisfy the following conditions.
\begin{itemize}
\item $k_0(\Bu)= k_0$.
\item There are exactly $m_k$ terms of the sequence $(\gamma_{\beta_{k_0}}^{(i)}(\Bu))$ that belong to the interval $I_k$. In other words, if $m_0+\dots+m_{k-1}+1 \le i \le m_0+\dots+m_k$ then $\gamma_{\beta_{k_0}}^{(i)}(\Bu)\in I_k$.
\end{itemize}

{\bf The approximation.} Now we will use Theorem \ref{theorem:ILOlinear} to approximate $\Bu \in \mathcal{B}_{k_0}^{(m_0,\dots,m_K)} $ as follows. 

\begin{itemize}
\item {\it First step}. Consider each index $i$ in the range $1 \le i\le m_0$. Because $\gamma_{\beta_{k_0}}^{(1)}\in I_0$, we apply Theorem \ref{theorem:ILOlinear} to approximate $u_i$ by $u_i'$ such that $|u_i-u_i'|\le \beta_{k_0}$ and the $u_i'$ belong to a GAP $Q_0$ of rank $O_{B,\ep}(1)$ and size $O(l_0^{-1}/\sqrt{n'})=O(l_0^{-1}/n^{1/2-\ep})$ for all but $n^{1-2\ep}$ indices $i$. Furthermore, all $u_i'$ have the form $\beta_{k_0}\cdot (\frac{p}{q}+\sqrt{-1}\frac{p'}{q'})$, where $|p|,|q|,|p'|,|q'| =O(n \beta_{k_0}^{-1})=O(n^{A/2+O_{B,\ep}(1)})$.  
\item {\it $k$-th step, $1\le k\le K$}. We focus on $i$ from the range $n_0+\dots+ n_{k-1}+1\le i\le n_0+\dots+n_k$. Because $\gamma_{\beta_{k_0}}^{(n_0+.\dots+n_{k-1}+1)}\in I_k$, we apply Theorem \ref{theorem:ILOlinear} to approximate $u_i$ by $u_i'$ such that $|u_i-u_i'|\le \beta_{k_0}$ and the $u_i'$ belong to a GAP $Q_k$ of rank $O_{B,\ep}(1)$ and size $O(l_k^{-1}/n^{1/2-\ep})$ for all but $n^{1-2\ep}$ indices $i$. Furthermore, all $u_i'$ have the form $\beta_{k_0}\cdot (p/q + \sqrt{-1} p'/q')$, where $|p|,|q|,|p'|,|q'| =O(n \beta_{k_0}^{-1})=O(n^{A/2+O_{B,\ep}(1)})$.  
\item For the remaining components $u_i$, we just simply approximate them by the closest point in $\beta_{i_0}\cdot (\Z+\sqrt{-1}\Z)$. 
\end{itemize}

We have thus provided an approximation of $\Bu$ by $\Bu'$ satisfying the following properties. 
\begin{enumerate}[(i)]
\item $|u_i-u_i'|\le \beta_{k_0}$ for all $i$. 
\item $u_i'\in Q_k$ for all but $n^{1-2\ep}$ indices $i$ in the range $m_0+\dots+ m_{k-1}+1\le i\le m_0+\dots+m_k$.
\item All the $u_i'$, including the generators of $Q_k$, belong to the set $\beta_{k_0}\cdot \{p/q +\sqrt{-1} p'/q' , |p|,|q|,|p'|,|q'|\le n^{A/2+O_{B,\ep}(1)}\}$. 
\item $Q_k$ has rank $O_{B,\ep}(1)$ and size $|Q_k|=O(l_k^{-1}/n^{1/2-\ep})$.
\end{enumerate}

{\bf Property of $\Bu'$.} Let $\mathcal{B'}_{k_0}^{(m_1,\dots,m_K)}$ be the collection of all $\Bu'$ obtained from $\Bu\in \mathcal{B}_{k_0}^{(m_1,\dots,m_K)}$ as above. Observe that, as $|\langle \Bu, \row_i(M_{n})\rangle|\le \beta_{k_0}$ for all $i\le n-O_{B,\ep}(1)$, we have
\begin{equation}\label{eqn:step2:u'}
|\langle \Bu', \row_i(M_{n})\rangle| \le (n^{B+2}+n^{\alpha+1}+1) \beta_{k_0}.
\end{equation}
Hence, in order to justify Theorem \ref{theorem:step2} in the case $\Bu\in \mathcal{B}_2$, it suffices to show that the probability that \eqref{eqn:step2:u'} holds for all $i\le n-O_{B,\ep}(1)$, for some $\Bu'\in \mathcal{B'}_{k_0}^{(m_1,\dots,m_K)}$, is small. 

Consider a $\Bu'\in \mathcal{B'}_{k_0}^{(m_1,\dots,m_K)}$ and the  probability $\P_{(n^{B+2}+n^{\alpha+1}+1)\beta_{k_0}}(\Bu')$ that  \eqref{eqn:step2:u'} holds for all $i\le n-O_{B,\ep}(1)$. By the discussion about \eqref{eqn:u':discussion}, we have
$$\P_{(n^{B+2}+n^{\alpha+1}+1)\beta_{k_0}}(\Bu')\le \pi_{\beta_{k_0+1}}(\Bu) \le n^{\ep n}\pi_{\beta_{k_0}}(\Bu),$$
where in the second inequality we used \eqref{eqn:step2:pigeon-hole}.

We recall from the definition of $\mathcal{B}_{k_0}^{(m_1,\dots,m_K)}$ that 
\begin{align*}
\pi_{\beta_{k_0}}(\Bu) \le \prod_{k=1}^K l_{k+1}^{m_k} &=  n^{\ep(m_1+\dots+m_k)} \prod_{k=1}^K l_k^{m_k}\\
&\le n^{\ep n} \prod_{k=1}^K l_k^{m_k}.
\end{align*}
Hence,
\begin{equation}\label{eqn:step2:Pu'}
\P_{(n^{B+2}+n^{\alpha+1}+1)\beta_{k_0}}(\Bu') \le n^{2\ep n}  \prod_{k=1}^K l_k^{m_k}.
\end{equation}

{\bf The size of $\mathcal{B'}_{k_0}^{(m_1,\dots,m_K)}$.} In the next step of the argument, we bound the size of $\mathcal{B'}_{k_0}^{(m_1,\dots,m_K)}$. Because each $Q_k$ is determined by its $O_{B,\ep}(1)$ generators from the set  $\beta_{k_0}\cdot \{\frac{p}{q}+i\frac{p'}{q'}, |p|,|q|,|p'|,|q'|\le n^{A/2+O_{B,\ep}(1)}\}$, and its dimensions from the integers bounded by $n^{O_{B,\ep}(1)}$, there are $n^{O_{A,B,\ep}(1)}$ ways to choose each $Q_k$. So the total number of ways to choose $Q_1,\dots, Q_K$ is bounded by
$$(n^{O_{A,B,\ep}(1)})^{K}= n^{O_{A,B,\ep}(1)}.$$

Next, after locating $Q_k$, the number $\mathcal{N}_1$ of ways to choose $u_i'$ from each $Q_k$ is 
\begin{align*}
\mathcal{N}_1&\le \prod_{k=1}^K \binom{m_k}{n^{1-2\ep}} |Q_k|^{m_k-n^{1-2\ep}}\\ 
&\le 2^{m_1+\dots+m_K} \prod_{k=1}^K |Q_k|^{m_k}\\
&\le (O(1))^n \prod_{k=1}^Kl_k^{-m_k}/n^{(1/2-\ep)(m_1+\dots+m_k)}\\
&\le \prod_{k=1}^Kl_k^{-m_k}/n^{(1/2-\ep-o(1))n},
\end{align*}
where we used the bound $|Q_k|=O(l_k^{-1}/n^{1/2-\ep})$ for each $k$.

The remaining components $u_i'$ can take any value from the set $\beta_{k_0}\cdot \{\frac{p}{q}+i \frac{p'}{q'}, |p|,|q|,|p'|,|q'|\le n^{A/2+O_{B,\ep}(1)}\}$, so the number $\mathcal{N}_2$ of ways to choose them is bounded by
$$\mathcal{N}_2 \le (n^{A+O_{B,\ep}(1)})^{2n^\ep + Kn^{1-2\ep}} = n^{O_{A,B,\ep}(n^{1-2\ep})}.$$
Putting the bound for $\mathcal{N}_1$ and ${N}_2$ together, we obtain a bound $\mathcal{N}'$ for $|\mathcal{B'}_{k_0}^{(m_1,\dots,m_K)}|$,

\begin{equation}\label{eqn:step2:N}
\mathcal{N}'\le \prod_{k=1}^Kl_k^{-m_k}/n^{(1/2-\ep-o(1))n}.
\end{equation}

{\bf Closing the argument.} It follows from \eqref{eqn:step2:Pu'} and \eqref{eqn:step2:N} that
\begin{align*}
\sum_{\Bu'\in \mathcal{B'}_{k_0}^{(m_1,\dots,m_K)}}\P_{(n^{B+2}+n^{\alpha+1}+1)\beta_{k_0}}(\Bu') &\le  n^{2\ep n}  \prod_{k=1}^K l_k^{m_k}  \prod_{k=1}^Kl_k^{-m_k}/n^{(1/2-\ep-o(1))n}\\
&\le n^{-(1/2-3\ep -o(1))n}.
\end{align*}

Summing over the choices of $k_0$ and $(m_1,\dots,m_K)$ we obtain the bound 
$$ \sum_{k_0,m_1,\dots,m_K} \sum_{\Bu'\in \mathcal{B'}_{k_0}^{(m_1,\dots,m_K)}}\P_{(n^{B+2}+n^{\alpha+1}+1)\beta_{k_0}}(\Bu') \le n^{-(1/2-3\ep-o(1))n},$$
completing the treatment for incompressible vectors, and hence the proof of Theorem \ref{theorem:step2}.

\section{Proof of the elliptic law,  Theorems \ref{thm:real} and \ref{thm:complex}}\label{section:maintheorem:proof}

This section is devoted to the proof of Theorems \ref{thm:real} and \ref{thm:complex}.  We introduce the following notation.  Given a $n \times n$ matrix $A_n$, we let $\mu_{A_n}$ denote the empirical measure built from the eigenvalues of $A_n$ and $\nu_{A_n}$ denote the empirical measure built from the singular values of $A_n$.  That is,
$$ \mu_{A_n} := \frac{1}{n} \sum_{i \leq n} \delta_{\lambda_i(A_n)} $$
and 
$$ \nu_{A_n} := \frac{1}{n} \sum_{i \leq n} \delta_{\sigma_i(A_n)}, $$
where $\lambda_1(A_n), \ldots, \lambda_n(A_n) \in \mathbb{C}$ are the eigenvalues of $A_n$ and $\sigma_1(A_n) \geq \cdots \geq \sigma_n(A_n)$ are the singular values of $A_n$.  

In order to prove Theorems \ref{thm:real} and \ref{thm:complex}, we will show that, with probability one,
\begin{equation} \label{eq:muxf}
	\mu_{\frac{1}{\sqrt{n}}(X_n + F_n)} \longrightarrow \mu_\rho
\end{equation}
as $n \rightarrow \infty$, where $\mu_\rho$ is the uniform probability measure on the ellipsoid $\mathcal{E}_\rho$.  In particular, \eqref{eq:muxf} implies the almost sure convergence of the ESD of $\frac{1}{\sqrt{n}}(X_n+F_n)$ to the elliptic law with parameter $\rho$.  

To this end, let $\mathcal{P}(\mathbb{C})$ be the set of probability measures on $\mathbb{C}$ which integrate $\log |\cdot|$ in a neighborhood of infinity.  If $\mu\in \mathcal{P}(\mathbb{C})$, we define the logarithmic potential to be the function 
$$ U_\mu(z) := - \int_{\mathbb{C}} \log |z-\lambda|d\mu(\lambda). $$
We will make use of the following uniqueness property \cite[Lemma 4.1]{BC}: if $\mu, \nu \in \mathcal{P}(\mathbb{C})$ and $U_{\mu}(z) = U_{\nu}(z)$ for a.e. $z \in \mathbb{C}$, then $\mu = \nu$. 

We say a Borel function $f$ is uniformly integrable for a sequence of probability measures $\{\mu_n\}_{n \geq 1}$ if 
$$ \lim_{t \rightarrow \infty} \sup_{n \geq 1} \int_{\{|f|>t\}} |f| d\mu_n = 0. $$

For a complex $n \times n$ random matrix $A_n$, there is a connection between the measure $\mu_{A_n}$ and the family of measures $\{\nu_{A_n-zI}\}_{z \in \mathbb{C}}$.  In particular,
$$ U_{\mu_{A_n}}(z) = - \frac{1}{2n} \log \det  (A_n - zI)^\ast (A_n - zI) = - \int_{0}^\infty \log(s) d \nu_{A_n - zI}(s). $$
The work of Goldsheid and Khoruzhenko \cite{GK} is one of the first rigorous uses of the logarithmic potential to study random matrices.  We also refer the reader to the survey \cite{BC} for more details.  A key tool in the proof of Theorems \ref{thm:real} and \ref{thm:complex} is the following result from \cite{BC}. 

\begin{lemma}[Hermitization lemma, \cite{BC}] \label{lemma:hermitization}
Let $\{A_n\}_{n \geq 1}$ be a sequence of complex random matrices where $A_n$ is of size $n \times n$ for every $n\geq 1$.  Suppose that there exists a family of (non-random) probability measures $\{\nu_z\}_{z \in \mathbb{C}}$ such that for a.a. $z \in \mathbb{C}$, a.s.
\begin{enumerate}[(i)]
\item $\nu_{A_n - zI} \rightarrow \nu_z$ as $n \rightarrow \infty$
\item $\log$ is uniformly integrable for $\{\nu_{A_n-zI} \}_{n \geq 1}$.  
\end{enumerate}
Then there exists a probability measure $\mu \in \mathcal{P}(\mathbb{C})$ such that
\begin{enumerate}[(i)]
\item a.s. $\mu_{A_n} \rightarrow \mu$ as $n \rightarrow \infty$

\item for a.a. $z \in \mathbb{C}$,
$$ U_\mu(z) = - \int_{0}^\infty \log(s) d \nu_z(s). $$
\end{enumerate}
\end{lemma} 

\begin{remark} \label{remark:sing}
Since the singular values (and eigenvalues) of $(A_n-zI)^\ast (A_n-zI)$ are just $\sigma_1^2(A_n -z I), \sigma_2^2(A_n -z I), \ldots, \sigma_n^2(A_n - zI)$, it follows that
$$ \nu_{(A_n-zI)^\ast (A_n - zI)}(-\infty, x) = \nu_{A_n - zI}(-\infty,\sqrt{x}) $$
for all $x \geq 0$.  As a consequence, Lemma \ref{lemma:hermitization} can be equivalently formulated with the family of measures $\{\nu_{(A_n-zI)^\ast (A_n - zI)}\}_{z \in \mathbb{C}}$ rather than $\{\nu_{A_n - zI}\}_{z \in \mathbb{C}}$.  We will take advantage of this fact below. 
\end{remark}

In conjunction with Remark \ref{remark:sing}, we define the matrix 
$$ H_n := \left( \frac{1}{\sqrt{n}} X_n - zI \right)^\ast \left( \frac{1}{\sqrt{n}} X_n - zI \right). $$

For our purposes, we will need to show that the limiting measure $\mu \in \mathcal{P}(\mathbb{C})$ in Lemma \ref{lemma:hermitization} is given by $\mu_\rho$.  Fix $-1 < \rho < 1$.  We say the family of measures $\{\nu_z\}_{z \in \mathbb{C}}$ determine the elliptic law with parameter $\rho$ by Lemma \ref{lemma:hermitization} if 
$$ U_{\mu_\rho}(z) = - \int_{0}^\infty \log(s) d \nu_z(s) $$
for all $z \in \mathbb{C}$.  The existence of this family of measures was verified and used in \cite{N}.  
  
The key tool we use to prove Theorems \ref{thm:real} and \ref{thm:complex} is the following comparison lemma.  

\begin{lemma} \label{lemma:compare}
Let $0 \leq \mu \leq 1$ and $-1 < \rho < 1$ be given.  Let $\{X_n\}_{n \geq 1}$ and $\{Y_n\}_{n \geq 1}$ be sequences of random matrices that satisfy condition {\bf C0} with atom variables $(\xi_1, \xi_2)$ and $(\eta_1, \eta_2)$, respectively.  Assume $(\xi_1, \xi_2)$ and $(\eta_1, \eta_2)$ are from the $(\mu,\rho)$-family.  Assume for a.a. $z \in \mathbb{C}$ that a.s.
$$ \nu_{\frac{1}{\sqrt{n}}Y_n -zI} \longrightarrow \nu_z $$
as $n \rightarrow \infty$ for a family of deterministic measures $\{\nu_z\}_{z \in \mathbb{C}}$.  Assume $\{F_n\}_{n\geq 1}$ is a sequence of deterministic matrices such that $\rank(F_n) = o(n)$ and $\sup_{n} \frac{1}{n^2} \|F_n\|_2^2 < \infty$.  Then a.s. 
$$ \mu_{\frac{1}{\sqrt{n}}(X_n+F_n)} - \mu_{\frac{1}{\sqrt{n}}Y_n} \longrightarrow 0 $$
as $n \rightarrow \infty$.  
\end{lemma}

Lemma \ref{lemma:compare} is useful when we know the limit of $\mu_{\frac{1}{\sqrt{n}}Y_n}$.  For our purposes, we will take $\{Y_n\}_{n \geq 1}$ to be a sequence of matrices that satisfy condition {\bf C0} with jointly Gaussian entries.  In the real case, the limiting ESD of $\frac{1}{\sqrt{n}}Y_n$ was computed in \cite{N}.  

We divide the proof of Theorems \ref{thm:real} and \ref{thm:complex} into a number of lemmas organized below by sub-section.  
\begin{enumerate}
\item In order to apply Lemma \ref{lemma:hermitization}, we need to show that $\log$ is uniformly integrable for $\{\nu_{\frac{1}{\sqrt{n}}(X_n+F_n)-zI} \}_{n \geq 1}$.  We prove this statement in sub-section \ref{sec:uniform-integrability}.  The arguments in this section are based on \cite{BC,N,TVuniv}.  We will also require the use of Theorem \ref{theorem:singularvalue} to control the least singular value.  

\item In sub-section \ref{sec:replacement} we prove a replacement lemma using a moment matching argument.  The lemma will allow us to compute the limit of $\nu_{\frac{1}{\sqrt{n}}X_n - zI}$ by comparing the Stieltjes transform of this measure to the corresponding Stieltjes transform in the Gaussian case.  In order to prove this lemma, we will first need to bound the variance of the resolvent (sub-section \ref{sec:variance}) and apply a truncation argument (sub-section \ref{sec:truncation}).  

\item In sub-section \ref{sub:proof:real}, we prove Lemma \ref{lemma:compare}.  We then apply the results of \cite{N} and Lemma \ref{lemma:compare} to prove Theorem \ref{thm:real}. 

\item In sub-section \ref{sub:proof:complex}, we prove Theorem \ref{thm:complex}.
\end{enumerate}

\subsection{Uniform Integrability} \label{sec:uniform-integrability}

In this sub-section, we prove the following Lemma.  

\begin{lemma} \label{lemma:uniform-integrability}
Let $0 \leq \mu \leq 1$ and $-1 < \rho < 1$ be given.  Let $\{X_n\}_{n \geq 1}$ be a sequence of random matrices that satisfies condition {\bf C0} with atom variables $(\xi_1, \xi_2)$ from the $(\mu,\rho)$-family.  Assume $\{F_n\}_{n \geq 1}$ is a sequence of deterministic matrices such that $\rank(F_n)=o(n)$ and $\sup_{n} \frac{1}{n^2} \|F_n\|^2_2 < \infty$.  Then for a.a. $z \in \mathbb{C}$ a.s. $\log$ is uniformly integrable for $\{\nu_{\frac{1}{\sqrt{n}}(X_n+F_n) - zI} \}_{n \geq 1}$.  
\end{lemma}

The proof of Lemma \ref{lemma:uniform-integrability} is based on the arguments of \cite{BC,N,TVuniv}.  In order to prove Lemma \ref{lemma:uniform-integrability}, we will need the following bound for small singular values.  

\begin{lemma} \label{lemma:small}
There exists $c_0>0$ and $0 < \gamma < 1$ such that the following holds.  Let $\{X_n\}_{n \geq 1}$ be a sequence of random matrices that satisfies condition {\bf C0}.  Then a.s. for $n \gg 1$ and for all $n^{1-\gamma} \leq i \leq n-1$ and all deterministic $n \times n$ matrices $M$,
$$ \sigma_{n-i}(n^{-1/2} X_n + M) \geq c_0 \frac{i}{n}. $$
\end{lemma}

\begin{proof}
Let $\sigma_1 \geq \sigma_2 \geq \cdots \geq \sigma_n$ denote the singular values of $A = \frac{1}{\sqrt{n}}X_n +M$.  It suffices to prove the lemma for $2 n^{1-\gamma} \leq i \leq n-1$ for some $0 < \gamma < 1$ to be chosen later.  Let $A'$ be the matrix formed from the first $m = \lceil n-i/2 \rceil $ rows of $\sqrt{n}A$.  Let $\sigma_1' \geq \cdots \geq \sigma_m'$ denote the singular values of $A'$.  From eigenvalue interlacing it follows that
$$ \frac{1}{\sqrt{n}} \sigma_{n-i}' \leq \sigma_{n-i}. $$
By \cite[Lemma A.4]{TVuniv}, 
$$ \sigma_1'^{-2} + \cdots + \sigma_m'^{-2} =  \dist_1^{-2} + \cdots + \dist_m^{-2} $$
where $\dist_i = \dist(r_i, H_i)$, $r_i$ is the $i$-th row of $A'$, and 
$$ H_i = \mathrm{Span}\{r_j : j = 1, \ldots m; j \neq i \}. $$
Since
$$ \sigma_{n-i}^{-2} \leq n \sigma_{n-i}'^{-2} $$
it follows that
\begin{equation} \label{eq:sni}
	\frac{i}{2n} \sigma_{n-i}^{-2} \leq \frac{i}{2} \sigma_{n-i}'^{-2} \leq \sum_{j=n-i}^m \sigma_{n-j}'^{-2} \leq \sum_{j=1}^m \dist_j^{-2}. 
\end{equation}

We now wish to estimate $\dist(r_j,H_j)$.  However, $r_j$ and $H_j$ are not independent.  To work around this problem, we define the matrix $A'_j$ to be the matrix $A'$ with the $j$-th column removed.  Let $Y_j$ be the $j$-th row of $A'_j$ and let 
$$ H'_j = \mathrm{Span}\{\mathbf{r}_k(A'_j) : k = 1, \ldots, m; k \neq j \}. $$
Note that $Y_j$ and $H'_j$ are independent for each $j=1,\ldots, m$.  

We also have
$$ \dist(r_j, H_j) = \inf_{v \in H_j} \| r_j - v \| \geq \inf_{v \in H'_j} \| Y_j - v\| = \dist(Y_j, H'_j) $$
where
$$ \dim(H'_j) \leq \dim(H_j) \leq n-1 - \frac{i}{2} \leq n-1 - (n-1)^{1-\gamma}.  $$

By Lemma \ref{lemma:dist} below and the union bound, we obtain
$$ \sum_{n=1}^\infty \P \left( \bigcup_{i=2n^{1-\gamma}}^n \bigcup_{j=1}^m \left\{ \dist_j \leq c_o \sqrt{i} \right\} \right) < \infty. $$
Thus, by the Borel-Cantelli lemma, for all $2n^{1-\gamma} \leq i \leq n-1$ and all $1 \leq j \leq m$ 
$$ \dist_j \geq c_0 \sqrt{i} \text{ a.s.} $$
The proof of Lemma \ref{lemma:small} is then complete by the above estimate and \eqref{eq:sni}.  
\end{proof}

\begin{lemma}[Distance of a random vector to a subspace] \label{lemma:dist}
Let $x$ and $y$ be complex-valued random variables with unit variance.  Then there exists $\gamma>0$ and $\varepsilon>0$ such that the following holds.  Let $(\xi_1, \xi_2, \ldots, \xi_n)$ be a random vector in $\mathbb{C}^n$ with independent entries.  Assume further that for each $1 \leq i \leq n$, $\xi_i$ is equal in distribution to either $x$ or $y$.  Then for all $n \gg 1$, any deterministic vector $v \in \mathbb{C}^n$ and any subspace $H$ of $\mathbb{C}^n$ with $1 \leq \dim(H) \leq n-n^{1-\gamma}$, we have 
$$ \P \left( \dist(R,H) \leq \frac{1}{2} \sqrt{n - \dim(H)} \right) \leq \exp(-n^\varepsilon) $$
where $R = (\xi_1, \xi_2, \ldots, \xi_n) + v$.  
\end{lemma}
\begin{proof}
Let $H'$ be the subspace spanned by $H$, $v$, and $\E[R]$.  Then $\dim(H') \leq \dim(H) + 2$ and $\dist(R,H) \geq \dist(R,H') = \dist(R',H')$ where $R' = R-\E[R]$.  Thus it suffices to prove the lemma when $v=0$ and $\E[x] = \E[y] = 0$.  

We now perform a truncation.  By Chebyshev's inequality, 
$$ \P(|\xi_i| > n^\varepsilon) \leq n^{-2\epsilon}. $$
Furthermore, by Hoeffding's inequality
$$ \P \left( \sum_{i=1}^n \indicator{|\xi_i|\leq n^\epsilon} < n - n^{1-\varepsilon} \right) \leq \exp(-n^{1-2\varepsilon}) $$
where we take $\varepsilon \in (0,1/3)$.  Therefore we will prove the lemma by conditioning on the event
$$ \Omega_m = \{|\xi_1| \leq n^\varepsilon, \ldots, |\xi_m| \leq n^\varepsilon \} $$
with $m = \lceil n-n^{1-\varepsilon} \rceil $.   

We now deal with the fact that on the event $\Omega_m$, the random vector $(\xi_1, \ldots, \xi_m)$ may have non-zero mean.  Let $\E_m$ denote the conditional expectation with respect to the event $\Omega_m$ and the $\sigma$-algebra $\mathcal{F}_m = \sigma(\xi_{m+1}, \ldots, \xi_n)$. Let $W$ be the subspace spanned by $H$, $u$, and $w$ where
$$ u = (0, \ldots, 0, \xi_{m+1}, \ldots, \xi_n), \qquad w = (\E_m[\xi_1], \ldots, \E_m[\xi_m], 0, \ldots, 0). $$
Clearly $W$ is $\mathcal{F}_m$-measurable.  Moreover, $\dim(W) \leq \dim (H) + 2$.  Define 
$$ Y = (\xi_1 - \E_m[\xi_1], \ldots, \xi_m -\E_m[\xi_m], 0, \ldots, 0) = R-u-w. $$
Then $\dist(R,H) \geq \dist(R,W) = \dist(Y,W)$.  By construction each entry of $Y$ has mean zero.  Since each entry of the original vector $R$ is equal in distribution to either $x$ or $y$, it follows that
$$ \sup_{1 \leq i \leq m} |\sigma_i^2 - 1| = o(1) $$
where $\sigma_i^2 = \E_m|Y_i|^2$.

By Talagrand's concentration inequality \cite{T}, 
\begin{equation} \label{eq:pmd}
	\P_m (|\dist(Y,W) - M_m| \geq t) \leq 4 \exp\left( -\frac{t^2}{16 n^{2\varepsilon}} \right) 
\end{equation}
where $M_m$ is the median of $\dist(Y,W)$ under $\Omega_m$.  Using \eqref{eq:pmd} one can verify that 
$$ M_m \geq \sqrt{\E_m \dist^2(Y,W) } - C n^{4\varepsilon} $$
for some positive constant $C$ (see for instance \cite[Lemma E.3]{TVhard}).  
Let $P$ denote the orthogonal projection onto $W^\perp$.  Then
\begin{align*}
	\E_m \dist^2(Y,W) = \sum_{k=1}^m \E_m[Y_k^2] P_{kk} &\geq c \left( \sum_{k=1}^n P_{kk} - \sum_{k=m+1}^n P_{kk} \right) \\
		&\geq c (n - \dim(H) - (n-m))
\end{align*}
for any $ 1/2 < c < 1$ and $n \gg_c 1$.  Thus
$$ M_m \geq c \sqrt{n - \dim(H)}$$
for $n$ sufficiently large.  Finally, we choose $0 < \gamma < \varepsilon/2$ and the proof of the lemma is complete by taking $t = (c-1/2) \sqrt{n - \dim(H)}$ in \eqref{eq:pmd}.  
\end{proof}

We now prove Lemma \ref{lemma:uniform-integrability}.
\begin{proof}[Proof of Lemma \ref{lemma:uniform-integrability}]
By Markov's inequality, it suffices to show that there exists $p>0$ such that for a.a. $z \in \mathbb{C}$ a.s.
$$ \limsup_{n \rightarrow \infty} \int s^{-p} d\nu_{\frac{1}{\sqrt{n}} X_n - zI} < \infty \qquad \text{ and }\qquad \limsup_{n \rightarrow \infty} \int s^{p} d\nu_{\frac{1}{\sqrt{n}} X_n - zI} < \infty. $$
Fix $z \in \mathbb{C}$.  Then
\begin{align*}
	\int s^p d\nu_{\frac{1}{\sqrt{n}} X_n + \frac{1}{\sqrt{n}} F_n - zI} &\leq 1 + \frac{1}{n} \tr \left( \frac{1}{\sqrt{n}} X_n + \frac{1}{\sqrt{n}}F_n -zI \right)^\ast \left( \frac{1}{\sqrt{n}} X_n + \frac{1}{\sqrt{n}} F_n -zI \right)
\end{align*}
for $p \leq 2$.  We expand out the right-hand side and consider three separate terms.  First, by the law of large numbers,
\begin{align*}
	\frac{1}{n} \tr \left( \frac{1}{\sqrt{n}} X_n -zI \right)^\ast \left( \frac{1}{\sqrt{n}} X_n  -zI \right)
	&\leq 1 + \frac{1}{n^2} \sum_{i,j=1}^n |x_{ij}|^2 - 2 \Re \left( \frac{z}{n^{3/2}} \sum_{k=1}^n x_{kk} \right) + |z|^2 \\
	& \longrightarrow 2 + |z|^2
\end{align*}
a.s. as $n \rightarrow \infty$.  Here, we first divide the sums into three parts in order to apply the law of large numbers.  The first when $i < j$, the second when $i > j$, and the third when $i=j$.  In this way the summands in each sum are i.i.d. random variables and the law of large numbers applies.  

Second,
$$ \left| \frac{1}{n} \tr \left( \frac{1}{\sqrt{n}}F_n^\ast \right) \left( \frac{1}{\sqrt{n}} X_n + \frac{1}{\sqrt{n}} F_n - zI \right) \right| \leq 1 + \frac{1+|z|^2}{n^2} \|F_n\|_2^2 + \left| \frac{1}{n^2} \tr (F_n^\ast X_n) \right|. $$
Since $\sup_{n} \frac{1}{n^2} \|F_n\|^2_2 < \infty$ by assumption, it suffices to show that $\limsup_{n \rightarrow \infty} \frac{1}{n^2} \tr (F_n^\ast X_n) < \infty$ a.s.  We apply the bounds
$$ \left| \frac{1}{n^2} \tr (F_n^\ast X_n) \right| \leq \frac{1}{n^2} \|F_n\|_2 \|X_n\|_2 \leq \frac{1}{n^2} \|F_n\|_2^2 + \frac{1}{n^2} \|X_n\|_2^2. $$
By the law of large numbers (again considering three separate terms), it follows that a.s.
$$ \limsup_{n \rightarrow \infty} \frac{1}{n^2} \tr (X_n^\ast X_n ) < \infty. $$

Similarly, for the third term, we have that a.s.
$$ \limsup_{n \rightarrow \infty} \left| \frac{1}{n} \tr \left( \frac{1}{\sqrt{n}} X_n + \frac{1}{\sqrt{n}} F_n - zI \right)^\ast  \left( \frac{1}{\sqrt{n}}F_n \right) \right| < \infty. $$

We simplify our notation for the remainder of the proof and write $\sigma_1 \geq \sigma_2 \geq \cdots \geq \sigma_n$ for the singular values of $\frac{1}{\sqrt{n}}(X_n+F_n) - zI$.  By Theorem \ref{theorem:singularvalue}, we have that for some $A>0$, 
$$ \sigma_n > n^{-A} \text{ a.s.} $$
Thus,
\begin{align*}
	\frac{1}{n}\sum_{i=1}^n \sigma_i^{-p} &\leq \frac{1}{n} \sum_{i=1}^{n-n^{1-\gamma}} \sigma_i^{-p} + \frac{1}{n} \sum_{i=n-n^{1-\gamma}}^n \sigma_i^{-p} \\
	& \leq \frac{1}{n} \sum_{i=n^{1-\gamma}}^{n-1} \sigma_{n-i}^{-p} + \frac{1}{n} n^{1-\gamma}n^{Ap} \\
	& \leq \frac{1}{c_0^p} \left[ \frac{1}{n} \sum_{i=1}^n \left( \frac{n}{i} \right)^p \right] + n^{Ap-\gamma}
\end{align*}
a.s. by Lemma \ref{lemma:small}.  The remaining sum is just the Riemann sum of the integral $\int_0^1 u^{-p}du$.  Therefore, we have that
$$ \frac{1}{n}\sum_{i=1}^n \sigma_i^{-p} < \infty \text{ a.s.} $$
for $p < \min\{1,\gamma/A\}$.  
\end{proof}

\subsection{Variance Bound} \label{sec:variance}

In this sub-section, we prove the following lemma.

\begin{lemma} \label{lemma:variance}
There exists a positive constant $C$ such that the following holds.  Let $\{X_n\}_{n \geq 1}$ be a sequence of random matrices that satisfies condition {\bf C0} with atom variables $(\xi_1, \xi_2)$.  Define
$$ R_n := \left( \frac{1}{\sqrt{n}} X_n - zI \right), \qquad H_n(\alpha) := \left( R_n^\ast R_n - \alpha I\right)^{-1}, $$
where $\alpha \in \mathbb{C}$ with $\Im(\alpha) \neq 0$.  Then
\begin{equation} \label{eq:mom4}
	\E \left| \frac{1}{n} \tr H_n(\alpha) - \E\left[ \frac{1}{n} \tr H_n(\alpha) \right] \right|^4 \leq C\frac{c^4_\alpha }{n^2 }
\end{equation}
uniformly for $z \in \mathbb{C}$ where 
$$ c_\alpha = \frac{1}{|\Im(\alpha)|} + \frac{|\alpha|}{|\Im(\alpha)|^2}. $$
Moreover, for every fixed $\alpha$, 
\begin{equation} \label{eq:asconv}
	\frac{1}{n} \tr H_n(\alpha) = \E \left[ \frac{1}{n} \tr H_n(\alpha) \right] + O(n^{-1/8}) \text{ a.s.}
\end{equation}
uniformly for $z \in \mathbb{C}$.
\end{lemma}

\begin{proof}
Let $\E_{\leq k}$ denote conditional expectation with respect to the $\sigma$-algebra generated by $\mathbf{r}_1(X_n), \ldots, \mathbf{r}_k(X_n), \mathbf{c}_1(X_n), \ldots, \mathbf{c}_k(X_n)$.  Define
$$ Y_k := \E_{\leq k} \frac{1}{n} \tr H_n(\alpha) $$
for $k=0,1,\ldots, n$.  Clearly $\{Y_k\}_{k=0}^n$ is a martingale.  Define the martingale difference sequence
$$ \alpha_k := Y_k - Y_{k-1} $$
for $k=1,2,\ldots, n$.  Then by construction
$$ \sum_{k=1}^n \alpha_k = \frac{1}{n} \tr H_n(\alpha) - \E \frac{1}{n} \tr H_n(\alpha). $$
We will bound the fourth moment of the sum, but first we obtain a bound on the individual summands.  Let $X_{n,k}$ denote the matrix $X_n$ with the $k$-th row and $k$-th column replaced by zeros.  Let
$$ R_{n,k} := \frac{1}{\sqrt{n}} X_{n,k} - zI, \qquad H_{n,k}(\alpha) := \left( R_{n,k}^\ast R_{n,k} - \alpha I \right)^{-1}. $$
It follows that
$$ \E_{\leq k} \frac{1}{n} \tr H_{n,k}(\alpha) = \E_{\leq k-1} \frac{1}{n} \tr H_{n,k}(\alpha) $$
and hence
$$ \alpha_k = \E_{\leq k} \left[ \frac{1}{n} \tr H_n(\alpha) - \frac{1}{n} \tr H_{n,k}(\alpha) \right] - \E_{\leq k-1} \left[ \frac{1}{n} \tr H_n(\alpha) - \frac{1}{n} \tr H_{n,k}(\alpha) \right]. $$

By the resolvent identity,
$$ |\tr H_n(\alpha) - \tr H_{n,k}(\alpha)| = | \tr [ H_n (R_n^\ast R_n - R_{n,k}^\ast R_{n,k}) H_{n,k} |. $$
Since $R_n^\ast R_n - R_{n,k}^\ast R_{n,k}$ is at most rank $4$, it follows that
$$ |\tr H_n(\alpha) - \tr H_{n,k}(\alpha)| \leq 4 \| H_n (R_n^\ast R_n - R_{n,k}^\ast R_{n,k}) H_{n,k} \|. $$
We then note that
\begin{align*}
	\|H_n(\alpha) R_n^\ast R_n \| &\leq \sup_{t \geq 0} \frac{t}{|t-\alpha|}  \leq 1 + \sup_{t \geq 0} \frac{|\alpha|}{|t-\alpha|}  \leq 1 + \frac{|\alpha|}{|\Im(\alpha)|}
\end{align*}
since the eigenvalues of $R_n^\ast R_n$ are non-negative.  Similarly,
$$ \|H_{n,k}(\alpha) R_{n,k}^\ast R_{n,k} \| \leq 1 + \frac{|\alpha|}{|\Im(\alpha)|}. $$
Since we always have the bound $\|H_n(\alpha)\| \leq |\Im(\alpha)|^{-1}$, it follows that
$$ |\tr H_n(\alpha) - \tr H_{n,k}(\alpha)| \leq 8 c_\alpha. $$  
Thus we conclude that
$$ |\alpha_k| \leq \frac{16 c_\alpha}{n}. $$

By the Burkholder inquality (see \cite[Lemma 2.12]{BS} for a complex martingale version of the Burkholder inequality), there exists an absolute constant $C>0$ such that
\begin{align*}
	\E\left| \sum_{k=1}^n \alpha_k \right|^4 & \leq C \E \left( \sum_{k=1}^n |\alpha_k|^2 \right)^2  \leq C \left( n \frac{ 16^2 c_\alpha^2 }{n^2} \right)^2 \leq 16^4 C \frac{c_\alpha^4}{n^2}.
\end{align*}
The proof of \eqref{eq:mom4} is complete.

To prove \eqref{eq:asconv}, we use Markov's inequality and \eqref{eq:mom4} to obtain
$$ \P \left( \left| \frac{1}{n} \tr H_n(\alpha) - \E\frac{1}{n} \tr H_n(\alpha) \right| > \varepsilon \right) \leq C \frac{c_\alpha^4}{n^2 \varepsilon^4}. $$
The result follows by taking $\varepsilon = n^{-1/8}$ and applying the Borel-Cantelli Lemma.
\end{proof}

\subsection{Truncation} \label{sec:truncation}

Given a sequence of random matrices $\{X_n\}_{n\geq 1}$ that satisfies condition {\bf C0}, we define the sequences $\{\hat{X}_n\}_{n \geq 1}$ and $\{\tilde{X}_n\}_{n \geq 1}$ where for each $n \geq 1$, $\hat{X}_n = (\hat{x}_{ij})_{1 \leq i,j \leq n}$ and $\tilde{X}_n = (\tilde{x}_{ij})_{1 \leq i,j \leq n}$ with
$$ \hat{x}_{ij} = \left\{
     		\begin{array}{lr}
     		  x_{ij} \indicator{|x_ij| \leq n^\delta} - \E[x_{ij} \indicator{|x_ij| \leq n^\delta}], &  i \neq j\\
     		  0, &  i = j
     		\end{array}
   		\right. $$
and
$$ \tilde{x}_{ij} = \left\{
     		\begin{array}{lr}
     		  \frac{\hat{x}_{ij}}{\sqrt{\E|\hat{x}_{ij}|^2}}, &  i \neq j\\
     		  0, &  i = j
     		\end{array}
   		\right. $$
for some $\delta > 0$, which we will choose later.  For each $n \geq 1$, define the matrices
$$ \hat{H}_n = \left( \frac{1}{\sqrt{n}} \hat{X}_n - zI \right)^\ast \left( \frac{1}{\sqrt{n}} \hat{X}_n - zI \right) $$
and
$$\tilde{H}_n = \left( \frac{1}{\sqrt{n}} \tilde{X}_n - zI \right)^\ast \left( \frac{1}{\sqrt{n}} \tilde{X}_n - zI \right). $$

We let $L(\mu, \nu)$ denote the Levy distance between the probability measures $\mu$ and $\nu$.  We prove the following truncation lemma.

\begin{lemma} \label{lemma:truncation}
Let $\{X_n\}_{n \geq 1}$ be a sequence of random matrices that satisfies condition {\bf C0}.  Then uniformly for any $|z| \leq M$, we have that
$$ L(\nu_{H_n}, \nu_{\tilde{H}_n}) = o(1) \text{ a.s. }$$
Moreover, 
\begin{equation} \label{eq:trunc:mom}
	\E[\Re(\tilde{x}_{ij})^k\Im(\tilde{x}_{ij})^l\Re(\tilde{x}_{ji})^m\Im(\tilde{x}_{ji})^p] = \E[\Re({x}_{ij})^k\Im({x}_{ij})^l\Re({x}_{ji})^m\Im({x}_{ji})^p] +o(1) 
\end{equation}
uniformly for $i \neq j$ and all non-negative integers $k,l,m,p$ such that $k+l+m+p \leq 2$.  
\end{lemma}
\begin{proof}
By \cite[Corollary A.42]{BS}, 
\begin{equation} \label{eq:l4h}
	L^4 (\nu_{H_n}, \nu_{\hat{H}_n}) \leq \frac{2}{n^3} \left[ \tr (H_n + \hat{H}_n) \tr \left( (X_n - \hat{X}_n)^\ast (X_n - \hat{X}_n) \right) \right]. 
\end{equation}

By the law of large numbers,
\begin{align*}
	\frac{1}{n} \tr H_n &= \frac{1}{n^2} \sum_{i,j=1}^n |x_{ij}|^2 - 2 \Re \left( \frac{z}{n^{3/2}} \sum_{k=1}^n x_{kk} \right) + |z|^2 \\
	&\qquad \longrightarrow 1 + |z|^2
\end{align*}
a.s. as $n \rightarrow \infty$.  Here, we first divide the sums into three parts in order to apply the law of large numbers.  The first when $i < j$, the second when $i > j$, and the third when $i=j$.  In this way the summands in each sum are i.i.d. and the law of large numbers applies.  

Similarly,
$$ \frac{1}{n} \tr \hat{H}_n \longrightarrow 1+|z|^2 $$
a.s. as $n \rightarrow \infty$.  

For the remaining terms, we note that
\begin{align*}
	\frac{1}{n^2} \tr \left( (X_n -\hat{X}_n)^\ast (X_n -\hat{X}_n) \right) &\leq \frac{2}{n^2} \sum_{1 \leq i < j \leq n} \left[|x_{ij}|^2 \indicator{|x_{ij} > n^\delta} + \E|x_{ij}|^2 \indicator{|x_{ij}| > n^\delta} \right] \\
	& \qquad + \frac{2}{n^2} \sum_{1 \leq j < i \leq n} \left[|x_{ij}|^2 \indicator{|x_{ij} > n^\delta} + \E|x_{ij}|^2 \indicator{|x_{ij}| > n^\delta} \right] \\
	& \qquad + \frac{1}{n^2} \sum_{i=1}^n  |x_{ii}|^2.
\end{align*}
By the law of large numbers, each sum on the right-hand side converges to zero a.s. as $n \rightarrow \infty$.  Combining these estimates into \eqref{eq:l4h}, yields
\begin{equation} \label{eq:trunc:lhn}
	L(\nu_{H_n}, \nu_{\tilde{H}_n}) = o(1) 
\end{equation}
a.s. as $n \rightarrow \infty$.

Again using \cite[Corollary A.42]{BS},
$$ L^4(\nu_{\hat{H}_n}, \nu_{\tilde{H}_n}) \leq \frac{2}{n^3} \tr(\hat{H}_n + \tilde{H}_n) \tr \left((\hat{X}_n - \tilde{X}_n)^\ast (\hat{X}_n - \tilde{X}_n) \right). $$
It then follows that 
\begin{equation} \label{eq:trunc:lhh}
	L(\nu_{\hat{H}_n}, \nu_{\tilde{H}_n}) = o(1) 
\end{equation}
a.s. since
$$ 1 - \sqrt{\E|\hat{x}_{ij}|^2} = o(1) $$
uniformly for all $i \neq j$ by the identical distribution assumption of condition {\bf C0}.  The result then follows from estimates \eqref{eq:trunc:lhn} and \eqref{eq:trunc:lhh}.  

\eqref{eq:trunc:mom} can be obtained from the dominated convergence theorem; the identical distribution portion of condition {\bf C0} gives uniform control for all $i \neq j$.  
\end{proof}

\subsection{Replacement} \label{sec:replacement}

In 1922, Lindeberg \cite{L} gave an elegant proof of the central limit theorem using a replacement method.  Recently, this technique has also been applied to study random matrices with exchangeable or independent entries (see, for example, \cite{Ch,TVlocal} and references therein).  In this sub-section, we prove a comparison lemma for sequences of random matrices that satisfy condition {\bf C0} using this method.  We begin with a definition.  

\begin{definition}[Moment matching] \label{def:moment-matching}
Let $(\xi_1, \xi_2)$ and $(\eta_1, \eta_2)$ be two random vectors in $\mathbb{C}^2$.  We say that $(\xi_1, \xi_2)$ and $(\eta_1, \eta_2)$ match to order $k$ if
$$ \E[\Re(\xi_1)^i\Im(\xi_1)^j \Re(\xi_2)^l\Im(\xi_2)^m] = \E[\Re(\eta_1)^i\Im(\eta_1)^j \Re(\eta_2)^l\Im(\eta_2)^m] $$
for all non-negative integers $i,j,l,m$ with $i + j + l + m \leq k$.  
\end{definition}

The goal of this sub-section is to prove the following lemma.  

\begin{lemma} \label{lemma:replacement}
Let $\{X_n\}_{n \geq 1}$ and $\{Y_n\}_{n \geq 1}$ be sequences of random matrices that satisfy condition {\bf C0} with with atom variables $(\xi_1, \xi_2)$ and $(\eta_1, \eta_2)$, respectively.  Assume the moments of $(\xi_1, \xi_2)$ and $(\eta_1, \eta_2)$ match to order $2$.  Then for a.a. $z \in \mathbb{C}$ a.s.
$$ \nu_{\frac{1}{\sqrt{n}} X_n -zI} - \nu_{\frac{1}{\sqrt{n}} Y_n - zI} \longrightarrow 0 $$
as $n \rightarrow \infty$.  
\end{lemma}

We proceed using the Stieltjes transform.  For a $n \times n$ matrix $A$, we define the matrices
$$ R_n(A) = \frac{1}{\sqrt{n}}A-zI, \qquad G_n(A) = \left( R_n(A)^\ast R_n(A) - \alpha I \right)^{-1} $$
where $z, \alpha \in \mathbb{C}$ with $\Im(\alpha) > 0$.  Using the resolvent identity, we can compute
\begin{equation} \label{eq:deriv1}
	\frac{ \partial (G_n(A))_{ij}}{\partial \Re(A_{st})} = -\frac{1}{\sqrt{n}} \left[ (G_n(A) R_n(A)^\ast)_{is} (G_n(A))_{tj} + (G_n(A))_{it} (R_n(A) G_n(A))_{sj} \right] 
\end{equation}
and
\begin{equation} \label{eq:deriv2}
	\frac{ \partial (G_n(A))_{ij}}{\partial \Im(A_{st})} = -\frac{\sqrt{-1}}{\sqrt{n}} \left[ (G_n(A)R_n(A)^\ast)_{is} (G_n(A))_{tj} - (G_n(A))_{it} (R_n(A) G_n(A))_{sj} \right]. 
\end{equation}

Fix the indices $a \neq b$.  Let $V_1 = e_a e_b^\ast$ and $V_2 = e_b e_a^\ast$ where $e_1, \ldots, e_n$ is the standard basis in $\mathbb{C}^n$.  Let $x_1,x_2,x_3,x_4$ be real variables.  We define the function
$$ f(x_1,x_2,x_3,x_4) = \frac{1}{n} \tr G_n(A + x_1 V_1 + \sqrt{-1} x_2 V_1 + x_3 V_2 + \sqrt{-1} x_4 V_2). $$
Using the derivatives above, we write out the power series
\begin{equation} \label{eq:fxx}
	f(x_1,x_2,x_3,x_4) = f(0,0,0,0) + \sum_{k=1}^4 \frac{\partial f}{\partial x_k}(0,0,0,0) x_k + \sum_{i,j=1}^4 \frac{\partial^2 f}{\partial x_i \partial x_j}(0,0,0,0) x_i x_j + \varepsilon 
\end{equation}
where $|\varepsilon| \leq CM (|x_1|^3 + |x_2|^3 + |x_3|^3 + |x_4|^4)$ with $M$ defined by
$$ M = \sup_{1 \leq i,j,k \leq 4} \sup_{x_1, x_2, x_3, x_4} \left| \frac{\partial^3 f}{\partial x_i \partial x_j \partial x_k}(x_1, x_2, x_3, x_4) \right|. $$

We now obtain a bound for $M$ and the partial derivatives of $f$.  Note that the bounds we derive below hold uniformly for any matrix $A$.  We can write $R_n(A) = U \sqrt{R_n(A)^\ast R_n(A)}$ where $U$ is a partial isometry.  So
\begin{align*}
	\|R_n(A) G_n(A) \| &\leq \|U \sqrt{R_n(A) R_n(A)} (R_n^\ast R_n(A) - \alpha I)^{-1} \| \\
		& \leq \|\sqrt{R_n(A) R_n(A)} (R_n^\ast R_n(A) - \alpha I)^{-1} \| \\
		& \leq \sup_{t \geq 0} \left| \frac{\sqrt{t}}{t - \alpha} \right| \\
		& \leq \frac{1}{|\Im(\alpha)|} + \sup_{t \geq 0} \left| \frac{t}{t - \alpha} \right| \\
		& \leq 1 + \frac{ |\alpha|+1}{|\Im(\alpha)|}
\end{align*}
and similarly
$$ \| G_n(A) R_n(A)^\ast \| \leq 1 + \frac{ |\alpha|+1}{|\Im(\alpha)|}. $$

Thus, by \eqref{eq:deriv1}, \eqref{eq:deriv2}, and the bounds above, it follows that
\begin{equation} \label{eq:deriv_bounds}
	\frac{ \partial f}{\partial x_k} = O_{\alpha} \left( \frac{1}{n} \right), \qquad 	\frac{ \partial^2 f}{\partial x_k \partial x_i} = O_{\alpha} \left( \frac{1}{n} \right), \qquad \frac{ \partial^3 f}{\partial x_k \partial x_i \partial x_j} = O_{\alpha} \left( \frac{1}{n} \right) 
\end{equation}
uniformly for $1 \leq i,j,k \leq 4$, any $x_1, x_2, x_3, x_4 \in \mathbb{R}$, and any $A$.  

We are now ready to prove Lemma \ref{lemma:replacement}.  By Lemma \ref{lemma:truncation} and Remark \ref{remark:sing}, it suffices to show that a.s.
\begin{equation} \label{eq:rnxn}
	\nu_{R_n(\tilde{X}_n)^\ast R_n(\tilde{X}_n)} - \nu_{R_n(\tilde{Y}_n)^\ast R_n(\tilde{Y}_n)} \longrightarrow 0. 
\end{equation}
as $n \rightarrow \infty$.  So, without loss of generality, we assume $\xi_1,\xi_2, \eta_1,\eta_2$ have mean zero, unit variance, and are bounded almost surely in magnitude by $n^\delta$ for some $0 < \delta < \frac{1}{2}$.    

By \cite[Theorem B.9]{BS}, we can equivalently state \eqref{eq:rnxn} as
$$ \frac{1}{n} \tr G_n(\tilde{X}_n) - \frac{1}{n} \tr G_n(\tilde{Y}_n) \longrightarrow 0 $$
a.s. for each fixed $\alpha$ with $\Im(\alpha)>0$.  However by Lemma \ref{lemma:variance}, this reduces to showing that
\begin{equation} \label{eq:reducegxn}
	\E \frac{1}{n} \tr G_n(\tilde{X}_n) - \E\frac{1}{n} \tr G_n(\tilde{Y}_n) \longrightarrow 0. 
\end{equation}

We will verify \eqref{eq:reducegxn} by showing that for each fixed $\alpha$ with $\Im(\alpha)>0$,
\begin{equation} \label{eq:swap}
	\E f\left(\frac{\Re(\xi_1)}{\sqrt{n}}, \frac{\Im(\xi_1)}{\sqrt{n}}, \frac{\Re(\xi_2)}{\sqrt{n}}, \frac{\Im(\xi_2)}{\sqrt{n}} \right) = \E f\left( \frac{\Re(\eta_1)}{\sqrt{n}}, \frac{\Im(\eta_1)}{\sqrt{n}}, \frac{\Re(\eta_2)}{\sqrt{n}}, \frac{\Im(\eta_2)}{\sqrt{n}} \right) + o_{\alpha}(n^{-2}) 
\end{equation}
where we take $A$ to be any matrix independent of $(\xi_1, \xi_2)$ and $(\eta_1,\eta_2)$.  Indeed, by allowing $a$ and $b$ to range over all $O(n^2)$ indices, and by the triangle inequality, we obtain 
$$ \E \frac{1}{n} \tr G_n(\tilde{X}_n) = \E \frac{1}{n} \tr G_n(\tilde{Y}_n) + o_{\alpha}(1) $$
as desired.  

It suffices to verify \eqref{eq:swap} for the off-diagonal entries ($a \neq b$).  Indeed, all diagonal entries are assumed to be zero by our previous application of Lemma \ref{lemma:truncation}.  

Using \eqref{eq:fxx}, \eqref{eq:deriv_bounds}, and the independence assumption from condition {\bf C0}, we obtain 
$$ \E[f(x_1, x_2, x_2, x_4)] = \E[f(0,0,0,0)] + \sum_{i,j=1}^4 \E\frac{\partial^2 f}{\partial x_i \partial x_j}(0,0,0,0) \E[x_i x_j] + \E[\varepsilon] $$
where 
$$ x_1 = \frac{\Re(\xi_1)}{\sqrt{n}}, \qquad x_2 = \frac{\Im(\xi_1)}{\sqrt{n}}, \qquad x_3 = \frac{\Re(\xi_2)}{\sqrt{n}}, \qquad x_4 = \frac{\Im(\xi_2)}{\sqrt{n}}, \qquad $$
and
$$ \E|\varepsilon| = O_\alpha \left( \frac{n^\delta}{n^{2.5}} \right) $$
for some $0 < \delta < 1/2$ from Lemma \ref{lemma:truncation}.  

We repeat the same procedure for $(\eta_1,\eta_2)$ and obtain
$$ \E[f(y_1, y_2, y_2, y_4)] = \E[f(0,0,0,0)] + \sum_{i,j=1}^4 \E\frac{\partial^2 f}{\partial y_i \partial y_j}(0,0,0,0) \E[y_i y_j] + O_\alpha \left( \frac{n^\delta}{n^{2.5}} \right) $$
where
$$ y_1 = \frac{\Re(\eta_1)}{\sqrt{n}}, \qquad y_2 = \frac{\Im(\eta_1)}{\sqrt{n}}, \qquad y_3 = \frac{\Re(\eta_2)}{\sqrt{n}}, \qquad y_4 = \frac{\Im(\eta_2)}{\sqrt{n}}. $$

By \eqref{eq:trunc:mom}, we have that $\E[x_i x_j] = \E[y_i y_j] + o(n^{-1})$ uniformly for $1 \leq i,j \leq 4$.  Combining this with \eqref{eq:deriv_bounds} yields
$$ \E[f(y_1, y_2, y_2, y_4)] = \E[f(x_1, x_2, x_2, x_4)] + o_{\alpha}(n^{-2}) $$
and the proof of Lemma \ref{lemma:replacement} is complete.

\subsection{Proof of Lemma \ref{lemma:compare} and Theorem \ref{thm:real}} \label{sub:proof:real}

This sub-section is devoted to Lemma \ref{lemma:compare} and Theorem \ref{thm:real}.  The proof of Lemma \ref{lemma:compare} relies on Lemmas \ref{lemma:replacement}, \ref{lemma:uniform-integrability}, and \ref{lemma:hermitization}.

\begin{proof}[Proof of Lemma \ref{lemma:compare}]
Assume $\mu, \rho, \{X_n\}_{n \geq 1}, \{Y_n\}_{n \geq 1}, \{F_n\}_{n \geq 1}$ satisfy the assumptions in the statement of Lemma \ref{lemma:compare}.  By \cite[Theorem A.44]{BS}, it follows that for a.a. $z \in \mathbb{C}$ a.s.
$$ \nu_{\frac{1}{\sqrt{n}}(X_n + F_n)-zI} - \nu_{\frac{1}{\sqrt{n}}X_n-zI} \longrightarrow 0 $$
as $n \rightarrow \infty$.  Since both $(\xi_1,\xi_2)$ and $(\eta_1, \eta_2)$ are from the $(\mu,\rho)$-family, then $(\xi_1,\xi_2)$ and $(\eta_1,\eta_2)$ match to order $2$.  Thus by Lemma \ref{lemma:replacement}, for a.a. $z \in \mathbb{C}$ a.s.
$$ \nu_{\frac{1}{\sqrt{n}}X_n -zI} \longrightarrow \nu_z $$
as $n \rightarrow \infty$.  Therefore, for a.a. $z \in \mathbb{C}$ a.s.
$$ \nu_{\frac{1}{\sqrt{n}}(X_n + F_n)-zI} \longrightarrow \nu_z $$
as $n \rightarrow \infty$.  

Furthermore, by Lemma \ref{lemma:uniform-integrability}, for a.a. $z \in \mathbb{C}$ a.s. $\log$ is uniformly integrable for $\{ \nu_{\frac{1}{\sqrt{n}}(X_n+F_n) -zI} \}_{n \geq 1}$, $\{ \nu_{\frac{1}{\sqrt{n}}X_n -zI} \}_{n \geq 1}$, and $\{ \nu_{\frac{1}{\sqrt{n}}Y_n -zI} \}_{n \geq 1}$.  The result then follows by Lemma \ref{lemma:hermitization} and the uniqueness of the logarithmic potential \cite[Lemma 4.1]{BC}.  
\end{proof}

We can now prove Theorem \ref{thm:real}.   

\begin{proof}[Proof of Theorem \ref{thm:real}]
Let $\{X_n\}_{n \geq 1}$ be a sequence of real random matrices that satisfies condition {\bf C0} with atom variables $(\xi_1, \xi_2)$ and $\rho = \E[\xi_1 \xi_2]$ for some $-1 < \rho < 1$.  Let $\{Y_n\}_{n \geq 1}$ be the sequence of random matrices that satisfies condition {\bf C0} with atom variables $(\eta_1, \eta_2)$ where $\eta_1$ and $\eta_2$ are jointly Gaussian and $\E[\eta_1 \eta_2] = \rho$.  In \cite[Theorem 5.2]{N}, it is shown that 
$$ \E\nu_{\frac{1}{\sqrt{n}}Y_n - zI} \longrightarrow \nu_z $$
as $n \rightarrow \infty$ where the family $\{\nu_z\}_{z \in \mathbb{C}}$ determines the elliptic law with parameter $\rho$ by Lemma \ref{lemma:hermitization}.  In fact, using the variance bound in Lemma \ref{lemma:variance} and \cite[Theorem B.9]{BS} it can be shown that a.s.
$$ \nu_{\frac{1}{\sqrt{n}}Y_n - zI} \longrightarrow \nu_z $$
as $n \rightarrow \infty$.  By Lemma \ref{lemma:uniform-integrability}, for a.a. $z \in \mathbb{C}$ a.s. $\log$ is uniformly integrable for $\{ \nu_{\frac{1}{\sqrt{n}}Y_n - zI} \}_{n \geq 1}$ and hence by Lemma \ref{lemma:hermitization}, we conclude that
$$ \mu_{\frac{1}{\sqrt{n}}Y_n} \longrightarrow \mu_\rho $$
a.s. as $n \rightarrow \infty$.  

Since both $(\xi_1, \xi_2)$ and $(\eta_1, \eta_2)$ are from the $(1,\rho)$-family, the proof of the theorem is complete by an application of Lemma \ref{lemma:compare}.  
\end{proof}

\subsection{Proof of Theorem \ref{thm:complex}} \label{sub:proof:complex}

In the proof of Theorem \ref{thm:real} above, we relied on the previous results in \cite{N}, where the entries are assumed to be real.  In order to prove Theorem \ref{thm:complex}, we first need to study the complex Gaussian case.  

\begin{lemma} \label{lemma:complex}
Let $0 \leq \mu < 1$ and $-1 < \rho < 1$ be given.  Assume $\{X_n\}_{n \geq 1}$ is a sequence of complex matrices that satisfy condition {\bf C0} with atom variables $(\xi_1,\xi_2)$ from the $(\mu,\rho)$-family, where $ \Re(\xi_1), \Im(\xi_1), \Re(\xi_2), \Im(\xi_2) $ are jointly Gaussian.  Then for a.a. $z \in \mathbb{C}$ a.s. 
$$ \nu_{\frac{1}{\sqrt{n}}X_n - zI} \longrightarrow \nu_z $$
as $n \rightarrow \infty$ where $\{\nu_z\}_{z \in \mathbb{C}}$ determines the elliptic law with parameter $\rho$ by Lemma \ref{lemma:hermitization}.  
\end{lemma}

Let us assume Lemma \ref{lemma:complex} for now and complete the proof of Theorem \ref{thm:complex}. 

\begin{proof}[Proof of Theorem \ref{thm:complex}]
Let $\{X_n\}_{n \geq 1}$ be a sequence of complex random matrices that satisfy condition {\bf C0} with atom variables $(\xi_1,\xi_2)$ from the $(\mu,\rho)$-family.  Let $\{Y_n\}_{n \geq 1}$ be the sequence of complex random matrices that satisfy condition {\bf C0} with atom variables $(\eta_1,\eta_2)$ from the $(\mu,\rho)$-family, where $ \Re(\eta_1),\Im(\eta_1), \Re(\eta_2), \Im(\eta_2) $ are jointly Gaussian.  By Lemma \ref{lemma:complex}, for a.a. $z \in \mathbb{C}$ a.s.
$$ \nu_{\frac{1}{\sqrt{n}} Y_n - zI} \longrightarrow v_z $$
as $n \rightarrow \infty$ where the family $\{v_z\}_{z \in \mathbb{C}}$ determines the elliptic law with parameter $\rho$ by Lemma \ref{lemma:hermitization}.  Moreover, $\log$ is uniformly integrable for $\{ \nu_{\frac{1}{\sqrt{n}} Y_n - zI} \}_{n \geq 1}$ by Lemma \ref{lemma:uniform-integrability}.  Therefore, by Lemma \ref{lemma:hermitization}, it follows that a.s.
$$ \mu_{\frac{1}{\sqrt{n}}Y_n} \longrightarrow \mu_{\rho} $$
as $n \rightarrow \infty$.  The proof of Theorem \ref{thm:complex} is now complete by Lemma \ref{lemma:compare}.
\end{proof}

All that remains is to prove Lemma \ref{lemma:complex}.  Let $\{X_n\}_{n \geq 1}$ be the sequence of random matrices defined in Lemma \ref{lemma:complex} with jointly Gaussian off-diagonal entries.  We follow \cite{N} and introduce the following notation.  

For $n \times n$ matrices $A$ and $B$, we define the $2n \times 2n$ bock matrices
$$ V := \begin{bmatrix} \frac{1}{\sqrt{n}} A & 0 \\ 0 & \frac{1}{\sqrt{n}} B^\ast \end{bmatrix}, \qquad J_z := \begin{bmatrix} 0 & zI \\ \bar{z} I & 0 \end{bmatrix} $$
and set
$$ V(z) := VJ_1 - J_z, $$
where
$$ J_1 := \begin{bmatrix} 0 & I \\  I & 0 \end{bmatrix}.  $$
We let $R$ denote the resolvent of $V(z)$.  That is,
$$ R := [V(z) - \alpha I]^{-1} $$
for $\alpha \in \mathbb{C}$.  

Using the resolvent identity, we can compute 
\begin{align*}
	\frac{ \partial R_{ab}}{ \partial \Re(A_{cd})} &= - \frac{1}{\sqrt{n}} R_{ac} R_{d+n,b}, \\
	\frac{ \partial R_{ab}}{ \partial \Im(A_{cd})} &= - \frac{\sqrt{-1}}{\sqrt{n}} R_{ac} R_{d+n,b}, \\
	\frac{ \partial R_{ab}}{ \partial \Re(B_{cd})} &= -\frac{1}{\sqrt{n}} R_{a,d+n}R_{cb}, \\
	\frac{ \partial R_{ab}}{ \partial \Im(B_{cd})} &= \frac{\sqrt{-1}}{\sqrt{n}} R_{a,d+n} R_{cb},
\end{align*}
for $1 \leq c, d \leq n$ and $1 \leq a,b \leq 2n$.  For the remainder of the paper, we will take $A=B=X_n$.  

We will make use of the multivariate Gaussian decoupling formula \cite{PS}.  That is, if $Y=\{\xi_i\}_{i=1}^p$ is a real random Gaussian vector such that
$$ \E[\xi_j] = 0, \qquad \E[\xi_j \xi_k] = C_{jk} $$
for $j,k = 1, 2, \ldots, p$ and if $\Phi:\mathbb{R}^p \rightarrow \mathbb{C}$ has bounded partial derivatives, then
$$ \E[\xi_j \Phi] = \sum_{k=1}^p C_{jk} \E[ (\nabla \Phi)_k]. $$

Using the partial derivatives above and the Gaussian decoupling formula, we obtain 
\begin{align} \label{eq:decomp_real}
	\E[ R_{ab} x_{cd}] &= -\frac{1}{\sqrt{n}} \E[R_{a,d+n}R_{cb}] - \frac{\rho}{\sqrt{n}} \E[ R_{ad} R_{c+n,b}] \\
	& \qquad + \frac{1-2\mu}{\sqrt{n}} \E[R_{a,c}R_{d+n,b}] + \frac{(1-2\mu)\rho}{\sqrt{n}} \E[R_{a,c+n}R_{db}] \nonumber 
\end{align}
and
\begin{align} \label{eq:decomp_complex}
	\E[R_{ab} \bar{x}_{cd} ] &= -\frac{1}{\sqrt{n}} \E[R_{ac}R_{d+n,b}] - \frac{\rho}{\sqrt{n}} \E[ R_{a,c+n} R_{db}] \\
	& \qquad + \frac{(1-2\mu)\rho}{\sqrt{n}} \E[R_{a,d} R_{c+n,b}] + \frac{1-2\mu}{\sqrt{n}} \E[R_{a,d+n} R_{cb}] \nonumber
\end{align}
for $1 \leq c,d \leq n$, $c \neq d$, and $1 \leq a,b \leq 2n$.  

Following \cite{N}, we define the functions 
$$ s_n := s_n(\alpha,z) = \frac{1}{2n} \E[\tr R] = \frac{1}{n} \sum_{i=1}^n \E[R_{ii}] = \frac{1}{n} \sum_{i=1}^n \E[R_{i+n,i+n}] $$
and
$$ t_n := t_n(\alpha,z) = \frac{1}{n} \sum_{i=1}^n \E[R_{i+n,i}], \qquad u_n := u_n(\alpha,z) = \frac{1}{n} \sum_{i=1}^n \E[R_{i,i+n}]. $$

We now fix $z, \alpha \in \mathbb{C}$ with $\Im(\alpha) > 0$.  In the definitions above, we deal with the expectation of the summands instead of the random elements.  In order to justify this, we need control of the variance, which we obtain in the following lemma.

\begin{lemma} \label{lemma:complex-variance}
\begin{align}
	\Var\left( \frac{1}{2n} \tr R \right) = O_{\alpha,z}\left( \frac{1}{n} \right), \label{eq:var1}\\
	\Var\left( \frac{1}{n} \sum_{i=1}^n R_{i+n,i} \right) = O_{\alpha,z}\left( \frac{1}{n} \right), \label{eq:var2} \\
	\Var\left( \frac{1}{n} \sum_{i=1}^n R_{i,i+n} \right) = O_{\alpha,z}\left( \frac{1}{n} \right). \label{eq:var3}
\end{align}
\end{lemma}

\begin{proof}
We begin by noting that 
$$ \frac{1}{n} \sum_{i=1}^n R_{i+n,i} = \frac{1}{n} \tr ( P_2 R P_1 ) $$
and 
$$ \frac{1}{n} \sum_{i=1}^n R_{i,i+n} = \frac{1}{n} \tr (P_1 R P_1^\ast) $$
where $P_1$ and $P_2$ are partial isometries.  Thus, it suffices to prove 
$$ \Var \left( \frac{1}{n} \tr (P R Q ) \right) = O_{\alpha,z} \left( \frac{1}{n} \right) $$
for arbitrary partial isometries $P$ and $Q$.  

Let $\E_{\leq k}$ denote conditional expectation with respect to the $\sigma$-algebra generated by the random vectors 
$$ \mathbf{r}_1(X_n), \ldots, \mathbf{r}_k(X_n), \mathbf{c}_1(X_n), \ldots, \mathbf{c}_k(X_n). $$
Define
$$ Y_k := \E_{\leq k} \frac{1}{n} \tr (PRQ) $$
for $k=0,1,\ldots, 2n$.  Clearly $\{Y_k\}_{k=0}^{2n}$ is a martingale.  Define the martingale difference sequence
$$ \alpha_k := Y_k - Y_{k-1} $$
for $k=1,2,\ldots, 2n$.  Then by construction
$$ \sum_{k=1}^{2n} \alpha_k = \frac{1}{n} \tr (PRQ) - \E \frac{1}{n} \tr (PRQ). $$
Thus we need to show that
$$ \E \left| \sum_{k=1}^{2n} \alpha_k \right|^2 = O_{\alpha,z} \left( \frac{1}{n} \right). $$

Again we introduce the notation $X_{n,k}$ to denote the matrix $X_n$ with the $k$-th row and $k$-th column replaced by zeros.  Let 
$$ V_{k} := \begin{bmatrix} \frac{1}{\sqrt{n}} X_{n,k} & 0 \\ 0 & \frac{1}{\sqrt{n}} X_{n,k}^\ast \end{bmatrix} $$
and define
$$ R_{k} := \left[ V_{k}J_1-J_z - \alpha I \right]^{-1}. $$
Since 
$$ \E_{\leq k} \frac{1}{n} \tr (PR_{k}Q) = \E_{\leq k-1} \frac{1}{n} \tr (PR_{k}Q) $$
it follows that
$$ \alpha_k = \E_{\leq k} \left[ \frac{1}{n} \tr (PRQ) - \frac{1}{n} \tr (PR_{k}Q) \right] - \E_{\leq k-1} \left[ \frac{1}{n} (PRQ) - \frac{1}{n} \tr PR_{k}Q) \right]. $$

Because $V_{k} - V$ is at most rank $4$, we have that
\begin{align*}
	|\tr (PRQ) - \tr(PR_{k}Q)| \leq 4 \| R( (V_{k}J_1 - J_z) - (VJ_1 - J_z)) R_{k} \|.  
\end{align*}
We now claim that
\begin{equation} \label{eq:rvkj}
	\| R( (V_{k}J_1 - J_z) - (VJ_1 - J_z)) R_{k} \|  = O_{\alpha,z}(1) 
\end{equation}
uniformly in $k$.  Indeed, 
$$ \|R (VJ_1 - J_z)R_{k}\| \leq \frac{1}{|\Im(\alpha)|} \|R (VJ_1 - J_z)\| \leq \frac{1}{|\Im(\alpha)|} \sup_{t \in \mathbb{R}} \frac{|t|}{|t-\alpha|} = O_{\alpha,z}(1) $$
since $VJ_1-J_z$ is Hermitian.  Similarly, 
$$ \|R (V_{k}J_1 - J_z)R_{k}\| = O_{\alpha,z}(1) $$
and \eqref{eq:rvkj} follows.  Thus
$$ \alpha_k = O_{\alpha,z}\left(\frac{1}{n} \right) $$
uniformly in $k$.  

By the Burkholder inequality (see \cite[Lemma 2.12]{BS} for a complex martingale version of the Burkholder inequality), there exists an absolute constant $C>0$ such that
\begin{align*}
	\E\left| \sum_{k=1}^{2n} \alpha_k \right|^2 & \leq C \E \sum_{k=1}^{2n} |\alpha_k|^2 = O_{\alpha,z} \left( \frac{1}{n} \right).
\end{align*}  

\end{proof}

We are now ready to prove Lemma \ref{lemma:complex}.

\begin{proof}[Proof of Lemma \ref{lemma:complex}]
Fix $z,\alpha \in \mathbb{C}$ with $\Im(\alpha)>0$.  By the resolvent identity,
$$ 1 + \alpha s_n = \frac{1}{2n} \E \tr (RVJ_1) - \frac{\bar{z}}{2} u_n - \frac{z}{2} t_n. $$
We decompose,
$$ \frac{1}{2n} \E \tr (RVJ_1) = \frac{1}{2}A_1 + \frac{1}{2} A_2 $$
where 
$$ A_1 = \frac{1}{n} \E \sum_{i=1}^n (RVJ_1)_{ii} \qquad \text{and} \qquad A_2 = \frac{1}{n} \E \sum_{i=1}^n (RVJ_1)_{i+n,i+n}. $$
By \eqref{eq:decomp_complex} and Lemma \ref{lemma:complex-variance}, we have
\begin{align*}
	A_1 &= \frac{1}{n^{3/2}} \E \sum_{i,j=1}^n R_{i,j+n} \bar{x}_{ij} \\
		&= - \frac{1}{n^2} \E \sum_{i,j=1}^n [ R_{ii} R_{j+n,j+n} + \rho R_{i,i+n} R_{j,j+n} - (1-2\mu)\rho R_{i,j} R_{i+n,j+n} \\
	& \qquad\qquad\qquad - (1-2\mu) R_{i,j+n} R_{i,j+n} ] +o_{\alpha,z}(1) \\
	& = -s_n^2 - \rho u_n^2 + o_{\alpha,z}(1).
\end{align*}
For the second line, we used that the diagonal entries $i = j$ give total contribution $O_{\alpha,z}(n^{-1/2})$ which we write as the $o_{\alpha,z}(1)$ term.  For the third line, we used that if 
$$ R = \begin{bmatrix} R_1 & R_2 \\ R_3 & R_4 \end{bmatrix} $$
where $R_1, R_2, R_3, R_4$ are $n \times n$ matrices, then 
$$ \left | \frac{1}{n^2} \sum_{i,j=1}^n R_{i,j} R_{i+n,j+n}\right| = \left| \frac{1}{n^2} \tr (R_1^\mathrm{T} R_4) \right| \leq \frac{1}{n} \|R\|^2 \leq \frac{1}{n|\Im(\alpha)|^2} $$
and
$$ \left| \frac{1}{n^2} \sum_{i,j=1}^n R_{i,j+n} R_{i,j+n} \right| = \left| \frac{1}{n^2} \tr (R_2^\mathrm{T} R_2) \right| \leq \frac{1}{n |\Im(\alpha)|^2}. $$

Similarly (using \eqref{eq:decomp_real} and Lemma \ref{lemma:complex-variance}),
$$ A_2 = -s_n^2 - \rho t_n^2 + o_{\alpha,z}(1). $$
Thus
\begin{equation} \label{eq:1as}
	1+\alpha s_n = -s_n^2 - \frac{\rho}{2} u_n^2 - \frac{\rho}{2} t_n^2 - \frac{\bar{z}}{2} u_n - \frac{z}{2} t_n + o_{\alpha,z}(1). 
\end{equation}

We now obtain an equation for $t_n$.  Again by the resolvent identity
\begin{align*}
	\alpha t_n &= \frac{1}{n^{3/2}} \E \sum_{i,j=1}^n R_{i+n,j+n} \bar{x}_{ij} - \bar{z} \frac{1}{n} \E \sum_{i=1}^n R_{i+n,i+n} = A_3 - \bar{z} s_n,
\end{align*}
where
$$ A_3 = \frac{1}{n^{3/2}} \E \sum_{i,j=1}^n R_{i+n,j+n} \bar{x}_{ij}. $$
We repeat almost exactly the same procedure as above (using \eqref{eq:decomp_complex} and Lemma \ref{lemma:complex-variance}) and obtain
\begin{align*}
	A_3 &= -\frac{1}{n^2} \E \sum_{i,j=1}^n [  R_{i+n,i} R_{j+n,j+n} + \rho R_{i+n,i+n} R_{j,j+n} \\
	& \qquad \qquad \qquad - (1-2\mu)\rho R_{i+n,j} R_{i+n,j+n} - (1-2\mu) R_{i+n,j+n} R_{i,j+n}] + o_{\alpha,z}(1) \\
	&= -t_n s_n - \rho s_n u_n + o_{\alpha,z}(1).
\end{align*}
Again we used the fact the the diagonal entries give contribution $O_{\alpha,z}(n^{-1/2})$ and control the remaining error terms by writing each as a trace of products of $R_1, R_2, R_3, R_4$.  Thus we conclude
\begin{align*}
	\alpha t_n = -t_n s_n - \rho s_n u_n - \bar{z} s_n + o_{\alpha,z}(1).	
\end{align*}
Similarly, we obtain an equation for $u_n$:
$$ \alpha u_n = -u_n s_n - \rho s_n t_n - z s_n + o_{\alpha,z}(1). $$

Combining the above equations for $t_n$ and $u_n$ with \eqref{eq:1as}, we arrive at the following system of three equations:
\begin{align}
	1+\alpha s_n &= -s_n^2 - \frac{\rho}{2} u_n^2 - \frac{\rho}{2} t_n^2 - \frac{\bar{z}}{2} u_n - \frac{z}{2} t_n + o_{\alpha,z}(1) \nonumber \\
	\alpha t_n &=-t_n s_n - \rho s_n u_n - \bar{z} s_n + o_{\alpha,z}(1) \label{eq:3com} \\
	\alpha u_n &= -u_n s_n - \rho s_n t_n - z s_n + o_{\alpha,z}(1). \nonumber
\end{align}
We note that the system of equations above does not depend on $\mu$.  In addition to the case above, we also consider the case when $\mu=1$.  This corresponds to the real Gaussian case studied in \cite{N}.  Repeating the same calculations as above, we obtain the following system of equations in the real Gaussian case:
\begin{align}
	1+\alpha \hat{s}_n &= -\hat{s}_n^2 - \frac{\rho}{2} \hat{u}_n^2 - \frac{\rho}{2} \hat{t}_n^2 - \frac{\bar{z}}{2} \hat{u}_n - \frac{z}{2} \hat{t}_n + o_{\alpha,z}(1) \nonumber \\
	\alpha \hat{t}_n &=-\hat{t}_n \hat{s}_n - \rho \hat{s}_n \hat{u}_n - \bar{z} \hat{s}_n + o_{\alpha,z}(1) \label{eq:3real}\\
	\alpha \hat{u}_n &= -\hat{u}_n \hat{s}_n - \rho \hat{s}_n \hat{t}_n - z \hat{s}_n + o_{\alpha,z}(1) \nonumber.
\end{align}
One can also check that this system matches (5.4), (5.5), and (5.6) from \cite{N}.  In \cite{N}, it is shown that for every $\alpha,z \in \mathbb{C}$ with $\Im(\alpha)>0$,
$$ \lim_{n \rightarrow \infty} \hat{s}_n = s_{0} $$
where $s_0 = s_0(\alpha,z)$ is given by
$$ s_0 = \frac{1}{2} \int_{\mathbb{R}} \frac{d\nu_z(x)}{x-\alpha} - \frac{1}{2}\int_{\mathbb{R}} \frac{d\nu_z(x)}{x+\alpha} $$
and the family $\{\nu_z\}_{z \in \mathbb{C}}$ determines the elliptic law with parameter $\rho$ by Lemma \ref{lemma:hermitization}.  

By Lemma \ref{lemma:complex-variance} and \cite[Lemma B.9]{BS}, it suffices to show that for every $\alpha,z \in \mathbb{C}$ with $\Im(\alpha)>0$, 
\begin{equation} \label{eq:limns0}
	\lim_{n \rightarrow \infty} s_{n} = s_0 
\end{equation}
in order to complete the proof of Lemma \ref{lemma:complex}.  

Since $\|R\| \leq \frac{1}{\Im(\alpha)}$, it follows that $|s_n|,|t_n|, |u_n| \leq \frac{1}{\Im(\alpha)}$.  Similarly,
$|\hat{s}_n|, |\hat{t}_n|, |\hat{u}_n| \leq \frac{1}{\Im(\alpha)}$.  So by Vitali's convergence theorem, it suffices to show that for any fixed $z \in \mathbb{C}$, \eqref{eq:limns0} holds for any $\alpha \in \mathbb{C}$ with $\Im(\alpha)$ sufficiently large. 

Taking $\Im(\alpha)> 1$, we subtract the last two equations of \eqref{eq:3com} and \eqref{eq:3real} to obtain
\begin{align*}
	|t_n - \hat{t}_n| &\leq \frac{1 }{\Im(\alpha)^2 - 1} |u_n - \hat{u}_n| + \frac{1+\rho + \Im(\alpha)|z|}{\Im(\alpha)^2-1} |s_n - \hat{s}_n| +o_{\alpha,z}(1) \\
	|u_n - \hat{u}_n| &\leq \frac{1 }{\Im(\alpha)^2 - 1} |t_n - \hat{t}_n| + \frac{1+\rho + \Im(\alpha)|z|}{\Im(\alpha)^2-1} |s_n - \hat{s}_n| + o_{\alpha,z}(1).
\end{align*}
Taking $\Im(\alpha)$ sufficiently large (in terms of $\rho$ and $|z|$) we can write (say)
\begin{align*}
	|t_n - \hat{t}_n| &\leq \frac{1 }{100} |u_n - \hat{u}_n| + \frac{1}{100} |s_n - \hat{s}_n| +o_{\alpha,z}(1) \\
	|u_n - \hat{u}_n| &\leq \frac{1 }{100} |t_n - \hat{t}_n| + \frac{1}{100} |s_n - \hat{s}_n| + o_{\alpha,z}(1)
\end{align*}
and hence
\begin{align*}
	|t_n - \hat{t}_n| &\leq \frac{2}{99} |s_n - \hat{s}_n| + o_{\alpha,z}(1) \\
	|u_n - \hat{u}_n| &\leq \frac{2}{99} |s_n - \hat{s}_n| + o_{\alpha,z}(1).
\end{align*}

We now subtract the first equations of \eqref{eq:3com} and \eqref{eq:3real} and apply the bounds above to obtain
$$ |s_n - \hat{s}_n| \leq \frac{8}{99} |s_n - \hat{s}_n| + o_{\alpha,z}(1) $$
for $\Im(\alpha) > \max\{99,|z|\}$.  

This implies that for any $\alpha,z \in \mathbb{C}$ fixed with $\Im(\alpha)$ sufficiently large,
$$ s_n = \hat{s}_n + o(1) $$
and the proof of Lemma \ref{lemma:complex} is complete.
\end{proof}

\subsection*{Acknowledgements}
The authors would like to thank B.~Khoruzhenko for pointing out reference \cite{GK}, and for providing many valuable comments and suggestions.  The authors are grateful to the anonymous referees for many valuable suggestions and corrections.

\appendix

\section{Proof of Theorem \ref{theorem:ILOlinear:new}}\label{section:ILOlinear:proof}

Our first step is to obtain the following. 

\begin{claim}[upper bound for small ball probability]\label{lemma:smallball} We have 
\begin{align*}
&\sup_a\P\left(\left|\sum_i (a_ix_i+ b_ix_i')-a\right|\le r\right) \\
&\le \exp(\pi r^2)\int_{\C}\exp\left(-\sum_{i=1}^n \E_{\xi_1,\xi_2,\xi_1',\xi_2'}\|\Re (2(\xi_1-\xi_1')a_i t + 2(\xi_2-\xi_2') b_it)  \|_{\R/\Z}^2 - \pi |t|^2\right) dt,
\end{align*}
where $\|z\|_{\R/\Z}$ is the distance from a real number $z$ to its nearest integer. 
\end{claim}

\begin{proof}(of Claim \ref{lemma:smallball}) First of all, we have
\begin{align*}
\P\left(\sum_{i=1}^n (a_ix_i+b_ix_i') \in B(a,r)\right) &= \P\left(\left|\sum_{i=1}^n a_ix_i+b_ix_i' -a\right|^2 \le r^2\right)\\
&= \P\left(\exp(-\pi |\sum_{i=1}^n a_ix_i+b_ix_i'-a|^2) \ge \exp(-\pi r^2)\right)\\
&\le \exp(\pi r^2)\E \exp\left(-\pi\left|\sum_{i=1}^n a_ix_i+b_ix_i'-a\right|^2\right).
\end{align*}
Note that for any $z\in \R^2$, $\exp(-\pi|z|^2) = \int_{\C}e(zt) \exp(-\pi |t|^2) dt$, where $e(u):= \exp(2\pi \sqrt{-1} \Re(u))$. Thus,
$$\P\left(\sum_{i=1}^n a_ix_i+b_ix_i'\in B(a,r)\right)\le \exp(\pi r^2) \int_{\C}\E e\left(\left(\sum_{i=1}^n a_ix_i+b_ix_i'\right)t\right)e(-at)\exp(-\pi |t|^2)\\ dt.$$
Next, because of independence we have $|\E e((\sum_{i=1}^n a_ix_i+b_ix_i')t)| = \prod_{i=1}^n |\E e( x_i a_it +  x_i' b_it )|$, and so
\begin{align*}
|\E e( x_i a_it + x_i' b_i t)| &\le |\E e( x_ia_it + x_i' b_i t)|^2/2+1/2\\
&= \E_{\xi_1,\xi_2,\xi_1',\xi_2'} e\Big( (\xi_1-\xi_1') a_it + (\xi_2-\xi_2') b_i t\Big)/2+1/2\\
&= \E_{\xi_1,\xi_2, \xi_1',\xi_2'} \cos \Big(2\pi \Re\big( (\xi_1-\xi_1') a_it + (\xi_2-\xi_2') b_i t\big)\Big)/2+1/2\\
&\le \exp\Big(-\E_{\xi_1,\xi_2,\xi_1',\xi_2'}\|\Re\big(2(\xi_1-\xi_2) a_it+ 2(\xi_1'-\xi_2') b_i t\big)\|_{\R/\Z}^2\Big),
\end{align*}
where the random vector $(\xi_1',\xi_2')$ is an identical independent copy of $(\xi_1,\xi_2)$, and in the last inequality we estimated crudely $|\cos \pi z|\le  1- \sin^2 (\pi z)/2  \le 1 -2\|z\|_{\R/\Z}^2  < \exp( -\|z\|_{\R/\Z}^2)$.
\end{proof}

Observe that, as $(\xi_1,\xi_2)$ belongs to a given $(\mu,\rho)$-family, so does the pair $(\omega_1,\omega_2):=((\xi_1-\xi_1')/2,(\xi_2-\xi_2')/2)$. Intuitively, for $\E|\psi_1|^2=\E |\psi_2|^2=1$ and $|\rho| =|\E [\psi_1\psi_2]|<1$, these two random variables are essentially not multiples of each other. We summarize this useful fact as a claim below.

\begin{claim}\label{claim:difference}
Assume that $(\omega_1,\omega_2)$ belongs to a given $(\mu,\rho)$-family. Then there exist positive numbers $\alpha, \delta,c_0,C_0$ and two Lebesgue-measurable sets $R_1$ and $R_2$ in the set $\{(x,y)\in \C^2, c_0<|x|,|y| <C_0\}$ such that $\P((\omega_1,\omega_2)\in R_1),\P((\omega_1,\omega_2)\in R_2) \ge \delta$ and $|a/b-c/d|>\alpha$ for any $(a,b)\in R_1$ and $(c,d)\in R_2$. 
\end{claim}

\begin{proof}(of Claim \ref{claim:difference}) Let $\ep_0$ be a sufficiently small positive constant to be chosen. There exist positive numbers $c_0,C_0$ depending on $\omega_1,\omega_2$ and on $\ep_0$ such that the truncated random variables $\psi_1:=\omega_1 \mathbf{1}_{c_0<|\omega_1|<C_0}, \psi_2:= \omega_2 \mathbf{1}_{c_0<|\omega_2|<C_0}$ satisfy the following 
\begin{enumerate}
\item $1-\ep_0 \le \E |\psi_1|^2, \E |\psi_2|^2\le 1+\ep_0$, \label{item:ep0}
\item $|\rho| -\ep_0 \le |\E [\psi_1 \psi_2]| \le |\rho| +\ep_0.$ \label{item:rhoep0}
\end{enumerate} 

Observe that it suffices to justify the claim for the truncated pair $(\psi_1,\psi_2)$. Set $k$ to be a sufficiently large integer. We divide the square $Q:= \{z\in \C, |\Im(z)|,|\Re(z)| \le C_0/c_0\}$ into $k^2$ closed smaller squares $Q_1,\dots,Q_{k^2}$ of size $2C_0/kc_0$ each, and then divide the region $R:=\{(x,y)\in \C^2, c_0 <|x|,|y| < C_0\}$ into $k^2$ closed regions $R_i, i=1,\dots, k^2$ depending on whether $x/y$ belongs to $Q_i$ or not. Note that if $(x,y)\in R$ then the complex number $x/y$ has absolute value bounded from above and below by $C_0/c_0$ and $c_0/C_0$ respectively, and so $x/y \in Q$. 

We next claim that for sufficiently small $\delta>0$ (depending on $c_0,C_0,k$), there are squares $Q_{i_0},Q_{j_0}$ that are not adjacent (i.e. sharing a common edge) and that $\P(\psi_1/\psi_2\in Q_{i_0})\ge \delta$ and $\P(\psi_1/\psi_2\in Q_{j_0}) \ge \delta$. Indeed, assuming otherwise, $\P(\psi_1/\psi_2\in Q_{i})<\delta$  holds for all but at most 9 adjacent squares. The larger square $Q'$ formed by these adjacent ones has size at most $6C_0/kc_0$ which satisfies
$$\P(\psi_1/\psi_2\in Q')\ge 1 - (k^2-9)\delta.$$  

We now concentrate on the event $\psi_1/\psi_2\in Q'$. Because of the definition, there exists a number $c$ such that if $x/y \in Q'$ then the difference $|x/y-c|$ can be bounded crudely by $6C_0/kc_0$. Without loss of generality, we assume that $|c|\ge 1$. (Otherwise we consider the ratio $\psi_2/\psi_1$ instead). Clearly, 
$$|\E [\psi_1 \psi_2]| \ge |\E[\psi_1 \psi_2 \mathbf{1}_{\psi_1/\psi_2\in Q'}]|-|\E[\psi_1 \psi_2 (1-\mathbf{1}_{\psi_1/\psi_2\in Q'}]|.$$
The expectation of the second term can be bounded crudely from above by $C_0^2(k^2-9)\delta$, while the expectation of the first term can be bounded from below by $(|c|- 6C_0/kc_0)\E |\psi_1|^2-C_0^2(k^2-9)\delta$, which is at least $(1-6C_0/kc_0)(1-\ep_0) - C_0^2(k^2-9)\delta$ because $|c|\ge 1$ and $\E |\psi_1|^2 \ge 1-\ep_0$ from item \ref{item:ep0} above. Finally, by choosing $k$ to be large enough (depending on $\ep_0,c_0,C_0$) and then $\delta$ to be small enough (depending on  $\ep_0,C_0$ and $k$), we obtain a lower bound $1-2\ep_0$ for $|\E [\psi_1 \psi_2]|$. This is impossible as from item \ref{item:rhoep0} we have $|\E [\psi_1 \psi_2]| \le  |\rho| +\ep_0 < 1-2\ep_0$.

In summary, we have obtained two closed sub-regions $R_{i_0},R_{j_0}$ of $R$ such that the corresponding squares $Q_{i_0}$ and $Q_{j_0}$ are not adjacent and that both $\P(\psi_1/\psi_2\in Q_{i_0})$ and $\P(\psi_1/\psi_2\in Q_{j_0})$ are greater than $\delta$. By definition, as $Q_{i_0}$ and $Q_{j_0}$ are not adjacent, we have $|a/b-c/d| \ge 2C_0/kc_0$ as long as $(a,b)\in R_{i_0}$ and $(c,d)\in R_{j_0}$, completing the proof.  
\end{proof}

We now apply Claim \ref{lemma:smallball} and \ref{claim:difference} to prove Theorem \ref{theorem:ILOlinear:new}. Our method here follows \cite{NgV} with non-trivial modifications.
  
\begin{proof}(of Theorem \ref{theorem:ILOlinear:new}) 
For short, set $a_i':=\beta^{-1}a_i, b_i':=\beta^{-1}b_i$. Also, we will denote by $z$ and $z'$ the random variables $2(\xi_1-\xi_2)$ and $2(\xi_1'-\xi_2')$ respectively, where $(\xi_1',\xi_2')$ is an identical independent copy of $(\xi_1,\xi_2)$. By definition, we have
$$\gamma=\sup_a\P\left(\left|\sum_i (a_ix_i+ b_ix_i)-a\right|\le \beta\right) =\sup_a\P\left(\left|\sum_i (a_i'x_i+ b_i'x_i')-a\right|\le 1\right) =n^{-O(1)}.$$
Set $M:= 2A \log n$ where $A$ is large enough. From Claim \ref{lemma:smallball} and the fact that $\gamma\ge n^{-O(1)}$ we easily obtain
\begin{align}\label{eqn:continuous:integral}
\frac{\gamma}{2} &\le \int_{|t|\le M} \exp\left(-\sum_{i=1}^n\E_{\xi_1,\xi_2,\xi_1',\xi_2'}\left\|\Re\Big(2(\xi_1-\xi_1') a_i't +2(\xi_2-\xi_2')b_i't)\Big) \right\|_{\R/\Z}^2-\pi |t|^2\right)dt \nonumber\\
&=\int_{|t|\le M} \exp\left(-\sum_{i=1}^n\E_{z,z'}\left\|\Re\Big(z a_i't  +z'b_i't \Big)\right\|_{\R/\Z}^2-\pi |t|^2\right) dt.
\end{align}

{\bf Large level sets}. For each integer $0\le m \le M$ we define the level set
$$S_m:= \left \{t \in \C: \sum_{i=1}^n \E_{z,z'} \left\|\Re\Big(z a_i't + z' b_i't \Big) \right\|_{\R/\Z}^2 + |t|^2 \le m  \right \}.$$
Then it follows from \eqref{eqn:continuous:integral} that $\sum_{m\le M} \mu(S_m) \exp(-\frac{m}{2}+1)\ge \gamma$, where $\mu(\cdot)$ denotes the Lebesgue measure of a measurable set. Hence there exists $m\le M$ such that $\mu(S_m) \ge \gamma\exp(\frac{m}{4}-2)$.

Next, since $S_m\subset B(0,\sqrt{m})$, by the pigeon-hole principle there exists an absolute constant $c$ and a ball $B(x_0,\frac{1}{2})\subset B(0,\sqrt{m})$ such that
$$\mu \left(B\left(x_0,\frac{1}{2}\right)\cap S_m \right) \ge c\mu(S_m)m^{-1} \ge c\gamma\exp\left(\frac{m}{4}-2\right)m^{-1}.$$

Consider $t_1,t_2\in B(x_0,1/2)\cap S_m$. By the Cauchy-Schwarz inequality (note that $ \E_{z,z'} \|\Re(z a_i' t  + z' b_i't)\|_{\R/\Z}^2$ satisfies the triangle inequality in $t$) we have
\begin{align*}
&\sum_{i=1}^n  \E_{z,z'} \left\|\Re\Big(z a_i'(t_1-t_2)  + z' b_i'(t_1-t_2) \Big)\right\|_{\R/\Z}^2\\
&\le 2\left(\sum_{i=1}^n  \E_{z,z'} \left\|\Re\Big(z a_i't_1  + z' b_i't_1 \Big)\right\|_{\R/\Z}^2+ \sum_{i=1}^n  \E_{z,z'} \left\|\Re\Big(z a_i't_2  + z' b_i't_2 \Big)\right\|_{\R/\Z}^2\right) \le 4m.
\end{align*}

Since $t_1-t_2 \in B(0,1)$ and $\mu(B(x_0,\frac{1}{2})\cap S_m - B(x_0,\frac{1}{2})\cap S_m) \ge \mu(B(x_0,\frac{1}{2})\cap S_m)$, if we put
$$T:=\left\{ t\in B(0,1) : \sum_{i=1}^n  \E_{z,z'} \left\|\Re\Big(za_i't  + z' b_i't \Big) \right\|_{\R/\Z}^2 \le 4m \right\},$$
then
$$\mu(T)\ge c\gamma \exp\left(\frac{m}{4}-2\right)m^{-1}.$$

{\bf Discretization}. Choose $N$ to be a sufficiently large prime (depending on the set $T$). Define the discrete box
$$B_0:=\left\{k_1/N + \sqrt{-1}k_2/N: k_1,k_2\in \Z, -N\le k_1,k_2 \le N\right\}.$$

We consider all the shifted boxes $z+B_0$, where $(\Re (z), \Im (z)) \in [0,1/N]^2$. By the pigeon-hole principle, there exists $z_0$ such that the size of the discrete set $(z_0+B_0) \cap T$ is at least the expectation, $|(z_0+B_0)\cap T| \ge N^2 \mu(T)$ (to see this, we first consider the case when $T$ is a box itself).

Let us fix some $t_0\in (z_0+B_0)\cap T$. Then for any $t \in (z_0+B_0)\cap T$ we have
\begin{align*}
&\sum_{i=1}^n \E_{z,z'} \left\|\Re\Big(za_i'(t-t_0)+ z' b_i'(t-t_0) \Big) \right\|_{\R/\Z}^2 \\
&\le 2 \sum_{i=1}^n  \E_{z,z'} \left\|\Re\Big(za_i't+ z' b_i't \Big) \right\|_{\R/\Z}^2 \\
&+2 \sum_{i=1}^n \E_{z,z'} \left\|\Re\Big(za_i't_0+ z' b_i't_0 \Big) \right\|_{\R/\Z}^2  \le 16m.
\end{align*}

Notice that $t_0-t \in B_1:=B_0-B_0= \{k_1/N + \sqrt{-1}k_2/N: k_1,k_2\in \Z, -2N \le k_1,k_2 \le 2N\}$. Thus there exists a subset $S$ of size at least $cN^2\gamma \exp(\frac{m}{4}-2)m^{-1}$ of $B_1$ such that the following holds for any $s\in S$
$$\sum_{i=1}^n  \E_{z,z'} \left\|\Re\Big(za_i's + z' b_i's \Big) \right\|_{\R/\Z}^2 \le 16m.$$

{\bf Double counting and separation}. By definition of $S$, we have
\begin{align*}
\E_{z,z'} \sum_{s\in S} \sum_{i=1}^n \left\|\Re\Big(za_i's + z' b_i's \Big) \right\|_{\R/\Z}^2  &\le 16m|S|.
\end{align*}

Notice that, for $z=2(\xi_1-\xi_1')$ and $z'=2(\xi_2-\xi_2')$, $(z/4,z'/4)$ belongs to the $(\mu,\rho)$-family. By Claim \ref{claim:difference}, there exist $(c_1,c_2)\in \mathcal{R}_1$ and $(c_1',c_2')\in \mathcal{R}_2$ such that 
\begin{align*}
\sum_{s\in S} \sum_{i=1}^n \left\|\Re\Big((4c_1a_i'+4c_2 b_i')s \Big) \right\|_{\R/\Z}^2  &\le 16\delta^{-1}m|S|
\end{align*}
and 
\begin{align*}
\sum_{s\in S} \sum_{i=1}^n \left\|\Re\Big((4c_1'a_i'+4c_2' b_i')s \Big) \right\|_{\R/\Z}^2  &\le 16\delta^{-1}m|S|.
\end{align*}

From now on, for brevity, we denote by $v_i$ the complex number $4c_1 a_i'+4c_2b_i'$ for $1\le i\le n$, and by $v_{n+i}$ the complex number $4c_1'a_i'+ 4c_2'b_i'$ for $1\le i\le n$. We then have
$$\sum_{s\in S} \sum_{i=1}^{2n} \|\Re(v_is)\|_{\R/\Z}^2  \le 32\delta^{-1}m|S|.$$

{\bf Switching to $\R^2$.} Next, for convenience, we view each $v_i$ as the vector $(\Re (v_i),\Im (v_i))$ and each $s\in S$ as the vector $(\Re (s), -\Im (s))$ of $\R^2$. So we can write $\Re(v_is)$ as $\langle v_i,s\rangle$, and thus obtain the new estimate in $\R^2$,
$$\sum_{s\in S} \sum_{i=1}^{2n} \|\langle v_i,s \rangle \|_{\R/\Z}^2  \le 32 \delta^{-1}m|S|.$$

Let $n'$ be any number between $n^{\ep}$ and $2n$. We say that an index $1\le i\le 2n$ is {\it bad}  if
$$\sum_{s\in S} \|\langle v_i,s \rangle \|^2_{\R/\Z} \ge \frac{32\delta^{-1}m|S|}{n'}.$$
Then the number of bad indices is at most $n'$. Let $V$ be the set of remaining $v_i$'s. Thus $V$ contains at least $2n-n'$ elements (counting multiplicities).
In the rest of the proof, we are going to show that the set $V$ is close to a GAP.

{\bf Dual sets.} Consider an arbitrary good index $i$, we have 
$$\sum_{s\in S} \|\langle s, v_i \rangle \|^2_{\R/\Z} \le 32\delta^{-1}m|S|/n'.$$

Set $k:=\left\lfloor \sqrt{\frac{n'}{2048\pi^2 \delta^{-1}m}}\right\rfloor$ and let $V_k:=k(V\cup \{0\})$, the Minkowski sum of $k$ copies of $V\cup \{0\}$. By the Cauchy-Schwarz inequality, for any $v\in V_k$ we have
$$\sum_{s\in S} 2\pi^2 \|\langle s,v \rangle \|^2_{\R/\Z} \le \frac{|S|}{2},$$
which implies
$$\sum_{s\in S}\cos(2\pi \langle s,v\rangle) \ge \frac{|S|}{2}.$$

Observe that for any $x\in C(0,\frac{1}{512})$ (the ball of radius $1/512$ in the $\|\cdot\|_{\infty}$ norm) and any $s\in S\subset C(0,2)$ we always have $\cos(2\pi \langle s,x\rangle)\ge 1/2$ and $\sin(2\pi \langle s,x\rangle) \le 1/12$. Thus for any $x\in C(0,\frac{1}{512})$,
$$\sum_{s\in S}\cos\left(2\pi \langle s,(v+x)\rangle\right) \ge \frac{|S|}{4}-\frac{|S|}{12} = \frac{|S|}{6}.$$

On the other hand one can easily check that
\begin{align*}
\int_{x\in [0,N]^2} \left(\sum_{s\in S} \cos(2\pi \langle s, x\rangle)\right)^2 dx &\le \sum_{s_1,s_2\in S}\int_{x\in [0,N]^2} \exp \left( 2\pi \sqrt{-1}\langle s_1-s_2, x\rangle \right) dx\\
&\ll |S|N^2.
\end{align*}
Hence we deduce the following
$$\mu_{x\in [0,N]^2}\left (\Big(\sum_{s\in S} \cos(2\pi \langle s, x\rangle )\Big)^2 \ge (\frac{|S|}{6})^2 \right) \ll \frac{|S|N^2}{(|S|/6)^2} \ll \frac{N^2}{|S|}.$$

Now using the fact that $S$ has large size, $|S|\gg N^2 \gamma\exp(\frac{m}{4}-2)m^{-1}$, and $N$ was chosen to be large enough so that $V_k+C(0,\frac{1}{512}) \subset [0,N]^2$, we have
$$\mu \left(V_k+C\left(0,\frac{1}{512}\right)\right) \ll \gamma^{-1} \exp\left(-\frac{m}{4}+2\right)m.$$

Thus, we have obtained the following 
\begin{align}\label{eqn:sizeofkV'}
\mu \left( k(V\cup \{0\})+C\left(0,\frac{1}{512}\right) \right) &\ll \gamma^{-1} \exp\left(-\frac{m}{4}+2\right) m.
\end{align}



Let $D:=2048\times 16 \times \delta^{-1}=\Theta(\delta^{-1})$. We approximate each vector $v$ of $V$ by a closest vector in $(\frac{\Z}{Dk})^2$,
$$\left\|v-\frac{u}{Dk}\right\|_2 \le \frac{\sqrt{2}}{Dk}, \mbox{ with } u\in\Z^2.$$
Let $U$ be the union of the collection of all such $u$ with $\{0\}$. Since $\sum_{v\in V}\|v\|^2_2 = O(\beta^{-2})$, we have
\begin{equation}\label{eqn:magnitudeV''}
\sum_{u\in U}\|u\|^2_2=O_{\delta^{-1}}(k^2\beta^{-2}).
\end{equation}

It follows from \eqref{eqn:sizeofkV'} that
\begin{align}\label{eqn:revised'}
|k(U+ C_0(0,1))|&=O\left(\gamma^{-1}(Dk)^2 \exp(-\frac{m}{4}+2)m\right)\nonumber \\
&= O\left(\gamma^{-1}k^2 \exp(-\frac{m}{4}+2)m\right),
\end{align}
where $C_0(0,r)$ is the discrete cube $C_0(0,r)=\{(x,y)\in \Z^2, |x|,|y|\le r\}$.

We next pause to recall some useful results from \cite{NgV} and \cite{TV-John}. For any integer $t\ge 1$, we say that a symmetric GAP $Q=\{k_1g_1+\dots+k_rg_r,  k_i\in \Z, |k_i|\le N_i \}$ is {\it $t$-proper} if the GAP $tQ=\{k_1g_1+\dots+k_rg_r,  k_i\in \Z, |k_i|\le tN_i \}$ is proper.

\begin{lemma}\cite[Theorem 1.21, also Theorem 1.17]{TV-John}\label{lemma:TV}
Let $A > 0$ be a constant. Assume that $X$ is a subset of integers such that $|lX| \le l^A |X|$ for some number $l\ge 2$. Then $lX$ is contained in a symmetric 2-proper GAP $Q$ of rank $r = O_A(1)$, and of cardinality $O_A(|lX|)$.
\end{lemma}

\begin{lemma}\cite[Lemma A.2]{NgV}\label{lemma:dividing} 
Assume that $0 \in X$ and that $P=\{\sum_{i=1}^r
x_ia_i: |x_i|\le N_i\}$ is a symmetric 2-proper GAP that contains $kX$. Then $X\subset \{\sum_{i=1}^r x_ia_i: |x_i|\le
2N_i/k\}$.
\end{lemma}

Thus by \eqref{eqn:revised'} and Lemma \ref{lemma:TV}, there exists a 2-proper symmetric GAP $R=\{\sum_{i=1}^r x_i g_i: |x_i|\le M_i \}$ that contains $k(U+C_0(0,1))$ and
\begin{equation}\label{eqn:extra:0}
r=O(1) \mbox{ and } |R|= O(\gamma^{-1} k^2 \exp(-\frac{m}{4}+2)m).
\end{equation}

Furthermore, by Lemma \ref{lemma:dividing}
$$U+C_0(0,1)\subset P:=\left\{\sum_{i=1}^r x_i g_i: |x_i|\le 2M_i/k \right\}.$$

We remark that as $k(U+C_0(0,1))$ is a dense copy of $R$, there exist $m_1,m_2=O(1)$ such that the dilated GAP $m_1\cdot R$ can be contained in the set $m_2k(U+C_0(0,1))$ (see for instance \cite[Lemma B.3]{Tao}). Using \eqref{eqn:magnitudeV''}, we conclude that all the generators $g_i$ of $P$ are bounded,
$$\|g_i\|_2 =O(k \beta^{-1}).$$
Next, since $C_0(0,1)\subset P$, the rank $r$ of $P$ (and $R$) is at least $2$. We consider the following two cases.

{\bf Case 1}: $r\ge 3$. Recall that $|P|=O(\gamma^{-1}{k}^{(2-r)} \exp(-\frac{m}{4}+2)m)= O(\gamma^{-1}/\sqrt{n'})$. Let
$$Q:=\frac{\beta}{Dk}\cdot P.$$

It is clear that $Q$ satisfies all of the conditions of Theorem  \ref{theorem:ILOlinear:new}. (Note that, in this case, we obtain a stronger approximation; almost all elements of $V$ are $O(\frac{\beta \sqrt{\log n'}}{\sqrt{n'}})$-close to $Q$.)

{\bf Case 2}: $r=2$. Because the unit vectors $e_1=(1,0), e_2=(0,1)$ belong to $P=\{\sum_{i=1}^2 x_ig_i:|x_i|\le  2M_i/k \}\subset \Z^2$, the set of generators $g_1,g_2$ forms a basis with unit determinant in $\R^2$. In this case we will be making use of $R$. Note that by definition $C_0(0,k)\subset R$.

Let $r_0$ be the smallest positive number such that 
\begin{equation}\label{eqn:extra:1}
|6R \cap C_0(0,r_0 k)| \ge 50|R|/k.
\end{equation}

By definition, 
\begin{equation}\label{eqn:extra:2}
|6R \cap C_0(0,r_0 k/2)| < 50|R|/k.
\end{equation}
Now let $p\in P$ be an arbitrary element of $P$, then one has 
$$\|p\|_\infty \le 2r_0 k.$$
Indeed, assume otherwise, then the sets $jp + (6R\cap C_0(0,r_0k)), -k\le j \le k$ are disjointly sitting inside $2R+6R=8R$; thus 
$$2k |6R\cap C_0(0,r_0k)| \le |8R| <100 |R|,$$
where we used the fact that $R$ has rank 2 in the last estimate, a contradiction against \eqref{eqn:extra:1}.  

For later use, we record a useful fact as follows.

\begin{fact}\label{fact:extra:1} For any $z_0\in \R^2$ we have
$$|3R\cap (z_0+C_0(0,r_0k/4))|\le 50|R|/k.$$
\end{fact}
\begin{proof}(of Fact \ref{fact:extra:1})
By the elementary bound $|X|\le |X-X|$ and by \eqref{eqn:extra:2},
\begin{align*}
|3R\cap (z_0+C_0(0,r_0k/4))|&\le |[3R\cap (z_0+C_0(0,r_0k/4))] -[3R\cap (z_0+C_0(0,r_0k/4)]| \\
&\le |6R \cap C_0(0,r_0k/2)| \\
&\le 50|R|/k.
\end{align*}
\end{proof}

We next consider two subcases.

 {\bf Subcase 1}. Suppose $r_0<10$. Then as $P\subset C_0(0,2r_0k)\subset C_0(0,20k)$, it is easy to find a GAP $S$ of size $O(1)$ which is $k$-close to $P$ (and hence $k$-close to $U$ because $U\subset P$).
 
 {\bf Subcase 2.} Suppose $r_0\ge 10$. Note that $P \subset R \cap C_0(0,2r_0k)$. With room to spare, let $Z$ be the intersection of $2R \cap C_0(0,(2r_0+1)k)$ with the lattice $\Gamma:=\{(ki,kj), i,j\in \Z\}$. Note that $Z$ in non-empty. We next state some nice properties about $Z$.

\begin{claim}\label{claim:extra:1} We have
\begin{enumerate}
\item $P$ (and hence $U$) are $O(k)$-close to $Z$. \label{item:claim1}
\item There is a GAP $S$ of small rank and size $O(|Z|)$ that contains $Z$. \label{item:claim2}
\item The size of $Z$ is $O(\gamma^{-1}/\sqrt{n'})$. \label{item:claim3}
\end{enumerate}
\end{claim} 
\begin{proof}(of Claim \ref{claim:extra:1})
For part \ref{item:claim1}, note that
 $$P+C_0(0,k) \subset  [R \cap C_0(0,2r_0k)] + C_0(0,k) \subset 2R \cap C_0(0,(2r_0+1)k).$$
Thus $(P+C_0(0,k)) \cap \Gamma \subset Z$. In other words, for every element $p\in P$, there exists $z\in Z$ such that $\|z-p\|_2 \le \sqrt{2}k$.
 
For \ref{item:claim2}, because the generators $g_1,g_2$ of $Q$ form a basis with unit determinant in $\R^2$, we can view $Z$ as the intersection between $\Gamma$ and the symmetric convex body in $\R^2$ that corresponds to $\mathrm{conv}(2R\cap C_0(0,(2r_0+1)k))$. Thus $Z$ can be contained in a GAP $S$ of rank 2 and size $|S|=O(|Z|)$ by \cite[Lemma 3.36]{TVbook}.
 
 For \ref{item:claim3}, note that the sets $z+(R\cap C_0(0,k/4)), z\in Z$ are disjointly lying inside $3R\cap C_0(0,3r_0k)$, and so 
$$|Z|\le  \frac{|3R \cap C_0(0,3r_0k)|}{|R\cap C_0(0,k/4)|} = \frac{|3R \cap C_0(0,3r_0k)|}{|C_0(0,k/4)|} \le \frac{O(|R|/k )}{k^2} =O(\gamma^{-1}\exp(-\frac{m}{4}+2)m/k) = O(\gamma^{-1}/\sqrt{n'}),$$ 
where in the third estimate we decomposed $C_0(0,3r_0k)$ into disjoint copies of $C_0(0,r_0k/4)$ and applied Fact \ref{fact:extra:1}, and we used the bound in \eqref{eqn:extra:0} for $|R|$ in the second to last estimate.
\end{proof}
 
Therefore, in both subcases $P$ is $O(k)$-close a GAP $S$ of small rank and size $O(1+\gamma^{-1}/\sqrt{n'})$. To obtain $Q$ as concluded in Theorem \ref{theorem:ILOlinear:new} we just set
$$Q:=\frac{\beta}{Dk}\cdot S.$$   
\end{proof}

\section{Proof of Lemma \ref{lemma:decoupling}}\label{section:decoupling:proof} 

Set $a_{ij}':=a_{ij}/\beta$. By definition, 
$$\gamma=\sup_{a,b_i,b_i'} \P_{\Bx,\Bx'} \Big(|\sum_{i,j}a_{ij}'x_ix_j'+\sum_i b_ix_i+\sum_i b_i'x_i'-a|\le 1\Big) \ge n^{-B}.$$

By Markov's inequality we have
\begin{align*}
&\P_{\Bx,\Bx'}\big(|\sum_{i,j}a_{ij}'x_ix_j'+\sum_i b_ix_i+\sum_i b_i'x_i'-a|\le 1\big) \\
&= \P\big(\exp(-\frac{\pi}{2}|\sum_{i,j}a_{ij}'x_ix_j'+\sum_i b_ix_i+ \sum_i b_i'x_i'-a|^2 \ge \exp(-\frac{\pi}{2} )\big) \\
& \le \exp(\frac{\pi}{2} ) \E_{\Bx,\Bx'} \exp\big(-\frac{\pi}{2}|\sum_{i,j}a_{ij}'x_ix_j'+\sum_i b_ix_i+\sum_ib_i'x_i'-a|^2\big)  \\
&\le \exp(\frac{\pi}{2}) \int_{\C} \big|\E_{\Bx,\Bx'} e[(\sum_{i,j}a_{ij}'x_ix_j'+\sum_i b_ix_i+\sum_ib_i'x_i')\cdot t]\big| \exp(-\frac{\pi}{2} |t|^2)dt\\
&\le \exp(\frac{\pi}{2})(\sqrt{2\pi})^2 \int_{\C} \big|\E_{\Bx,\Bx'} e[(\sum_{i,j}a_{ij}'x_ix_j'+\sum_i b_ix_i+\sum_ib_i'x_i'))\cdot t]\big| \exp(-\frac{\pi}{2} |t|^2)/(\sqrt{2\pi})^2 dt
\end{align*}
where in the fourth equation we used the identity $\exp(-\frac{\pi}{2} |x|^2) =\int_{\C} e(xt) \exp(-\frac{\pi}{2} |t|^2) dt$.

Consider $\Bx=(x_1,\dots,x_n)$ as $(\Bx_U,\Bx_{\bar{U}})$ and $\Bx'=(x_1',\dots,x_n')$ as $(\Bx_U',\Bx_{\bar{U}}')$, where $\Bx_U,\Bx_U'$ and $\Bx_{\bar{U}},\Bx_{\bar{U}}'$ are the vectors corresponding to $i\in U$ and  $i\notin U$ respectively. After a series applications of the identity $\int_\C \exp(-\frac{\pi}{2} |t|^2)/(\sqrt{2\pi})^2 dt=1$ and the Cauchy-Schwarz inequality, we obtain
\begin{align}\label{eqn:bilinear:series}  
&\quad\Bigg[\int_{\C} \Big|\E_{\Bx,\Bx'} e\big((\sum_{i,j}a_{ij}'x_ix_j'+\sum_ib_ix_i+ \sum_i b_i'x_i')\cdot t\big) \Big| \exp(-\frac{\pi}{2} |t|^2)/(\sqrt{2\pi})^2 dt\Bigg]^4 \nonumber \\
&\le \Bigg[\int_{\C}  \Big|\E_{\Bx,\Bx'} e\big((\sum_{i,j}a_{ij}'x_ix_j'+\sum_ib_ix_i+\sum_ib_i'x_i'))\cdot t\big)\Big|^2 \exp(-\frac{\pi}{2} |t|^2)/(\sqrt{2\pi})^2 dt \Bigg]^2  \nonumber\\
&\le \Bigg[\int_{\C} \E_{\Bx_U,\Bx_U'}\Big|\E_{\Bx_{\bar{U}},\Bx_{\bar{U}}'}e\big((\sum_{i,j} a_{ij}'x_ix_j'+\sum_ib_ix_i+\sum_ib_i'x_i'))\cdot t\big)\Big|^2 \exp(-\frac{\pi}{2} |t|^2)/(\sqrt{2\pi})^2 dt\Bigg]^2  \nonumber\\ 
&=\Bigg[ \int_{\C} \E_{\Bx_U,\Bx_U'}\E_{\Bx_{\bar{U}},\Bx_{\bar{U}}',\By_{\bar{U}},\By_{\bar{U}}'} e\Big(\big(\sum_{i\in U,j\in \bar{U}}a_{ij}'x_i(x_j'-y_j')+\sum_{i\in \bar{U},j\in U}a_{ij}'(x_i-y_i)x_j'+\sum_{j\in \bar{U}} b_j(x_j-y_j)  \nonumber\\
&+\sum_{j\in \bar{U}} b_j'(x_j'-y_j')+\sum_{i\in \bar{U},j\in \bar{U}}a_{ij}'(x_ix_j'-y_iy_j')\big)\cdot t\Big) \exp(-\frac{\pi}{2} |t|^2)/(\sqrt{2\pi})^2  dt\Bigg]^2  \nonumber\\
&\le \int_{\C} \E_{\Bx_{\bar{U}},\Bx_{\bar{U}}',\By_{\bar{U}},\By_{\bar{U}}'}\Big|\E_{\Bx_{U},\Bx_{U}'} e\Big(\big(\sum_{i\in U,j\in \bar{U}}a_{ij}'x_i(x_j'-y_j')+\sum_{i\in \bar{U},j\in U}a_{ij}'(x_i-y_i)x_j')+\sum_{j\in \bar{U}} b_j(x_j-y_j)  \nonumber\\ 
&+\sum_{j\in \bar{U}} b_j'(x_j'-y_j')+\sum_{i\in \bar{U},j\in \bar{U}}a_{ij}'(x_ix_j'-y_iy_j')\big)\cdot t\Big)\exp(-\frac{\pi}{2} |t|^2)/(\sqrt{2\pi})^2  dt\Big|^2  \nonumber\\
&=\int_{\C} \E_{\Bx_U,\By_U,\Bx_{\bar{U}},\By_{\bar{U}},\Bx_U',\By_U',\Bx_{\bar{U}}',\By_{\bar{U}}'} e \Big(\big(\sum_{i\in U, j\in \bar{U}}a_{ij}'(x_i-y_i)(x_j'-y_j')  \nonumber\\
&+ \sum_{i\in \bar{U}, j\in U}a_{ij}'(x_i-y_i)(x_j'-y_j') \big)\cdot t\Big) \exp(-\frac{\pi}{2} |t|^2)/(\sqrt{2\pi})^2 dt  \nonumber\\
&=\int_{\C}\E_{\Bv,\Bw}e\Big((\sum_{i\in U,j\in \bar{U}} a_{ij}'v_iw_j+ \sum_{i\in \bar{U},j\in U}  a_{ij}' v_iw_j)\cdot t\Big)\exp(-\frac{\pi}{2} |t|^2)/(\sqrt{2\pi})^2 dt \nonumber \\
&=(1/\sqrt{2\pi})^2 \E_{\Bv,\Bw}\exp(-\frac{\pi}{2}|\sum_{i\in U,j\in \bar{U}} a_{ij}'v_iw_j+ \sum_{i\in \bar{U},j\in U}  a_{ij}' v_iw_j|^2),
\end{align}
where $(\By_U,\By_U')$ and $(\By_{\bar{U}},\By_{\bar{U}}')$ are independent identical copies of $(\Bx_U,\Bx_U')$ and $(\Bx_{\bar{U}},\Bx_{\bar{U}}')$ respectively, and $\Bv:=\Bx-\By,\Bw:=\Bx'-\By'$.

Thus 
\begin{align*}
\gamma^4 &= \Big(\P_{\Bx,\Bx'}(|\sum_{i,j}a_{ij}'x_ix_j'+\sum_ib_ix_i+\sum_i b_i'x_i' -a|\le 1)\Big)^4 \\
&\le \exp(4\pi)(2\pi)^{4}  \Bigg(\int_{\C} \Big|\E_{\Bx,\Bx'} e\big[(\sum_{i,j}a_{ij}'x_ix_j'+\sum_i b_ix_i+\sum_ib_i'x_i'))\cdot t\big]\Big| \exp(-\frac{\pi}{2} |t|^2)/(\sqrt{2\pi})^2 dt\Bigg)^4\\
&\le \exp(4\pi)(2\pi)^3 \E_{\Bv,\Bw}\exp(-\frac{\pi}{2}|\sum_{i\in U,j\in \bar{U}} a_{ij}'v_iw_j+ \sum_{i\in \bar{U},j\in U}  a_{ij}' v_iw_j|^2).
\end{align*}

Because $\gamma\ge n^{-B}$, the inequality above implies that 
$$\P_{\Bv,\Bw}\Big(|\sum_{i\in U,j\in \bar{U}} a_{ij}'v_iw_j+ \sum_{i\in \bar{U},j\in U}  a_{ij}' v_iw_j|=O_{B}(\sqrt{\log n})\Big) \ge \frac{1}{2} \gamma^4/((2\pi)^{3}\exp(4\pi)).$$

Scaling back to $a_{ij}$, we thus obtain
$$\P_{\Bv,\Bw}\Big(|\sum_{i\in U,j\in \bar{U}} a_{ij}v_iw_j+ \sum_{i\in \bar{U},j\in U}  a_{ij} v_iw_j|=O_{B}(\beta \sqrt{\log n})\Big) \ge \frac{1}{2}\gamma^4/((2\pi)^{3}\exp(4\pi)),$$
completing the proof.

\end{document}